\documentclass[reqno]{amsart}

\usepackage[applemac]{inputenc}

\usepackage{amssymb}
\usepackage{graphicx}
\usepackage[cmtip,all]{xy}
\usepackage{verbatim}
\usepackage{tikz-cd}
\usepackage{mathtools}
\usepackage{graphicx}
\usepackage{mathrsfs}
\usepackage{geometry}
\usepackage{enumitem}
\usepackage{dsfont}
\usepackage{bbm}
\usepackage{url}
\usepackage[numbers,sort&compress]{natbib}

\DeclareMathAlphabet{\mathpzc}{OT1}{pzc}{m}{it}

\newtheorem{theorem}{Theorem}[section]
\newtheorem{lemma}[theorem]{Lemma}
\newtheorem{corollary}[theorem]{Corollary}

\theoremstyle{definition}
\newtheorem{definition}[theorem]{Definition}

\theoremstyle{remark}
\newtheorem{remark}[theorem]{Remark}

\numberwithin{equation}{section}

\makeatletter
\DeclareRobustCommand{\cev}[1]{%
  \mathpalette\do@cev{#1}%
}
\newcommand{\do@cev}[2]{%
  \fix@cev{#1}{+}%
  \reflectbox{$\m@th#1\vec{\reflectbox{$\fix@cev{#1}{-}\m@th#1#2\fix@cev{#1}{+}$}}$}%
  \fix@cev{#1}{-}%
}
\newcommand{\fix@cev}[2]{%
  \ifx#1\displaystyle
    \mkern#23mu
  \else
    \ifx#1\textstyle
      \mkern#23mu
    \else
      \ifx#1\scriptstyle
        \mkern#22mu
      \else
        \mkern#22mu
      \fi
    \fi
  \fi
}

\makeatother

 \newcommand{\virgolette}{``}

\newcommand{\set}[1]{\ensuremath{\mathbb{#1}}}
\newcommand{\cat}[1]{\ensuremath{\mbox{\sffamily{#1}}}}

\newcommand{\slantone}[2]{{\raisebox{.1em}{$#1$}\left/\raisebox{-.1em}{$#2$}\right.}}
\newcommand*{\defeq}{\mathrel{\vcenter{\baselineskip0.5ex \lineskiplimit0pt
                     \hbox{\scriptsize.}\hbox{\scriptsize.}}}%
                     =}

\newcommand{\slanttwo}[2]{{\raisebox{.2em}{$#1$} \big/ \raisebox{-.2em}{$#2$}}}

\newcommand\asim{\mathrel{%
  \ooalign{\raise0.1ex\hbox{$\sim$}\cr\hidewidth\raise-0.8ex\hbox{\scalebox{0.9}{$\scriptstyle{x}$}}\hidewidth\cr}}}

\newcommand{\mani}{\ensuremath{\mathpzc{M}}}
\newcommand{\manir}{\ensuremath{\mathpzc{M}_{\mathpzc{red}}}}
\newcommand{\stsheaf}{\ensuremath{\mathcal{O}_{\mathpzc{M}}}}
\newcommand{\stsheafred}{\ensuremath{\mathcal{O}_{\mathpzc{M}_{red}}}}

\newcommand{\beq}{\begin{equation}}
\newcommand{\eeq}{\end{equation}}
\newcommand{\bear}{\begin{eqnarray}}
\newcommand{\eear}{\end{eqnarray}}
\newcommand{\longhookrightarrow}{\ensuremath{\lhook\joinrel\relbar\joinrel\rightarrow}}

\begin{document}

\title{On the Geometry of Forms on Supermanifolds}


\author{Simone Noja}
\address{Universität Heidelberg}
\curraddr{Im Neuenheimer Feld 205}
\email{simone.noja@mathi.uni-heidelberg.de}




\begin{abstract}
This paper provides a rigorous account on the geometry of forms on supermanifolds, with a focus on its algebraic-geometric aspects. First, we introduce the de Rham complex of differential forms and we compute its cohomology. We then discuss three intrinsic definitions of the Berezinian sheaf of a supermanifold - as a quotient sheaf, via cohomology of the super Koszul complex or via cohomology of the total de Rham complex. Further, we study the properties of the Berezinian sheaf, showing in particular that it defines a right $\mathcal{D}$-module. 
Then we introduce integral forms and their complex and we compute their cohomology, by providing a suitable Poincaré lemma. We show that the complex of differential forms and integral forms are quasi-isomorphic and their cohomology computes the de Rham cohomology of the reduced space of the supermanifold. The notion of Berezin integral is then introduced and put to good use to prove the superanalog of Stokes' theorem and Poincaré duality, which relates differential and integral forms on supermanifolds. Finally, a different point of view is discussed by introducing the total tangent supermanifold and (integrable) pseudoforms in a new way. In this context, it is shown that a particular class of integrable pseudoforms having a distributional dependence supported at a point on the fibers are isomorphic to integral forms. Within the general overview, several new proofs of results are scattered.




\end{abstract}

\maketitle

\tableofcontents

\section{Introduction - A Tale of Forms and Berezinians}

\noindent Supergeometry is the study of \emph{supermanifolds}, \emph{i.e.}\ manifolds characterized by sheaves of $\mathbb{Z}_2$-graded (commutative) algebras, called \emph{superalgebras}, whose \virgolette functions'' might commute or anticommute, depending on their even or odd $\mathbb{Z}_2$-degree \cite{Konstant, Leites, Manin}.   
Certain supermanifolds can also be endowed with a tighter structure, roughly speaking a symmetry that exchanges even and odd directions \cite{WittenRiem}. This provides a geometric realization of what physicists call \emph{supersymmetry transformations} in the context of modern quantum field theories, thus making supergeometry into the right mathematical environment to study these physical theories - one such being superstring theory \cite{Deligne}.\\ 
In some sense, the characteristic $\mathbb{Z}_2$-graded commutativity of supergeometry is the lowest possible degree of non-commutativity. For this reason, several constructions from ordinary purely commutative algebra and geometry can be easily generalized to a supergeometric setting. But this is not really the point. Indeed, supergeometry features new notions that resist such a trivial extension, challenging instead our geometric intuition. \\ 
 One such notion is of \emph{global} nature and related to the holomorphic theory. Indeed, complex holomorphic supermanifolds can be \emph{non-projected} or \emph{non-split}, meaning that they cannot be directly reconstructed from their underlying purely commutative manifolds \cite{Green, Manin}. As such, they are genuinely new geometric objects, having \virgolette a life of their own''. Interest in these non-split / non-projected geometries has remarkably grown within the last years \cite{Bettadapura1, Bettadapura2, P2}, prompted by a paper of Donagi and Witten \cite{DonWit}, where it is shown that the \emph{supermoduli space} of super Riemann surfaces is indeed non-projected for genus at least $5$. \\ 
But issues are also of \emph{local} nature. Indeed, whereas on an ordinary manifold differential forms anticommute, thus forcing the de Rham complex to terminate at the dimension of the manifold, on a supermanifold differentials of odd functions do instead commute, so that the de Rham complex is not bounded \cite{Deligne, Manin}. This easy fact has very far-reaching consequences. 
Poincaré duality - as known from ordinary commutative geometry - breaks down, and there exists a new complex, which is \virgolette dual'' to the de Rham complex. This is the so-called complex of \emph{integral forms}, where the Berezinian bundle - a characteristic supergeometric construction that controls integration on supermanifolds - sits and plays the role the canonical bundle plays in the ordinary de Rham complex \cite{BL1}. Accordingly, integration theory on supermanifolds is quite peculiar and highly non-trivial \cite{Manin, Voronov}. Most notably, once again, its subtleties are related to important physical questions, regarding both the foundations of supersymmetric theories, described as manifestly invariant theories on superspaces \cite{CDF}, and actual computations of quantities of physical interests, such as scattering amplitudes in superstring theory \cite{WittenSuper}. \\
This paper is mostly concerned with the algebraic-geometric aspects of this rich and beautiful theory of forms (and integration) on supermanifolds. Before we start, though, in order to provide the reader with some perspective and context - and also to make justice to the researchers who have contributed most to building and shaping this branch of mathematics -, we will briefly go through and comment on the historical development of the theory.   \vspace{.4cm}



\noindent The concept of differential forms on supermanifolds originated from the astounding creativity of the work of Felix Berezin - the \virgolette mastermind of \emph{super-mathematics}'' as in \cite{BerezinLife} - and his collaborators around 1970 \cite{Berezin}, years before the mathematical formalization of the notion of \emph{supermanifold} was even available \cite{Konstant, Leites}. As in the ordinary theory, differential forms on supermanifolds can be structured into a complex, the de Rham complex, which - differently from a purely commutative setting - is \emph{not} bounded from above. 
Also, differential forms allow for integration on reduced, or \emph{bosonic}, submanifolds of a supermanifold by their pull-back, but they do not control integration on supermanifolds. A meaningful notion of integration on a space with odd, or \emph{fermionic}, directions indeed requires the notion of \emph{Berezinian} bundle (after Berezin) - probably one of the most peculiar constructions in the theory of supermanifolds -, which plays the same role the determinant or canonical bundle play on an ordinary manifold. The fact that the Berezinian bundle does not appear in (the generalization of) the de Rham complex on supermanifolds - and, as such, its sections are not differential forms - marks a significant departure from the ordinary integration theory and its relation to differential forms. Later on, in 1977, in the pioneering \cite{BL1, BL2} Bernstein and Leites introduced the complex of \emph{integral} (or, perhaps more appropriately, integra\emph{ble}) \emph{forms}, which is bounded from above (but not from below) and whose \virgolette top'' bundle is given by the Berezinian line bundle of the supermanifold. Whereas the complex of differential forms controls integration on ordinary bosonic submanifolds of a supermanifold, the complex of integral forms controls integration on sub-supermanifolds of the same odd (or fermionic) dimension of the ambient supermanifold. In this integration procedure, fermionic variables are somehow \virgolette frozen'' and integrated over in the Berezin sense. Bernstein and Leites exploited further the characteristic geometry of supermanifolds where even and odd variables coexist and they also introduced the notion of \emph{pseudodifferential forms} \cite{BL2, BerLei}. This newly-defined type of forms generalizes the notion of differential forms allowing for more general, non-polynomial, dependences on the even 1-forms - while nilpotency constraints the odd 1-forms to have a polynomial dependence. In particular, pseudodifferential forms can be \emph{integrable}, provided that they vanish fast enough at infinity. It is this particular kind of pseudoforms that allow for integration on sub-supermanifolds of \emph{any} codimension, thus supplementing integral forms in the integration theory on supermanifolds, as they can only be integrated on codimension $k|0$ sub-supermanifolds. 
Nonetheless, differently than differential and integral forms, pseudodifferential forms do not carry any grading. 
This problem has been addressed in the Eighties and early Nineties by Voronov and Zorich \cite{VZ1, VZ2, VZ3} and later on by Voronov \cite{V4}. Building on earlier work by Gaiduk, Khudaverdian, and Schwarz \cite{Gaiduk}, these authors were able to develop a theory of forms graded by superdimension. These kinds of forms are now referred to as \emph{Voronov-Zorich forms}. Remarkably, Voronov-Zorich forms admit different descriptions: a quite surprising one being via \emph{Lagrangians} related to parametrized supersurfaces \cite{VZ1}. In more recent years, this intriguing point of view has been further pushed forward by Voronov \cite{V3, V4}, who provided an extension of the de Rham complex with forms carrying a \emph{negative} degree, something possible only due to the peculiar geometry of forms on supermanifolds. Quite remarkably, integral forms, (integrable) pseudodifferential forms, and Voronov-Zorich forms are all related via \emph{integral transformations}. For example, it is possible to define an integral transform \cite{VZ2}, called odd Fourier transform, which maps isomorphically integral forms to a particular class of pseudodifferential forms having a \emph{Dirac delta} distributional dependence on the even 1-forms.\\ 
Foundational problems in the theory of supermanifolds - and in particular in the theory of forms on supermanifolds - have been addressed during the Eighties also by other groups of researchers. Manin, together with the students at his school of algebraic geometry, worked extensively on issues regarding complex supermanifolds (and their deformations) in relation to the back then new-born \emph{superstring theory}, which was polarizing the attention of the high-energy physics community. The fundamental intrinsic definition of the Berezinian bundle as arising from the homology of a very non-trivial generalization of the \emph{Koszul complex} appeared first in \cite{OP}, where Ogievetsky and Penkov, students of Manin at that time, introduced \emph{Serre duality} for projective supermanifolds, showing that the Berezinian bundle provides the right super-analog of the dualizing sheaf in (commutative) algebraic geometry. Later on, this important construction has been briefly reported by Manin in his beautiful book \cite{Manin} - but the paper \cite{OP} was inadvertently not acknowledged, thus causing a bit of confusion regarding the attribution of the result \cite{Ogiev}. Quite recently a self-contained thorough discussion of the super Koszul complex and its relation with the Berezinian module has been given in \cite{NojaRe}. Further crucial properties of differential and integral forms on supermanifolds were discussed by Penkov in relation to the theory of $\mathcal{D}$-modules on supermanifolds in the very beautiful paper \cite{Penkov}: in this work, among many other things, the author shows that the Berezinian bundle carries a natural structure of right $\mathcal{D}$-module, showing once again similarities with the canonical bundle of an ordinary manifold. 
Beside the \virgolette Russian school'', other groups of researchers worked on \virgolette supermathematics'' starting in the Eighties. Bartocci, Bruzzo, and Hern\`andez Ruiperez were very active in the research on foundational problems in the theory of supermanifolds. In particular, it is due to Hern\`andez Ruiperez and Mu\~{n}oz Masque one of the neatest intrinsic construction of the Berezin bundle for smooth supermanifolds \cite{Ruiperez}, to be compared to the aforementioned contemporary construction via super Koszul complex by Ogievetsky and Penkov \cite{OP}, which has the algebraic category in sight instead. \\
The theory of forms and integration on supermanifolds and superspaces has many applications in modern physics. \emph{Supersymmetry} - a building pillar of contemporary high-energy physics, which relates \emph{bosonic} fields to \emph{fermionic} fields \cite{Gaw} -, can only be realized geometrically at the cost of upgrading the ambient manifold of the theory to a supermanifold. In this context, the action of the physical theory is written in terms of a Berezin integral of a section of the Berezinian of the supermanifold, which plays the role of the \emph{Lagrangian density} of the theory, and the supersymmetry transformations are generated by particular \emph{odd} vector fields. The machinery of the Berezin integral and the property of the Berezinian bundle make it possible for the Lie derivative of the action of the physical theory to integrate to zero, thus making supersymmetry into a manifest symmetry. In other words, to set a supersymmetric theory on a supermanifold, with its peculiar notion of integration, has the same meaning as setting a Lorentz-invariant theory on a semi-Riemannian Lorentzian manifold. In this way symmetries of the physical theory become geometric transformations, uncovering mathematical structures and making the symmetries of the theory manifest. \\
It is fair to say, though, that physics has probably not yet absorbed the elegance and subtleness of the theory of forms on supermanifolds: integration on superspaces is regarded as an algebraic formal machinery and the theory of integral forms - in its various realizations - is mostly ignored.  As a result, the full power of this formalism has not been exploited yet in its applications. There exist remarkable exceptions and efforts in this direction though. In \cite{Belo1, Belo2} Belopolsky used a variation on the theme of Voronov-Zorich forms in relation to the physical problem of computing scattering amplitudes in superstring perturbation theory, putting forward a supergeometric version of the so-called \emph{picture changing operators} introduced in conformal field theory in \cite{FMS1, FMS2}. In more recent years, Catenacci, Castellani, and Grassi - together with several other collaborators - realized that the so-called \emph{rheonomic principle} \cite{CDF}, that lies at the basis of the geometric formulation of supergravity theories on supermanifolds, needed to be lifted at the level of the integral forms complex to make sense. This led to the formulation of many supergravity and supersymmetric theories via integral forms and a new understanding of their structures and supersymmetries \cite{CCG, CCGN, CGNinf, CGN, CA1, CGP}. Whereas high-energy physics and string theory communities were rather unwary to the subtleties related to forms and integration theory on supermanifolds, in a totally opposite fashion, the development of the theory of \emph{Batalin-Vilkovisky} (BV) \emph{quantization} has been highly influenced by supergeometry and prompted several advances in the field. In particular, integral forms, in their incarnation as \emph{semi-densities} on odd symplectic supermanifolds, play a major role in BV quantization, see \cite{Mnev} for a supergeometric-aware detailed review of the topic. In this context the very influential work \cite{Schwarz} by Schwarz and the seminal contributions \cite{Khudaverdian1, KN1, KN2, Khudaverdian2} by Khudaverdian and Neressian, regarding the relations between BV geometry and forms and integration on supermanifolds are to be cited. In particular, the first supergeometric definition of the BV Laplacian \cite{Khuda2, Khuda1} is due to Khudaverdian. Later on, building upon the work of Khudaverdian, in \cite{Severa} \v{S}evera provided a wonderful homological construction of the BV Laplacian via a (quite surprising) spectral sequence related to a \virgolette deformed'' de Rham operator by the (odd) symplectic form of the ambient \emph{odd symplectic} supermanifold. Notably, the construction is based on the cohomology of the aforementioned super Koszul complex, introduced by Ogievetsky and Penkov. Finally, in recent years the theory of integral forms and the related integration theory on supermanifolds has been revitalized and drawn back to the attention of the physics community by the review \cite{Witten} of Witten, which was written with an eye to applications to superstrings \cite{WittenSuper}. The point of view of the author emphasizes the relationship of differential and integral forms with Clifford-Weyl (super)algebras and their representations. This prompted new works in the field, also in the realm of pure mathematics, see for example the interesting \cite{SuZhang} in relation to Lie superalgebras. \\  
The books on the topics deserve a separate mention. First off, Voronov's \cite{Voronov} provides a thorough discussion on integration on supermanifolds, emphasizing the author's construction of $r|s$-forms as variations of Lagrangians and the related integration theory. This is the only dedicated book on the topic to this day - we warn the reader, though, that this paper takes a different perspective. Further, Manin's \cite{Manin} features a thorough chapter on supergeometry which, among many other things, introduces integral forms, Berezin integral, and densities. The exposition leans toward an algebraic geometric point of view, which is the one taken also in the present paper: as such, \cite{Manin} could be considered as a main reference. \vspace{.4cm}

The paper is structured as follows. In section 2 the basic constructions in supergeometry are introduced. In particular, the notion of supermanifold is discussed from the point of view of locally-ringed space, and certain natural sheaves are introduced. Section 3 is dedicated to differential forms on supermanifolds and their (de Rham) cohomology, in particular Poincaré lemma is discussed. Section 4 is dedicated to one of the most peculiar constructions in supergeometry, that of Berezinian. More in detail, we will define the Berezinian bundle via three constructions: as a quotient sheaf, as constructed via the super Koszul complex and as byproduct of the cohomology of the so-called total de Rham complex, which mediates between the first two and it is substantially new. The relations between these constructions are commented on. In section 5 we will study some properties of the Berezinian bundle, in particular we will see that it carries a right $\mathcal{D}$-module structure, such as the canonical bundle of an ordinary manifold. Section 6 deals with integral forms and their (Spencer) cohomology, in particular we will prove that the complexes of differential and integral forms are quasi-isomorphic, \emph{i.e.}\ they compute the same cohomology, actually the de Rham cohomology of the reduced manifold.  In section 7 the Berezin integral is introduced. We will prove Stokes' theorem for supermanifolds (without boundaries) and, as an application, we will see how it allows manifest supersymmetry invariance in physics. In section 8 we will see how the notion of Poincaré duality gets modified when working on supermanifolds. Finally, in section 9 we will present a different point of view, by introducing pseudoforms on supermanifolds as particular functions defined on the supermanifold associated to the (parity shifted) tangent bundle of a given supermanifold. The geometry of this peculiar supermanifold is then discussed in detail in a new fashion. We will show that for a specific class of these forms - namely those having a distributional dependence supported at zero in the fiber directions -, there is an isomorphism with the previously defined integral forms. The approach taken in this last section differs from the available literature, 
in that - consistently with the spirit of the paper - only algebraic-geometric inspired ideas and methods have been employed, in place of the traditional analytic approach via integral transforms. It is fair to say also that this last section is more speculative and perhaps less complete and systematic compared to all of the previous ones, as it deals with open problems and ongoing research.  \\
We stress that this paper does not aim to be fully encompassing - and indeed some points of view that we have hinted upon in this introduction are not discussed here - for these, we refer for example to \cite{Voronov}. Instead, we have chosen to provide a - hopefully - conceptually clear and mathematically rigorous exposition, keeping our focus on the algebraic-geometric aspects of the theory. The most important and peculiar constructions - which are not well-known outside a rather small community of experts - have been spelled out in great detail, trying to provide a firmly founded systematization of the results, alongside new comprehensive proofs which are often not available in the literature. Finally, efforts have been put to make the exposition as self-contained as possible in the hope to provide a readable, but not overwhelming, reference to the subject.

\vspace{.3cm}

\noindent {\bf Acknowledgements.} The author wishes to thank the anonymous referee for attentive reading and useful comments on this manuscript. This work is funded by Deutsche Forschungsgemeinschaft (DFG, German Research Foundation) under Germany’s Excellence Strategy EXC-2181/1 - 390900948 (the Heidelberg STRUCTURES Cluster of Excellence).

\section{Elements of Geometry of Supermanifolds}

\noindent In this section we briefly recall the main definitions in the theory of supermanifolds, see for example the classical \cite{Konstant, Leites} or \cite{BR, Manin}. We start with one of the most fundamental concepts in supergeometry, that of superspace.

\begin{definition}[Superspace] A superspace is a pair $\mani \defeq (|\mathpzc{M}|, \mathcal{O}_{\mathpzc{M}})$, where $|\mathpzc{M}|$ is a topological space and $\mathcal{O}_{\mathpzc{M}}$ is a sheaf of $\mathbb{Z}_2$-graded supercommutative rings over $|\mathpzc{M}|$, such that the stalks $\mathcal{O}_{\mathpzc{M}, x}$ at every point of $|\mathpzc{M}|$ are {local rings}.
Analogously, a superspace is a locally ringed space whose structure sheaf is given by a sheaf of $\mathbb{Z}_2$-graded supercommutative rings.
\end{definition}

\noindent Note that the requirement about the stalks being local rings reduces to asking that the \emph{even} component of the stalk is a usual commutative local ring. Indeed if $A = A_0 \oplus A_1$ is a super ring, then $A$ is local if and only if its even part $A_0$ is, see for example \cite{Varadarajan}. \\
Given two superspaces we can define a morphism between them in the usual fashion. 
\begin{definition}[Morphisms of Superspaces] Given two superspaces $\mathpzc{M}$ and $\mathpzc{N}$ a morphism $\varphi : \mathpzc{M} \rightarrow \mathpzc{N}$ is a pair $\varphi \defeq (\phi, \phi^\sharp)$ where 
\begin{enumerate}[leftmargin=*]
\item
$\phi : |\mathpzc{M}| \rightarrow |\mathpzc{N}|$ is a continuous morphism of topological spaces; 
\item $\phi^\sharp : \mathcal{O}_{\mathpzc{N}} \rightarrow \phi_* \mathcal{O}_\mathpzc{M}$ is a morphism of sheaves of $\set{Z}_2$-graded rings, having the properties that it preserves the $\set{Z}_2$-grading and that given any point $x\in |\mathpzc{M}|$, the homomorphism
$
\phi^\sharp_x : \mathcal{O}_{\mathpzc{N}, \phi (x)} \rightarrow \mathcal{O}_{\mathpzc{M}, x}
$
is local, \emph{i.e.}\ it preserves the (unique) maximal ideal, \emph{i.e.}\ $\phi^\sharp_x (\mathfrak{m}_{\phi (x)}) \subseteq \mathfrak{m}_x.$
\end{enumerate}
\end{definition} 
{\noindent It is easy to see that superspaces together with their morphisms form a category, we call it $\mathbb{S}\mathbf{Sp}$. Before we go on, some remarks on the previous definitions are in order. 
\remark With an eye to the ordinary theory of schemes in algebraic geometry, we stress that the request that the morphism $\phi^\sharp_x : \mathcal{O}_{\mathpzc{N}, \phi (x)} \rightarrow \mathcal{O}_{\mathpzc{M}, x} $ preserves the maximal ideal in the second point of the definition above is of particular significance in supergeometry. Indeed it is important to notice that the structure sheaf $\stsheaf$ of a superspace is in general \emph{not} a sheaf of functions. As long as the structure sheaf $\stsheaf $ of a certain space or, more in general, of a scheme, is a sheaf of functions, then a section $s$ of $\stsheaf$ takes values in the field of fractions $k (x) = \mathcal{O}_{\mani, x} / \mathfrak{m}_x$ that depends on the point $x\in |\mani|$, as a function $x \mapsto s(x) \in k(x)$, and the maximal ideal $\mathfrak{m}_x $ contains the germs of functions that vanish at $x \in |\mani|$. In the case of superspaces, nilpotent sections - and thus in particular all of the odd sections - would be identically equal to zero as functions on points, and indeed the maximal ideal $\mathfrak{m}_x$ contains the germs of all the nilpotent sections in $\mathcal{O}_{\mani, x}.$ In this context, the request that $\phi^\sharp_x : \mathcal{O}_{\mathpzc{N}, \phi (x)} \rightarrow \mathcal{O}_{\mathpzc{M}, x} $ is local becomes crucial, while in the case of a genuine sheaf of functions the locality is automatically achieved. In particular, locality implies that a \emph{non} unit element in the stalk $\mathcal{O}_{\mathpzc{N}, \phi (x)}$, such as a germ of a nilpotent section, can only be mapped to another \emph{non} unit element in $\mathcal{O}_{\mathpzc{M}, x}$, such as another germ of a nilpotent section. In other words, nilpotent elements cannot be mapped to invertible elements.\\
We advise the reader, though, that we will often abuse the notation by keep denoting the morphism of sheaves related to a superspace morphism $\varphi : \mathpzc{M} \rightarrow \mathpzc{N}$ by $\phi$, instead of $\phi^\sharp$.}
{\remark It is crucial to observe that one can always construct a superspace out of two \virgolette classical'' data: given a smooth manifold $\mani_{\mathpzc{red}}$ with underlying topological space $|\mathpzc{M}|$, and a smooth vector bundle over $\mathpzc{M}_{\mathpzc{red}}$, call it $\mathcal{E}$, 
then the sheaf of smooth sections of the bundle of exterior algebras $\mathcal{\bigwedge^\bullet} \mathcal{E}^\ast$ has an obvious $\set{Z}_2$-grading (by taking its natural $\set{Z}$-grading $\mbox{mod}\,2$). Therefore in order to realize a superspace it is enough to take the structure sheaf $\mathcal{O}_\mathpzc{M}$ of the superspace to be the sheaf of smooth sections of the bundle of exterior algebras of $\mathcal{E}$. This construction is so important to bear its own name \cite{DonWit}.}
\begin{definition}[Local Model $\mathfrak{S}(\mathpzc{M}_{\mathpzc{red}}, \mathcal{E})$] Given a pair $(\mathpzc{M}_{\mathpzc{red}}, \mathcal{E})$, where $\mathpzc{M}$ is a smooth manifold and $\mathcal{E}$ is a smooth vector bundle over $\mani_{\mathpzc{red}}$, we call $\mathfrak{S}(\mathpzc{M}_{\mathpzc{red}}, \mathcal{E})$ the superspace modelled on the pair $(\mathpzc{M}_{\mathpzc{red}}, \mathcal{E})$, where the structure sheaf is given by the 
smooth sections of the exterior algebra bundle $\bigwedge^\bullet \mathcal{E}^\ast$. 
\end{definition}  
\noindent  
One can also work in richer and more structured categories, such as the \emph{complex analytic} or \emph{algebraic} categories - not only in the \emph{real smooth} category, as in the above definition of local model. This amount to consider local models based on the pair $(\manir, \mathcal{E})$, where $\manir$ is a complex manifold or an algebraic variety - we keep denoting its underlying topological space with $|\mani|,$ as above - with $\mathcal{O}_{\manir}$ being its sheaf holomorphic of algebraic functions and $\mathcal{E}$ being a holomorphic or algebraic vector bundle.  \\
We will call \emph{smooth, holomorphic} or \emph{algebraic local model} a local model $\mathfrak{S} (\manir,\mathcal{E})$ which is constructed from the above data. This leads to the definition of supermanifold in the appropriate category, which we explicitly give in the real smooth or complex analytic category.  
\begin{definition}[Real / Complex Supermanifold] A real (complex) supermanifold $\mani $ of dimension $n|m$ is a superspace that is \emph{locally} isomorphic to some smooth (holomorphic) local model $\mathfrak{S} ( \manir, \mathcal{E})$, where $\manir$ is a smooth (complex) manifold of dimension $n$ and $\mathcal{E}$ is a smooth (holomorphic) vector bundle on $\manir $ of rank $m$.
\end{definition}
\noindent In other words, if $\manir$ is covered by an atlas $\{{U}_i \}_{i \in I}$, the structure sheaf $\stsheaf = \mathcal{O}_{\mani, 0} \oplus \mathcal{O}_{\mani, 1}$ of the supermanifold $\mani$ is described via a collection $\{ \psi_{{U}_i} \}_{i\in I}$ of \emph{local} isomorphisms of sheaves 
\bear
U_i \longmapsto \psi_{{U}_i} :  \mathcal{O}_{\mathpzc{M}}\lfloor_{U_i} \stackrel{\cong}{\longrightarrow} \wedge^\bullet \mathcal{E}^\ast \lfloor_{U_i} 
\eear 
where we have denoted with $\bigwedge^\bullet \mathcal{E}^\ast$ the sheaf of $\mathcal{O}_{\manir}$-valued sections of the exterior algebra of $\mathcal{E}$ considered with its $\mathbb{Z}_2$-gradation. Also, notice that a morphism of supermanifolds is nothing but a morphism of superspaces, so that one has the related category of supermanifolds, that we denote with $\mathbb{S}\mathbf{Man}$. \\
The special case in which there is a single \emph{global} isomorphism instead of a family of local isomorphisms deserves a name of its own. 
\begin{definition}[Split Supermanifold] We say that a supermanifold $\mani$ is a split supermanifold if it is \emph{globally} isomorphic to its local model. Analogously, $\mani$ is split if there exists a sheaf isomorphism $\mathcal{O}_\mani \cong \bigwedge^\bullet \mathcal{E}^\ast.$
\end{definition} 
\noindent Clearly, split supermanifolds are the easiest supermanifolds to deal with, as their structure sheaves are simply sheaves of $\mathcal{O}_{\manir}$-valued sections of exterior algebras, and as such they are locally-free sheaves of $\mathcal{O}_{\manir}$-modules. \\
In order to see how real and complex supermanifolds might differ from the point of view of their global geometry,  we need to introduce some further pieces of information related to a supermanifold.
\begin{definition}[Nilpotent Sheaf] Let $\mani $ be a real (complex) supermanifold with structure sheaf $\mathcal{O}_\mani$. We call the {nilpotent sheaf} $\mathcal{J}_\mani$ of $\mani$ the sheaf of ideals of $\stsheaf = \mathcal{O}_{\mani, 0} \oplus \mathcal{O}_{\mani, 1}$ generated by \emph{all} of the nilpotent sections in $\mathcal{O}_\mani$, \emph{i.e.} we put $\mathcal{J}_\mani \defeq \mathcal{O}_{\mani,1} \oplus \mathcal{O}_{\mani, 1}^2$. 
\end{definition}
\noindent It is crucial to note that modding out all of the nilpotent sections from the structure sheaf $\stsheaf $ of the supermanifold $\mani$ we recover the structure sheaf $\mathcal{O}_{\manir}$ of the underlying ordinary  manifold $\manir$.  
\begin{definition}[Reduced Space] Let $\mani$ be a real (complex) supermanifold with structure sheaf $\mathcal{O}_\mani$. We call the reduced space of $\mani$ the smooth (complex) manifold $\manir$ with structure sheaf given by the quotient $\mathcal{O}_{\manir} \defeq \mathcal{O}_{\mani} / \mathcal{J}_\mani$.   
\end{definition}
\noindent Loosely speaking, the reduced manifold in $\manir$ arises by setting all the nilpotent sections in $\stsheaf $ to zero. In other words, more invariantly, attached to any real or complex supermanifold there is a short exact sequence that relates the supermanifold to its reduced manifold
\bear \label{ses}
\xymatrix@R=1.5pt{ 
0 \ar[rr] && \mathcal{J}_\mani \ar[rr] &&  \stsheaf  \ar[rr]^{\iota^\sharp\; \; \;} && \mathcal{O}_{\mani_{\mathpzc{red}}} \ar[rr] && 0. 
}
\eear
The surjective sheaf morphism $\iota^\sharp : \mathcal{O}_\mani \rightarrow \mathcal{O}_{\mani_{\mathpzc{red}}}$ corresponds to the existence of an \emph{embedding} $\manir \stackrel{\iota}{\longhookrightarrow} \mani$ of the reduced manifold $\manir $ inside the supermanifold $\mani$. Notice that $\mathcal{J}_\mani = \ker (\iota)$, where $\iota : \stsheaf \rightarrow \mathcal{O}_{\mani_{\mathpzc{red}}}$ is the surjective sheaf morphism in \eqref{ses}.\\ 
Using the nilpotent sheaf associated to a supermanifold $\mani$, say of odd dimension $m$ we can construct a \emph{descending filtration} of length $m$ of $\stsheaf $ as follows,
\bear \label{filt}
\stsheaf \supset \mathcal{J}_\mani \supset \mathcal{J}_\mani^2 \supset \mathcal{J}_\mani^3 \supset \ldots \supset \mathcal{J}_\mani^{q} \supset \mathcal{J}_\mani^{m+1} = 0.
\eear
This allows us to give the following definition.
\begin{definition}[$\mbox{Gr} \, \stsheaf$ and $\mbox{Gr}\, \mani$] Let $\mani $ be a supermanifold having odd dimension $m$ together with the filtration of its structure sheaf $\stsheaf$ as in \eqref{filt}. We define the following sheaf of supercommutative algebras
\bear \label{grsheaf}
\mbox{{Gr}}\, \stsheaf \defeq \bigoplus_{i=0}^m \mbox{{Gr}}^{(i)}  \stsheaf = \stsheafred \oplus \slantone{\mathcal{J}_\mani}{\mathcal{J}^2_{\mani}} \oplus \ldots \oplus \slantone{\mathcal{J}_\mani^{m-1}}{\mathcal{J}^m_{\mani}} \oplus {\mathcal{J}^m_\mani}.
\eear
where $\mbox{{Gr}}^{(i)} \stsheaf \defeq \slantone{\mathcal{J}_\mani^i}{\mathcal{J}_\mani^{i+1}}$ and the $\set{Z}_2$-grading is obtained by taking the obvious $\set{Z}$-grading $\mbox{\emph{mod}}\,2$. We call the \emph{split supermanifold associated to} $\mani$ the supermanifold arising from the superspace  $(|\mani|, \mbox{{Gr}}\, \stsheaf)$ and we denote it by $\mbox{{Gr}}\, \mani$.
\end{definition}
\noindent If the supermanifold $\mani$ has odd dimension $m$, the quotient $\mbox{Gr}^{(1)} \, \stsheaf = \mathcal{J}_\mani / \mathcal{J}_\mani^2$ in \eqref{grsheaf} is a locally-free sheaf of $\mathcal{O}_{\manir}$-modules of rank $0|m$ - \emph{i.e.} it is locally generated by $m$ odd sections - and it plays a special role so that for notational convenience we denote it $\mathcal{F}_\mani \defeq \mathcal{J}_{\mani} / \mathcal{J}^2_{\mani}$, as to recall its \virgolette fermionic'' behavior. \\
The importance of $\mathcal{F}_\mani$ lies in that - up to parity - it is isomorphic to $\mathcal{E}^\ast$, the vector bundle appearing in the local model $\mathfrak{S} (\manir, \mathcal{E})$ whose the supermanifold $\mani$ is based upon, \emph{i.e.} one has $\mathcal{E}^\ast \cong \Pi \mathcal{F}_\mani$, where $\Pi : \mathbf{Sh}_{\mathcal{O}_\mani} \rightarrow \mathbf{Sh}_{\mathcal{O}_\mani}$ is the so-called \emph{parity changing} or \emph{parity shifting} functor, which maps a locally-free sheaf of rank of rank $p|q$ to one of rank $q|p$, by reversing its parity \cite{Deligne, Manin}. In view of this one has that $\mbox{Gr}\, \mani = \mathfrak{S} ({\manir, \Pi \mathcal{F}_\mani^\ast})$ and the local isomorphisms characterizing the structure sheaf $\mathcal{O}_\mani$ can be rewritten, over an open set $U \subset |\mani|$, as 
\bear
\mathcal{O}_\mani \lfloor_{U} \cong \wedge^\bullet \mathcal{E}^\ast\lfloor_{U} \cong \cat{S}^\bullet \mathcal{F}_\mani \lfloor_{U},
\eear
where $\cat{S}^i : \mathbf{Sh}_{\stsheaf} \rightarrow \mathbf{Sh}_{\stsheaf}$ is the $i$-th \emph{super}symmetric power functor \cite{Manin}, so that in particular $\mbox{Gr}\, \stsheaf = \bigoplus_{i=0}^m \cat{S}^i \mathcal{F}_\mani$.  \\

\noindent We conclude this section by stressing out the major difference between the realm of real and complex supermanifolds, which lies in the fact that \emph{real supermanifold are always split} and actually all isomorphic to $\mbox{Gr} \, \mani$. This result - in a slightly different form - was first proved by Marjorie Batchelor in \cite{Bat}.
\begin{theorem}[Batchelor] \label{Batchelor} Let $\mani$ be a smooth supermanifold. Then its structure sheaf is \emph{non-canonically} isomorphic to a sheaf of exterior algebras for some smooth vector bundle $\mathcal{E}$ over $\manir$. In particular, $\mani$ is split and one has $\mani \cong \mbox{\emph{Gr}}\, \mani$. 
\end{theorem}
{\remark It is important to note that the theorem only states the existence of such isomorphism. This means that it guarantees the existence of a certain covering of open sets together with charts such that the isomorphism is realized, but it is not \emph{constructive}: in other words, it does not tell how to concretely realize such an isomorphism.  
\remark One can further spell out the meaning of Batchelor's theorem \ref{Batchelor} in a local-to-global fashion. Indeed, the result guarantees  
that charts $\{ U_\ell, x_i^\ell | \theta^\ell_\alpha \}_{\ell \in I}$ for $i = 1, \ldots, n$ and $\alpha = 1, \ldots, m$ can always be found such that if $x^\mathpzc{j}_i | \theta^\mathpzc{j}_\alpha$ and $x^\mathpzc{k}_i | \theta^\mathpzc{k}_\alpha$ are local coordinates in any two open sets $U_\mathpzc{j}$ and $U_\mathpzc{k}$ in $\{ U_\ell \}_{\ell \in I}$ having non-empty intersection $U_\mathpzc{k} \cap U_\mathpzc{j} \neq \emptyset$, then we will have 
\bear
x^\mathpzc{k}_i = x^\mathpzc{k}_i (x^\mathpzc{j}_1, \ldots, x^\mathpzc{j}_n), \qquad \theta^\prime_\alpha = \sum_{\beta = 1}^m [g_{\mathpzc{k \, j}} (x)]_{\alpha \beta} \theta_\beta,
\eear
where $[g_{\mathpzc{k \, j}}(x)]_{\alpha \beta} \in \check{Z}^1 (U_{\mathpzc{k}} \cap U_{\mathpzc{j}}, GL (q, \mathbb{R}))$ are the transition functions of the vector bundle $\mathcal{E}^\ast.$ That is, when changing charts, the even local coordinates $x$'s transform as the coordinates of the ordinary smooth manifolds $\manir$ and the odd local coordinates $\theta$'s transform linearly, as the generating sections of the vector bundle $\mathcal{E}^\ast.$}
{\remark Loosely speaking, a general \emph{obstruction theory} to split the structure sheaf of a supermanifold can be constructed by filtering $\mathcal{O}_\mani$ as in \eqref{grsheaf}. This was first done by Green in \cite{Green}. For a supermanifold based on the local model $\mathfrak{S} (\manir, \mathcal{E})$, obstructions are given by cohomology classes $\omega_i$ lying in the cohomology groups 
\bear
\omega_{2i} \in H^1 (\manir, \mathcal{T}_{\manir} \otimes \wedge^{2i} \mathcal{E}^\ast), \qquad \omega_{2i+1} \in H^1 (\manir, \mathcal{E} \otimes \wedge^{2i+1} \mathcal{E}^\ast) 
\eear
for $i\geq 1$ and where $\mathcal{T}_{\manir}$ is the tangent sheaf of the reduced space $\manir.$ The subtlety here is that whereas the \emph{fundamental obstruction} $\omega_2$ is always defined, the higher obstructions $\omega_{i}$ for $i \geq 3$ are only defined if all the lower ones vanish instead \cite{DonWit, Green}. \\
In the smooth category all the sheaves are \emph{fine} as a consequence of the existence of smooth \emph{partitions of unity}, then these cohomology groups are automatically zero for a real supermanifold and there are no obstructions to split $\mathcal{O}_\mani$ as $\mbox{Gr} \, \mani$, providing another proof of Batchelor theorem.\\
On the other hand things change dramatically for complex supermanifolds, where the above cohomology groups can indeed be non-zero, thus leading to the peculiar - and very interesting! - complex supergeometry of \emph{non-split} and \emph{non-projected} supermanifolds, \emph{i.e.}\ those complex supermanifolds that do not admit a retraction or \emph{projection} for their defining exact sequence \eqref{ses}
\bear \label{splittingseq}
\xymatrix@R=1.5pt{ 
0 \ar[rr] && \mathcal{J}_\mani \ar[rr] & &  \stsheaf  \ar[rr]_{\iota} && \ar@{-->}@/_1.3pc/[ll]_{\pi} \mathcal{O}_{\mani_{\mathpzc{red}}} \ar[rr] && 0,
 }
\eear
where $\pi^\sharp : \mathcal{O}_{\mani_{\mathpzc{red}}} \rightarrow \stsheaf$ is such that $\iota^\sharp \circ \pi^\sharp = id_{\mathcal{O}_{\manir}}$ and corresponds to a map $\pi : \mani \rightarrow \manir$. Notice that, in particular, the structure sheaf of a non-projected supermanifold is \emph{not} a sheaf of $\mathcal{O}_{\manir}$-algebras, and as such a non-projected supermanifold cannot be \virgolette reconstructed'' easily from its underlying ordinary complex manifold: in the words of Donagi and Witten in \cite{DonWit}, a non-projected supermanifold \virgolette has a life of its own''.\\
In complex supergeometry the absence of a projection should not be looked at as an exotic and mostly rare phenomenon, but as a characterizing and quite common phenomenon instead. For example, the easiest supergeometric generalization of a conic in the complex projective superspace $\mathbb{CP}^{2|2}$ cut out by the equation 
\bear
X_0^2 + X_1^2 + X_2^2 + \Theta_1 \Theta_2 = 0 \; \subset \; \mathbb{CP}^{2|2},
\eear  
turns out to be a $1|2$-dimensional non-projected supermanifold, see for example \cite{DonWit} or \cite{Manin}, which is characterized up to isomorphism by the triple 
\bear
\mathcal{C}_{\omega} \defeq (\manir = \mathbb{CP}^1,\; \; \mathcal{E}^\ast = \mathcal{O}_{\mathbb{CP}^1} (-2)^{\oplus 2},\; \;  H^1 (\manir, \mathcal{T}_{\manir} \otimes \wedge^2 \mathcal{E}^\ast) \owns \omega_2 \neq 0),
\eear
notice that in this case $\omega_2 \in H^1 (\mathbb{CP}^1, \mathcal{O}_{\mathbb{CP}^1} (-2)) \cong \mathbb{C}$. In light of the Batchelor theorem \cite{Bat}, one can see the difference in terms of the structure of transition functions between split supermanifolds and non-projected (or non-split) supermanifolds. Indeed, covering the reduced space $\mathbb{CP}^1$ by the standard (two) open sets, and considering the related systems of local coordinates to be given by $(z | \theta_1, \theta_2) $ and $(w | \psi_1, \psi_2)$ respectively, for the above super conic $\mathcal{C}_{\omega}$ one finds 
\bear \label{transp1}
w = \frac{1}{z} +  \frac{\theta_1 \theta_2}{z^3}, \quad \psi_1 = \frac{\theta_1}{z^2}, \quad \psi_2 = \frac{\theta_2}{z^2}.
\eear
The non-vanishing of the obstruction class $\omega_2 \neq 0 $ in cohomology guarantees that there are \emph{no} choices of coordinates or redefinitions of charts that bring the even transition functions back to the form of that of the reduced space $\mathbb{CP}^1$, \emph{i.e.} $w = 1/z$. On the contrary, for a non-projected supermanifold the even transition functions depend crucially also on the odd part of the geometry, as one can see in the \eqref{transp1}.
}
\subsection{Derivations, Differential Operators and $\mathcal{D}$-modules} Having introduced the concept of supermanifolds - in particular in the smooth real and complex analytic category -, we now briefly discuss the most important \emph{natural} sheaves that can be defined on them and that will be used the most in this paper. \\
In particular, working for instance over a complex supermanifold $\mani$, the structure sheaf $\stsheaf$ can be looked at as a \emph{subsheaf} of the sheaf of its $\mathbb{C}$-endomorphisms, $\mathcal{E}nd_\mathbb{C}(\mathcal{O}_\mani)$, taking $g \mapsto \sigma_{f} (g) \defeq fg$ for any sections $f, g \in \mathcal{O}_\mani$. Likewise, the \emph{tangent sheaf} $\mathcal{T}_\mani$ of $\mani$ - or, analogously, the \emph{sheaf of derivations} $\mathcal{D}er_{\mathbb{C}}(\mathcal{O}_\mani)$ - can also be looked at as a subsheaf of $\mathcal{E}nd_{\mathbb{C}} (\mathcal{O}_\mani)$ defining
\bear
\mathcal{T}_\mani \defeq \{ X \in \mathcal{E}nd_{\mathbb{C}} (\mathcal{O}_\mani) : X (fg) = X(f)g + (-1)^{|X||f|} f X(g), \; f, g \in \stsheaf\}.
\eear
Since we only consider \emph{smooth} supermanifolds, then $\mathcal{T}_\mani$ is \emph{locally-free} of rank $n | m= \dim_{\mathbb{C}}\mani$ and if $ x_i |\theta_\alpha$ are local coordinates for a chart $U \subset \mani$, then a section $X \in \mathcal{T}_{\mani}$ over $U$ is given by
\bear
X_{U} = \sum_{i=1}^n f_i^U \cdot \partial_{x_i} + \sum_{\alpha = 1}^m f_\alpha^U \cdot \partial_{\theta_\alpha}, 
\eear
where $f^U_i, f^U_\alpha \in \mathcal{O}_{\mani} (U)$. This means that the tangent sheaf is freely locally-generated by the even and odd derivations 
\bear
\mathcal{T}_\mani ({U}) = \mathcal{O}_{\mani} (U) \cdot \{ \partial_{x_1}, \ldots, \partial_{x_n} | \partial_{\theta_1}, \ldots \partial_{\theta_m} \}. 
\eear
The above point of view, aimed at relating the structure sheaf and the tangent sheaf with the sheaf of endomorphisms of $\stsheaf$ is particularly useful when one is interested in introducing the sheaf of \emph{differential operators} on $\mani$, which will play an important role in what follows. We give the definition for a complex supermanifold, but the same can be done for a real and also algebraic supermanifold.
\begin{definition}[The Sheaf $\mathcal{D}_\mani$] Let $\mani$ be a complex supermanifold. We define the sheaf of differential operators of $\mani$ to be the subsheaf of $\mathcal{E}nd_\mathbb{C} (\stsheaf)$ generated by $\stsheaf $ and $\mathcal{T}_\mani$, and we denote it by $\mathcal{D}_\mani$.
\end{definition}
\noindent If $x_a | \theta_\alpha$ is a coordinate system over an open set $U$, then $\{x_i | \theta_\alpha, \partial_{x_i} | \partial_{\theta_\alpha} \}_{i=1, \ldots, n, \alpha = 1, \ldots, m} $ induces a local trivialization of $\mathcal{D}_\mani \lfloor_U$, where $x_i | \theta_\alpha \in \mathcal{O}_\mani \lfloor_U$ and $\partial_{x_i} | \partial_{\theta_\alpha} \in \mathcal{T}_\mani \lfloor_{U}$ satisfy the following defining relations   
\begin{align} \label{comm}
 [x_i, x_j] = 0, \quad  &[\partial_{x_i} , \partial_{x_j} ] = 0, \quad  [x_i , \partial_{x_j}] = \delta_{ij}, \quad  \{\theta_\alpha , \theta_{\beta} \} = 0, \quad  \{ \partial_{\theta_\alpha}, \partial_{\theta_\beta} \} = 0, \quad  \{ \theta_\alpha, \partial_{\theta_\beta} \} = \delta_{\alpha \beta} \nonumber \\
 & [x_i, \theta_\alpha] = 0, \qquad  [\partial_{x_i} , \partial_{\theta_\alpha} ] = 0, \qquad  [x_i , \partial_{\theta_\alpha}] = 0, \qquad [\theta_\alpha, \partial_{x_i}] = 0,   
\end{align}
where $[\cdot, \cdot ]$ denotes a commutator and $\{ \cdot, \cdot \}$ denotes an \emph{anti}commutator. Notice that locally, these relations define the \emph{Weyl superalgebra} $\mathcal{D}_{\mathbb{C}^{n|m}} $, so that posing 
\bear
U \longmapsto \mathcal{D}_\mani (U) \defeq \{ D_U \mbox{ is a differential operator on } \mathcal{O}_\mani (U)\},
\eear 
then one has 
\bear \label{localD}
D_U = \bigoplus_{\ell \in \mathbb{N}^n, \varepsilon \in \mathbb{Z}_2^m} \mathcal{O}_{\mani}\lfloor_{U} \partial_x^{\ell} \partial_{\theta}^{\varepsilon} \qquad \mbox{where} \qquad \left \{ \begin{array}{l} \partial^\ell_x \defeq \partial_{x_1}^{\ell_1} \partial_{x_2}^{\ell_2} \cdots \partial_{x_n}^{\ell_n}, \\ \partial_{\theta}^{\varepsilon} \defeq \partial_{\theta_1}^{\varepsilon_1} \partial_{\theta_2}^{\varepsilon_2} \cdots \partial_{\theta_n}^{\varepsilon_n} \end{array} \right. 
\eear
where $\ell_i \in \mathbb{N}$ and $\varepsilon_\alpha \in \mathbb{Z}_2 = \{0,1\}.$
In particular we define $\deg (D_U) \defeq \max (|\ell| + |\varepsilon|)$, if $|\ell| = \sum_{i = 1}^n \ell_i$ and $| \varepsilon | = \sum_{\alpha = 1}^m \varepsilon_\alpha$ and we call $\deg (D_U)$ the \emph{degree} of the differential operator $D_U \in \mathcal{D}_\mani \lfloor_{U}$. \\
The above local description of equation \eqref{localD} leads to a natural \emph{filtration} of $\mathcal{D}_\mani$. We define
\bear
{F}^i \mathcal{D}_\mani (U) \defeq \left \{ D_U \in \mathcal{D}_{\mani } (U) : \deg (D_U \lfloor_{V} ) \leq i \mbox{ for all } V \subseteq U \right  \}. 
\eear
Notice that the above filtration is \emph{increasing}, \emph{i.e.}\
$F^{i} \mathcal{D}_\mani \subseteq F^{i+1} \mathcal{D}_\mani $, and \emph{exhaustive}, \emph{i.e.}\ $\cup_i F^i \mathcal{D}_\mani = \mathcal{D}_\mani$. More in general one has the following relations
\bear
F^i \mathcal{D}_{\mani} \cdot F^j \mathcal{D}_{\mani} \subseteq F^{i+j} \mathcal{D}_{\mani}, \quad \qquad [F^i \mathcal{D}_{\mani}, F^j \mathcal{D}_{\mani} ] \subseteq F^{i+j-1}\mathcal{D}_{\mani},
\eear
and we define the \emph{associated graded module} $\mbox{gr}^\bullet_F (\mathcal{D}_\mani)$ with respect to the above filtration:
\bear
\mbox{gr}^{\bullet}_F (\mathcal{D}_\mani) \defeq \bigoplus_{i = 0}^\infty \mbox{gr}^i_F (\mathcal{D}_\mani) = \bigoplus_{i = 0}^\infty \slantone{F^i \mathcal{D}_\mani}{F^{i-1} \mathcal{D}_\mani},
\eear
where we have defined $\mbox{gr}^i_F (\mathcal{D}_\mani) \defeq F^i \mathcal{D}_\mani / F^{i-1}\mathcal{D}_\mani$. 
{\remark It has to be stressed that, in particular, $F^1 \mathcal{D}_\mani$ is a \emph{Lie sub-superalgebra} of $\mathcal{D}_\mani$, since indeed one has that $
[ \cdot , \cdot ]: F^1 \mathcal{D}_{\mani} \times F^1 \mathcal{D}_{\mani} \rightarrow F^1 \mathcal{D}_{\mani}.$ Further, $F^0 \mathcal{D}_\mani$ is a \emph{Lie ideal} of this Lie superalgebra: indeed if $f \in F^0 \mathcal{D}_\mani$, then $[f, G] \in F^0 \mathcal{D}_\mani$ for any $G \in F^1 \mathcal{D}_\mani$. It thus makes sense to consider the quotient of $F^1 \mathcal{D}_\mani$ by $F^0 \mathcal{D}_\mani$, and it is not hard to realize that
\bear
\mathcal{T}_\mani \cong \slantone{F^1 \mathcal{D}_\mani}{F^0 \mathcal{D}_\mani},
\eear
which in turn leads to  
\bear
\mbox{gr}_F^{\bullet} (\mathcal{D}_\mani) = \cat{S}^\bullet_{\stsheaf} \mathcal{T}_\mani.
\eear}
{\remark One could think about this in analogy with the \emph{Poincaré-Birkhoff-Witt} (PBW) theorem for the \emph{universal enveloping algebra} $U (\mathfrak{g})$ of a certain Lie (super)algebra $\mathfrak{g}$. Indeed, given a filtration as above, the quotient operation does \emph{not} just reduce to the leading terms, but it does remarkably more: it sets all the commutators in $\mbox{gr}_F^{i} (\mathcal{D}_{\mani})$ to zero. Indeed, if two elements $F, G$ are such that their \emph{product} is in $\mbox{gr}_F^{i} (\mathcal{D}_{\mani})$, then their \emph{commutator} is in $\mbox{gr}_F^{i-1} (\mathcal{D}_{\mani })$ so that it is set to zero in the quotient. In other words, elements surviving the quotient are those that do \emph{not} come from commutators: in particular, all of the elements commute and this leads naturally to the symmetric (super)algebras, where all of the (super)commutators vanish.
Notice that, similarly, 
$\mbox{gr}^{\bullet} U(\mathfrak{g}) \cong S^{\bullet} (\mathfrak{g})$, which is the meaning of PBW theorem. In general, we might think about the following analogies between a Lie (super)algebra and its universal enveloping algebra (together with its PBW filtration) and the derivations on a (super)manifolds, and the differential operators on $\mani$: 
\begin{align}
&\mathcal{T}_{\mani} \quad \leftrightsquigarrow \quad \mathfrak{g}\\
&\mathcal{D}_{\mani} \quad \leftrightsquigarrow \quad U (\mathfrak{g})
\end{align} }
{\remark Finally, it is worth observing that the associated graded module $\mbox{gr}^\bullet_F (\mathcal{D}_\mani)$ is naturally isomorphic to $\pi_{\ast} \mathcal{O}_{\tiny{\cat{T}}^\ast \mani}$, where $\cat{T}^\ast \mani$ is the supermanifold constructed from $\pi : \mathcal{T}^\ast_\mani \rightarrow \mani$, the cotangent bundle of the supermanifold $\mani$. This expresses the fact that \virgolette functions'' on a certain space are given by the symmetric algebra over the dual space - for more on this, see for example the last part of this paper or the recent \cite{Noja}. 

{\remark As made clear by the previous discussion, the sheaf $\mathcal{D}_\mani$ is a sheaf of \emph{non-commutative} algebras, as it follows from the non-trivial relations $[x_i, \partial_{x_j}] = \delta_{ij}$ and $\{ \theta_\alpha, \partial_{\theta_{\beta}}\} = \delta_{\alpha \beta}$. The sheaf of differential operators over a certain (super)manifold, defined as above, is the prototypical example of sheaf of \emph{non-commutative} algebras. Clearly, this leads to the fact that when an action involving $\mathcal{D}_\mani$ is concerned, then it is  important to distinguish between \emph{left} and \emph{right} actions, as in the following definition.}
\begin{definition}[$\mathcal{D}_\mani$-Modules] Let $\mani$ be a complex supermanifold and let $\mathcal{E}$ be a sheaf over $\mani$. We say that $\mathcal{E}$ is a sheaf of \emph{left}/\emph{right} $\mathcal{D}_\mani$-modules, or simply a left/right $\mathcal{D}_\mani$-module, if $\mathcal{E} (U)$ is endowed with a left/right $\mathcal{D}_\mani (U)$-module structure for any open set $U$, which is compatible with the restriction morphisms of the sheaf.
\end{definition}
\noindent It is to be observed that a left $\mathcal{D}_\mani$-module can not at all be endowed with the structure of right $\mathcal{D}_\mani$-module and vice versa. In this paper we will indeed come across left $\mathcal{D}_\mani$-module which are \emph{not} right $\mathcal{D}_\mani$-modules and vice versa. 

\section{Differential Forms and de Rham Cohomology of Supermanifolds}

\noindent Given a smooth real supermanifold $\mani$ of dimension $n|m$ with structure sheaf $\mathcal{O}_\mani$ and having introduced its tangent sheaf $\mathcal{T}_\mani = \mathcal{D}er_{\mathbb{R}} (\stsheaf)$, it is immediate to consider its dual sheaf $\mathcal{H}om_{\stsheaf} (\mathcal{T}_\mani, \mathcal{O}_\mani) \cong \mathcal{T}^{\ast}_\mani$. This is a locally-free sheaf of $\mathcal{O}_\mani$-module of rank $n|m$, 
and one has the usual duality pairing $\mathcal{T}^\ast_\mani \otimes \mathcal{T}_\mani \rightarrow \mathcal{O}_{\mani}$ defined as $dy_a (\partial_{y_b}) = \delta_{ab}$ if $y_a \defeq x_i | \theta_{\alpha}$. Notice that for this to be a \emph{perfect} pairing one needs to assign a $\mathbb{Z}_2$-parity to the generators in a way such that $|dx_i | = 0$ and $|d \theta_\alpha | =1$. When reducing to the underlying real manifold $\manir$, with an eye to the associated de Rham algebra of differential forms with respect to the wedge product, this would lead to the awkward situation of having a basis of commuting $dx$'s. \\
In order to restore the correspondence with the usual de Rham theory on the underlying reduced space, it is therefore customary to define the sheaf 1-forms $\Omega^1_{\mani} $ on a supermanifold $\mani $ to be the \emph{parity shifted} of $\mathcal{T}^\ast_\mani$, \emph{i.e.} we take $\Omega^{1}_\mani \defeq \mathcal{H}om_{\stsheaf} (\Pi \mathcal{T}_\mani, \mathcal{O}_\mani )$. This is a locally-free sheaf of $\mathcal{O}_\mani$-modules of rank $m|n$ and if $x_i | \theta_\alpha$ are local coordinates over a certain open set $U$, then one has
\bear
\Omega^1_{\mani} (U) = \mathcal{O}_{\mani} (U) \cdot \left \{ d\theta_1, \ldots, d\theta_{m} \, | \, dx_1, \ldots, dx_n \right \},
\eear 
where now the parity is assigned such that $|d\theta_\alpha| = 0$ and $|dx_i| = 1$ for any $\alpha = 1, \ldots, m$ and $i = 1, \ldots, n.$ It is important to stress that now $\Omega^{1}_\mani$ has a (perfect) duality pairing with the \emph{parity shifted tangent sheaf} $\Pi \mathcal{T}_\mani$, not with the tangent sheaf $\mathcal{T}_\mani$. Again $\Pi \mathcal{T}_\mani$ is a locally-free sheaf of $\mathcal{O}_\mani$-modules of rank $m|n$ and if $x_i | \theta_\alpha$ are local coordinates over an open set $U$ as above, one has  
\bear
\Pi \mathcal{T}_{\mani} (U) = \mathcal{O}_{\mani} (U)\cdot \{ \pi \partial_{\theta_1}, \ldots, \pi \partial_{\theta_m} \, | \, \pi \partial_{x_1}, \ldots, \pi \partial_{x_n} \},  
\eear
where $| \pi \partial_{x_i} | = |\partial_{x_i}| + 1 = 1$ and $|\pi \partial_{\theta_\alpha}| = |\partial_{\theta_\alpha} | + 1 = 0$ for any $i= 1, \ldots, n $ and $\alpha = 1, \ldots, m$. With this convention, if $y_a = x_i | \theta_\alpha$ is a local system of coordinates, the duality pairing $\langle \cdot , \cdot \rangle : \Omega^{1}_\mani \otimes \Pi \mathcal{T}_\mani \rightarrow \stsheaf$ reads
\bear \label{obpair}
 \langle dy_a , \pi \partial_{y_b} \rangle = \delta_{ab}. 
\eear
Notice that in what follows we will often simply write $dy_a (\pi \partial_{y_b}) = \delta_{ab}$ as indeed $dy_a$ and $\pi \partial_{y_b}$ are generating sections of $\mathcal{H}om (\Pi \mathcal{T}_\mani, \mathcal{O}_\mani) $ and $\Pi \mathcal{T}_\mani$ respectively on a certain open set. Once given these basic definitions and fixed our conventions, we introduce the following \cite{Manin, Penkov}.
\begin{definition}[de Rham Superalgebra] Let $\mani$ be a real supermanifold of dimension $n|m$ with structure sheaf $\mathcal{O}_\mani$. We call the de Rham algebra of $\mani$ the sheaf of $\mathcal{O}_\mani$-superalgebras given by 
\bear
\Omega^\bullet_{\mani} \defeq \bigoplus_{k = 0}^{\infty} \cat{S}^k_{\mathcal{O}_\mani} \Omega^1_{\mani},
\eear
where $\cat{S}^k : \mathbf{Sh}_{\stsheaf} \rightarrow \mathbf{Sh}_{\stsheaf}$ is the $k$-supersymmetric functor and $\Omega^1_{\mani} $ is the sheaf of $1$-forms of $\mani$ defined as above, where the $\mathbb{Z}_2$-grading is induced by that of $\Omega^1_\mani$. 
\end{definition} 
{\remark Notice that the above algebra is readily made into a $\mathbb{Z}$-graded algebra by assigning $\deg ( \stsheaf ) =0$ and $\deg (\Omega^1_\mani) =1$: this is the obvious degree induced by the (local) polynomial superalgebra 
\bear 
\cat{S}^\bullet \Omega^1_{\mani} (U) = \mathcal{O}_{\mani} (U) \cdot [d\theta_1, \ldots, d\theta_m | dx_1, \ldots, dx_n].
\eear
The sections of the locally-free $\mathcal{O}_\mani$-submodule $\cat{S}^k \Omega^1_\mani$ of degree $k$ are called $k$-\emph{forms}, as in the ordinary commutative setting.}
{\remark It is crucial to observe the difference with respect to the usual de Rham or \emph{exterior} algebra on an ordinary manifold $X$, whose non-zero sections can be of degree $n = \dim X$ at most, due to anti-commutativity of the exterior product. Over a supermanifold, instead, a system of local generator obeys the supercommutation relation of the \emph{super}symmetric algebra, \emph{i.e.} if $y_a = x_i |\theta_\alpha$ then $[dy_a, dy_b]_{\pm} = 0$ \cite{Deligne, Manin}. This leads in particular, for any $\alpha, \beta = 1, \ldots, m$ and $i, j = 1, \ldots, n$ to 
\bear 
[dx_i, dx_j ]_+ \defeq dx_i dx_j + dx_j  dx_i \quad \rightsquigarrow \quad dx_i dx_j = - dx_j dx_i. 
\eear
\bear
[ d\theta_\alpha,  d\theta_\beta ]_- \defeq d\theta_\alpha d\theta_\beta - d\theta_\beta d\theta_\alpha  = 0 \quad \rightsquigarrow \quad d\theta_\alpha d\theta_\beta = d\theta_\beta d\theta_\alpha, 
\eear
where, for notational convenience, we have left the supersymmetric product of two elements in $\Omega^\bullet_{\mani}$ understood. 
While the first of these relations are just the characterizing relations of an exterior algebra over an ordinary manifold, thus stating the anticommutativity of two 1-forms, the second of these relations, instead, implies for example that $d\theta_{\alpha}^n \neq 0$ for any $n \geq 1$ and for any $\alpha = 1, \ldots, m.$ It follows that, provided that the supermanifold has odd dimension greater or equal to 1, there are non-zero forms at \emph{any} degree. \\
\noindent We now introduce the following odd homomorphism acting on the de Rham superalgebra. }
\begin{definition}[de Rham differential] Let $\mani$ be a real supermanifold of dimension $n|m$ with structure sheaf $\mathcal{O}_\mani$. Given a section $\omega$ of the de Rham superalgebra $\Omega^\bullet_\mani$ represented as $\omega =  dy^I \otimes f_I$ in a certain trivilization $y_a = x_i |\theta_\alpha$ for some multi-index $I$ with $f_I \in \mathcal{O}_\mani$, we define the de Rham differential $d : \Omega^\bullet_{\mani} \rightarrow \Omega^\bullet_\mani$ to be the \emph{odd derivation} $d \defeq \sum_{a} dy_a \otimes \partial_{y_a}$. More precisely  we have
\bear
\xymatrix@R=1.5pt{
d : \Omega^\bullet_\mani \ar[r] & \Omega^\bullet_{\mani}  \\
dy^I \otimes f_I \ar@{|->}[r] & d (dy^I \otimes f_I) = \sum_a (-1)^{|y_a| |dy^I| }dy_a dy^I \otimes \partial_{y_a }(f_I).
}
\eear 
\end{definition}
{\remark First of all it is easy to see that the $d$ is well-defined, \emph{i.e.} it is \emph{globally defined} - this follows easily from the transformation properties of the $dy_a$'s and the $\partial_{y_a}$'s -, it is \emph{odd} since $|dy_a| = |\partial_{y_a}| + 1$ and it is a \emph{derivation} on the de Rham superalgebra, \emph{i.e.}\ it satisfies the $\mathbb{Z}_2$-graded Leibniz rule in the form 
\bear
d (\omega \eta) = (d \omega) \omega + (-1)^{|\omega|} \omega (d\eta),
\eear
for $\omega$ and $\eta$ any two forms in the de Rham superalgebra, and where $|\omega|$ is the parity of $\omega$.\\
Let indeed $\omega $ and $\eta$ be represented as $\omega = dy^I \otimes f_I$ and $\eta = dy^J \otimes g_J$, for some multi-indices $I$ and $J$, then we have 
\begin{align}
d (\omega \eta) & = d ((dy^I \otimes f_I) (dy^J \otimes g_J )) = (-1)^{|dy^J| |f_I|} d (dy^I dy^J \otimes f_I g_J ) \nonumber \\
& =  \sum_a  (-1)^{ |dy^J| |f_I| + |y_a| ( |dy^I| + |dy^J|)} dy_a  dy^I  dy^J \otimes \partial_{y_a} (f_I g_J) \nonumber \\
& =  \sum_a  (-1)^{ |dy^J| |f_I| + |y_a| ( |dy^I| + |dy^J|)} dy_a  dy^I  dy^J \otimes  \left ( \partial_{y_a} (f_I) g_J + (-1)^{|y_a| |f_I|} f_I \partial_{y_a} (g_J) \right ) \nonumber \\
& = \sum_a (-1)^{|y_a| | dy_I |} (dy_a  dy^I \otimes \partial_{y_a} (f_I)) (dy^J \otimes g_J) + \nonumber \\
& \quad + \sum_a (-1)^{|x_a| |dy^J| + (|y_a| + |y_a| + 1) (|f_I| + |dy^I|)} (dy^I \otimes f_I ) (dx_a  dy^J \otimes \partial_{y_a} (g_J) )\\
& = \sum_a (-1)^{|y_a| | dy_I |} (dy_a  dy^I \otimes \partial_{y_a} (f_I))  (dy^J \otimes g_J) + \nonumber \\
& \quad + \sum_a (-1)^{|y_a| |dy^J| + |d| |\omega|} (dy^I \otimes f_I)   (dx_a dy^J \otimes \partial_{x_a} (g_J) )\nonumber \\
& = (d\omega) \eta + (-1)^{|\omega|} \omega (d \eta),
\end{align}
where we have used that $\partial_{y_a}$ is a (super)derivation of the structure sheaf $\mathcal{O}_\mani $ and that $|d| = 1.$ We are now in the position to prove the following lemma. }
\begin{lemma} \label{dgsa} The pair $(\Omega^\bullet_\mani, d)$ defines a differential graded superalgebra (DGsA).
\end{lemma}
\begin{proof} We have already shown that $d$ is a derivation of the de Rham superalgebra. We are left to prove that $d $ is nilpotent, \emph{i.e.} $d^2 = 0$. To this end, we simply observe that (up to a constant) 
\begin{align}
d^2 & = \sum_{a,b} (dy_a \otimes \partial_{y_a} )(dy_b \otimes \partial_{y_b}) + \sum_{a,b}(dy_b \otimes \partial_{y_b} ) (dy_a \otimes \partial_{y_a }) \nonumber \\
& = \sum_{a,b} (-1)^{|y_a| |dx_b|} dy_a  dy_b \otimes \partial_{y_a} \partial_{y_b} + \sum_{i,j} (-1)^{|y_b| |dx_a|} (dy_b  dy_a  \otimes \partial_{y_b} \partial_{y_a}) \nonumber \\
& = \sum_{a,b} \left ( (-1)^{|y_a| |y_b| + |y_a|} + (-1)^{|y_a||y_b| + |y_a| + 1} \right ) dy_a  dy_b \otimes \partial_{y_a} \partial_{y_b} \nonumber\\
&= 0,
\end{align}
where we have used that $|dy_a | =|y_a| + 1$ for any $a$ even and odd. \end{proof}
\noindent Notice that, with a slight abuse of notation, we are confusing a $DGsA$ with a sheaf of $DGsA$'s. The previous lemma justifies the following definition.
\begin{definition}[de Rham Complex of $\mani$] We call the differential graded superalgebra $(\Omega^\bullet_\mani, d)$ the de Rham complex of $\mani$.
\bear
\xymatrix{
0 \ar[r]  & \stsheaf \ar[r]^{d  } &  \Omega^1_\mani \ar[r]^{ \; d} & \ldots \ar[r]^{d\; } & \Omega^n_\mani \ar[r]^{ \; d} &\ldots
}
\eear
\end{definition}
{\remark Once again notice that the de Rham complex of a supermanifold is \emph{not} bounded from above ultimately due to the characteristic odd dimensions of the supermanifold itself, which reverberate on the presence of commuting even 1-forms.}\\ 

\noindent We are now interested in proving the main result concerning the cohomology of the de Rham complex on a supermanifold, namely a generalization of the ordinary Poincaré lemma. 
\begin{theorem}[Poincaré Lemma for Differential Forms] \label{PoincLemm} Let $\mani$ be a real supermanifold and let $(\Omega^\bullet_\mani, d)$ be the de Rham complex of $\mani$. Then one has 
\bear
H^k_{d} (\Omega^\bullet_\mani) \cong \left \{ \begin{array}{lll}
\mathbb{R}_\mani & & k = 0\\
0 & & k >0,
\end{array}
\right.
\eear
where $H^k_d (\Omega^\bullet_\mani) $ is the $k$-cohomology sheaf of the de Rham complex of $\mani$ and $\mathbb{R}_\mani$ is the sheaf of locally constant function on $\mani$. In particular: 
\begin{enumerate}[leftmargin=*]
\item The de Rham complex is a (right) resolution of the sheaf $\mathbb{R}_\mani$, 
\bear
\xymatrix{
0 \ar[r] & \mathbb{R}_\mani \ar[r] & \Omega^{\bullet}_\mani.
}
\eear 
\item Any \emph{closed} form is \emph{locally exact} on a real supermanifold. 
\end{enumerate}
\end{theorem}
\begin{proof} We start observing that a form of zero degree is a section $f \in \mathcal{O}_\mani$ and, as in the ordinary setting it is immediate to see that the request $d f = 0 $ forces $f$ to be locally constant, so that one finds indeed $H^0_{d} (\Omega^{\bullet}_\mani) \cong \mathbb{R}_\mani$.\\
We now work locally, and we consider the basic case $\mathpzc{U} = \mathbb{R}^{p|q}$. We let $\omega \in \Omega^k_{\mani}$ for $k \geq 1$ and we show that for any $k \geq 1$ there exists a \emph{homotopy} $h^k : \Omega^{k}_{\mani} \rightarrow \Omega^{k-1}_{\mani}$ for the differential $d$, \emph{i.e.}\ a map such that
\bear
h^{k+1} \circ d^{k} + d^{k-1} \circ h^k = id_{\Omega^k},
\eear 
where we have specified the degree involved for the sake of clarity, and where the maps go as follows \bear
\xymatrix{
\cdots \ar[r] & \Omega_{\mathpzc{U}}^{k-1} \ar[r] & \Omega^{k}_{\mathpzc{U}} \ar@{-->}[d]_{id}\ar[dl]_{h^k} \ar[r]^{d^k} & \Omega^{k+1}_{\mathpzc{U}} \ar[dl]^{h^{k+1}} \ar[r] & \cdots \\
\cdots \ar[r] & \Omega^{k-1}_{\mathpzc{U}} \ar[r]_{d^{k-1}} & \Omega^{k}_{\mathpzc{U}} \ar[r] & \Omega_{\mathpzc{U}}^{k+1} \ar[r] & \cdots. \\
}
\eear
Let us consider the homotopy between the identity and the constant zero map in $\mathpzc{U} = \mathbb{R}^{p|q}$ defined by 
\bear
\xymatrix@R=1.5pt{ 
G : [0,1] \times \mathpzc{U} \ar[r] & \mathpzc{U} \\
(t, y_a \defeq x_1, \ldots, x_p | \theta_1, \ldots, \theta_q) \ar@{|->}[r] & t y_a \defeq (tx_1, \ldots, t x_p | t\theta_1, \ldots t \theta_q).
}
\eear
This induces a map on differential forms via its pull-back, that can be written as a family
of pull-back maps $G^\ast_t : \Omega^k_{\mathpzc{U}} \rightarrow \Omega^k_{\mathpzc{U}}$, indexed by $t$. 
We define the homotopy $h^k$ to be the map
\bear
\Omega^{k}_{\mathpzc{U}} \owns \omega \longmapsto \int^{t=1}_{t=0}dt \left (\iota_{\partial_t} G^\ast_t (\omega) \right ) \in \Omega^{k-1}_{\mathpzc{U}}.
\eear
where $\iota_{\partial_t}$ is the contraction with respect to the vector field $\partial_{t}$ on the interval $[0,1]$.
We have 
\begin{align}
\left (h^{k+1} \circ d^{k} + d^{k-1} \circ h^k \right ) (\omega) & = \int_{0}^1 dt \left ( \iota_{\partial_t} G^\ast_t (d^k \omega) \right ) + \int_0^1 dt \left ( d^{k-1} \iota_{\partial_t} G^\ast_t (\omega) \right ) \nonumber \\
& = \int_{0}^1 dt \left ( \iota_{\partial_t} d^k + d^{k-1} \iota_{\partial_t} \right )G^\ast_t (\omega) \nonumber \\
& = \int_{0}^1 dt \mathcal{L}_{\partial_{t}} G^\ast_t (\omega) \nonumber \\
& = \int^{1}_0 dt \, \partial_t G^\ast_t (\omega) \nonumber \\
& = G^\ast_1 (\omega) - G^\ast_0 (\omega).
\end{align}
We now observe that, by definition, $G^\ast_0 \omega = 0$ and $G^\ast_1 \omega = \omega, $ so that $
\left (h^{k+1} \circ d^{k} + d^{k-1} \circ h^k \right ) (\omega) = \omega$ for any $\omega \in \Omega^{k > 0}_{\mani}.
$
In particular, one has that $H^{k > 0}_{d} (\Omega^{\bullet}_\mathpzc{U}) = 0 $ for $k > 0,$ and the cohomology is concentrated in degree zero. This is the \virgolette standard'' Poincaré lemma for $\mathbb{R}^{p|q}$, the local model for a generic supermanifold $\mani$. Getting back to $\mani$ and $\Omega_{\mani}^\bullet$ it is enough to observe that on a stalk over a point $x \in \mani$ clearly $\mbox{Im}(d^k_x) \subseteq \ker (d^{k+1}_x)$. On the other hand, conversely, if $\omega_x \in \ker (d_x^{k+1})$ with representative $\omega \in \Omega^{k+1}_\mani (U)$, then by Poincaré lemma for $\mathbb{R}^{p|q}$ there exists a $V \subset U$ with $x\in V$ such that $\omega \lfloor_{V} = d\eta$ for some $\eta \in \Omega^{p}_\mani (V)$, therefore $\omega_x = d\eta_x \in \mbox{Im} (d^{p}_x)$.
\end{proof}
{\remark This result says something remarkable, but in some sense predictable: we cannot expect new \emph{topological} invariants arising from the odd part of the geometry of a supermanifold. Instead, the topology is fully encoded in the reduced manifold $\manir$. Intuitively, this is to be ascribed to the characteristic nilpotency of the odd part of the geometry, which is not \virgolette strong enough'' to modify rather rough invariants related to a geometric space, such as the topological ones. In some circles this goes under the slogan \emph{\virgolette fermions are very small''} - whatever it means.}
\begin{definition}[de Rham Cohomology of $\mani$] Let $\mani$ be a real supermanifold. We define the de Rham cohomology of $\mani$ to be the cohomology of the \emph{global sections} of the (sheaf of) differentially graded superalgebras $(\Omega^\bullet_{\mani}, d)$, \emph{i.e.}
\bear
H_{\mathpzc{dR}}^k ( \mani) \defeq H^k_d (\check{H}^0_{\check \delta} (\Omega^{\bullet}_\mani) ), 
\eear
where $\check{H}^0_{\check \delta}(\Omega^\bullet_{\mani})$ is $0$-\v{C}ech cohomology group of $\Omega^{\bullet}_\mani$, \emph{i.e.} the global sections or $\Omega^{\bullet}_\mani$.
\end{definition}
{\remark Note that this is the usual definition of de Rham cohomology of a real manifold upon considering an ordinary manifold instead of a supermanifold $\mani.$ In particular, we have the following easy consequence of the previous Poincaré lemma \ref{PoincLemm}.}
\begin{theorem}[Quasi-Isomorphism I]\label{qiso} Let $\mani$ be a real supermanifold and let $\manir$ be its reduced manifold. Then the de Rham complex of $\mani$ is \emph{quasi-isomorphic} to the de Rham complex of $\manir.$
In particular, one has that 
\bear
H^\bullet_{\mathpzc{dR}} (\mani) \cong H^\bullet_{\mathpzc{dR}} (\manir).
\eear
\end{theorem} 
\begin{proof} The theorem is an easy consequence of the \emph{\v{C}ech-to-de Rham spectral sequence} for the double complex $(\Omega^{\bullet}_{\mani}, \delta, d)$, where $\delta$ is the \v{C}ech differential and $d$ is the de Rham differential, see \cite{BottTu}. More in detail, the \emph{generalized} Mayer-Vietoris short exact sequence (hence the existence of a partition of unity) and Poincaré lemma yield $H^\bullet_{\mathpzc{dR}} (\manir) \cong \check{H}^\bullet (|\manir|, \mathbb{R}_\mani)$ in the ordinary setting and, analogously, $H^\bullet_{\mathpzc{dR}} (\mani) \cong \check{H}^\bullet (|\manir|, \mathbb{R}_\mani )$ in the supergeometric setting respectively, where $\check{H}^\bullet (|\manir|, \mathbb{R}_{\mani})$ is the \v{C}ech cohomology of the sheaf of locally-constant functions $\mathbb{R}_\mani$. It follows that   
\bear
H^\bullet_{\mathpzc{dR}} (\mani) \cong \check{H}^\bullet (|\manir|, \mathbb{R}_\mani) \cong H^\bullet_{\mathpzc{dR}} (\manir),
\eear
thus concluding the proof.\end{proof} 
\noindent In particular, one has a generalization of the ordinary Poincaré Lemma to the \virgolette local model'' real supermanifold $\mathbb{R}^{p|q}$.
\begin{theorem}[Poincaré Lemma for $\mathbb{R}^{p|q}$] The de Rham cohomology of the supermanifold $\mathbb{R}^{p|q}$ is concentrated in degree zero, \emph{i.e.} 
\bear
H_{\mathpzc{dR}}^k (\mathbb{R}^{p|q}) \cong \left \{ \begin{array}{lll}
\mathbb{R} & & k = 0\\
0 & & k >0.
\end{array}
\right.
\eear
\end{theorem}
\begin{proof} Follows immediately from the above \ref{PoincLemm} and \ref{qiso}.
\end{proof}
{\remark Analogous results hold true for the \emph{compactly suppported} de Rham cohomology on a supermanifold. More in detail one finds that 
\bear
H^\bullet_{\mathpzc{dR}, \mathpzc{c}} (\mani) \cong H^\bullet_{\mathpzc{dR}, \mathpzc{c}} (\manir),
\eear 
so that in particular, one gets the compactly supported Poincaré lemma for $\mathbb{R}^{p|q}$, which reads
\bear
H_{\mathpzc{dR}, \mathpzc{c}}^k (\mathbb{R}^{p|q}) \cong \left \{ \begin{array}{lll}
\mathbb{R} & & k = p\\
0 & & k \neq p.
\end{array}
\right.
\eear
As can be imagined, a representative is given by the lift of the top form on $\manir$ on the supermanifold, \emph{i.e.} $dz_1  \ldots  dz_p \mathpzc{B}_{\mathpzc{c}} (z_1, \ldots z_p)$, where $\mathpzc{B}_{\mathpzc{c}}$ is a bump function.}
{\remark \label{cohomorem} The above theorem \ref{qiso} guarantees that even if the de Rham complex of a supermanifold $\mani$ of dimension $n|m$ is not bounded from above, its de Rham cohomology groups $H^k_{\mathpzc{dR}} (\mani)$ can only be non-zero up to degree $n$, \emph{i.e.} up to the degree which equals the even dimension of the supermanifold or analogously the dimension of its reduced manifold $\manir$. All the other higher degrees do not contribute to the de Rham cohomology or $\mani$. In other words, the de Rham cohomology of the supermanifold can be non-zero only in the framed part of the complex 
\bear 
\xymatrix{
0 \ar[r] & \Omega^0_\mani (\mani)  \ar[r] & \Omega^1_{\mani} (\mani) \ar[r] & \ldots \ar[r] & \Omega^{n-1}_\mani (\mani) \ar[r] & \Omega^n_{\mani} (\mani) \ar[r] & \Omega^{n+1}_\mani (\mani) \ar[r] & \ldots
\save "1,2"."1,6"*[F]\frm{}
\restore 
}
\eear
where we have denoted $\Omega^k_{\mani} (\mani)$ the global sections of the sheaf $\Omega^{k}_{\mani}$ for any $k$.
\remark Finally, it is to be noted that the fact that the de Rham complex of a supermanifold is not bounded from above implies - among things - that there is \emph{no} tensor density playing the role of the top exterior bundle $\Omega^{\dim X}_{X}$ over an ordinary manifold $X$. This is the starting point of the quite unique integration theory on supermanifold, whose main character is introduced in the following section. }

\section{The Berezinian Sheaf: Constructions and Geometry}

\noindent In this section we introduce one of the crucial and most peculiar notions arising in supergeometry, that of \emph{Berezinian sheaf} of a supermanifold, which we will denote as $\mathcal{B}er (\mani)$. An \virgolette operative'' first definition can be given by characterizing $\mathcal{B}er (\mani)$ in terms of its transition functions. 
\begin{definition}[Berezinian - via Transition Functions] Let $\mani$ be a real or complex supermanifold of dimension $n|m$ and let $ \{ ( U_i , x_{U_i} | \theta_{U_i} )\}_{i \in I}$ be an atlas of charts covering $\mani$. Then we define the Berezinian sheaf $\mathcal{B}er (\mani)$ of $\mani$ as the locally-free sheaf of right $\mathcal{O}_\mani$-modules of rank $\delta_{0, n+m} | \delta_{0, n+m +1}$, whose local generators $\{ \mathcal{D}_{U_i} (x_{U_i}, \theta_{U_i}) \}_{i \in I}$ transforms as
\bear
\mathcal{D}_{U_j \lfloor_{U_i \cap U_j}} (x_{U_j} | \theta_{U_j}) =  \mathcal{D}_{U_i \lfloor_{U_i \cap U_j}} (x_{U_i} |\theta_{U_i})  Ber ( \mathcal{J}ac (\varphi_{ij} ))
\eear 
with 
\bear \label{transf}
Ber ( \mathcal{J}ac (\varphi_{ij} )) = \det (A - BD^{-1}C) \det (D)^{-1}, 
\eear
where $\mathcal{J}ac (\varphi_{ij})$ is the super Jacobian of the change of coordinates 
\bear
\xymatrix@R=1.5pt{
\varphi_{ij} : U_i \lfloor_{U_i \cap U_j} \ar[r] &  U_j \lfloor_{U_i \cap U_j} \nonumber \\
x_{U_i} | \theta_{U_i} \ar@{|->}[r] &  \varphi_{ij, 0} (x | \theta) = x_{U_j} | \,\varphi_{ij, 1} (x | \theta) = \theta_{U_j}, 
}
\eear
and where we have posed 
\bear
\left ( \begin{array}{c|c}
A & B \\
\hline
C & D
\end{array}
\right )
\defeq 
 \left ( \begin{array}{c|c}
\partial_x \varphi_{ij, 0} & \partial_\theta \varphi_{ij, 0} \\
\hline
\partial_x \varphi_{ij, 1} & \partial_{\theta} \varphi_{ij, 1}
\end{array}
\right ) = \mathcal{J}ac (\varphi_{ij}). 
\eear
\end{definition}
\noindent Besides its importance in relation to the quite unique and highly non-trivial integration theory on supermanifolds \cite{Voronov, Witten}, the Berezinian sheaf deserves special attention on its own. In what follows we will present and review three of its constructions, all of them inspired by algebraic-geometric methods. \\
The first one, due to Hern\'andez Ruip\'erez and Mu\~{n}oz Masque \cite{Ruiperez}, is a beautiful realization of the Berezinian sheaf of a real smooth supermanifold as a certain quotient of natural sheaves, which has the merit of being relatively easy and to make apparent its relation with integration theory. \\
The second construction that we will present has the complex analytic or algebraic category in sight: in this context the Berezinian sheaf emerges from the cohomology of a supergeometric generalization of the ordinary Koszul complex. As mentioned in the introduction, the original idea is due to Ogievetsky and Penkov, and it appeared first in \cite{OP}. Later on, building upon \cite{OP}, the super Koszul complex has been briefly sketched by Manin in \cite{Manin}. Quite recently a self-contained and encompassing discussion of the topic has been given by the author and Re in \cite{NojaRe}.   

\subsection{Berezinian Sheaf as a Quotient Sheaf} Working on a real smooth supermanifold of dimension $n|m$, the construction of the Berezinian sheaf of $\mani$ as a quotient sheaf is obtained starting from the sheaf $\Omega^n_{\mani} \otimes_{\stsheaf} \mathcal{D}^{(m)}_\mani $ of differential operators $\mathcal{D}^{(m)}_\mani$ of degree $m$ on $\stsheaf $ taking values in (the sheaf of) differential forms $\Omega^n_{\mani}$ of degree $n$. We denote this (locally-free) sheaf of $\mathcal{O}_\mani$-modules as $  
\mathcal{D}^{(m)}_{\mani} (\Omega^n_{\mani}) 
$ for short and its elements will be written as $\omega \otimes F \in   
\mathcal{D}^{(m)}_{\mani} (\Omega^n_{\mani}).$ \\
Notice that, locally over an open set $U$ with coordinates $y_a = x_i | \theta_\alpha$, for $i = 1, \ldots, n$ and $\alpha = 1, \ldots, m$ a system of generators over $\mathcal{O}_\mani$ is given by
\bear  \label{GenFormOp}
\mathcal{D}^{(m)}_{\mani} (\Omega^n_{\mani}) (U) = \Bigg \{ \underbrace{dy_{j_1} \ldots dy_{j_n}}_{n} \otimes \underbrace{\frac{\partial}{\partial{y_{k_1}}} \ldots \frac{\partial}{\partial{y_{k_m}}}}_{m} \Bigg \} \cdot \stsheaf (U) \defeq \bigg \{ dy_{I_{j}} \otimes \frac{\partial}{\partial y_{I_k}} \bigg \} \cdot \stsheaf (U)
\eear
where $I_j$ and $I_k$ are two multi-indices such that $|I_j| = n$ and $|I_k| = m$ and where $\mathcal{D}^{(m)}_{\mani} (\Omega^n_{\mani})$ has been considered with the structure of right $\mathcal{O}_\mani$-module.\\
The key object in the construction is a pretty subtle sub-sheaf of $\mathcal{D}^{(m)}_\mani (\Omega^n_{\mani})$. If $\iota : \manir \rightarrow \mani$ is the embedding of the reduced manifold into the supermanifold $\mani$ as in \eqref{splittingseq}, we have a corresponding pull-back map $\iota^\ast : \Omega^\bullet_{\mani} \rightarrow \Omega^\bullet_{\manir}$ on differential forms. We introduce a sheaf of $\mathcal{O}_\mani$-modules $\mathcal{K}_{\mani}$ as the sub-sheaf of sections of  $ \mathcal{D}^{(m)}_{\mani} (\Omega^{n}_{\mani})$ having the following property  
\bear
\mathcal{K}_{\mani} (U) \defeq \left \{\forall f \in \mathcal{O}_{\mani, \mathpzc{c}}(U),\;  \omega \otimes F \in \mathcal{D}^{(m)}_{\mani} (\Omega^n_{\mani}) (U) \mbox{ if } \exists \, \eta \in \Omega^{n-1}_{\manir , \mathpzc{c}}(U): \iota^\ast (\omega \otimes F(f) ) = d\eta \; \;  \right \}, 
\eear
where $U$ is an open set and $f \in \mathcal{O}_{\mani, \mathpzc{c}}(U) \defeq \Gamma_{\mathpzc{c}} (U, \mathcal{O}_\mani )$ and $\eta \in \Omega^{n-1}_{\manir, \mathpzc{c}}(U) \defeq \Gamma_\mathpzc{c} (U, \Omega^{n-1}_{\manir} )$ are a compactly supported function and a compactly supported form on $U$ respectively. 
Then one can prove the following theorem, see \cite{Ruiperez}.
\begin{theorem}[Berezinian as a Quotient] Let $\mani$ be a real supermanifold of dimension $n|m$ and let $\mathcal{D}^{(m)}_{\mani} (\Omega^n_{\mani}) $ and $\mathcal{K}_{\mani}$ be defined as above. Then the sheafification of the quotient pre-sheaf $\mathcal{D}^{(m)}_{\mani} (\Omega^n_{\mani}) / \mathcal{K}_{\mani} $ is a locally-free sheaf of (right) $\stsheaf$-module of rank $\delta_{0, n+m} | \delta_{0, n+m +1}$, whose generator reads on an open set $U$ with coordinates $x_i | \theta_\alpha$ for $i = 1,\ldots, n$ and $\alpha = 1, \ldots, m$
\bear
\slantone{\mathcal{D}^{(m)}_{\mani} (\Omega^n_{\mani}) (U)}{\mathcal{K}_{\mani} (U)} \cong \left [ dx_1 \ldots dx_n \otimes \frac{\partial}{\partial \theta_1 } \ldots \frac{\partial}{\partial \theta_m} \right ] \cdot \mathcal{O}_{\mani} (U),
\eear  
where the square bracket stays for the class of the form-valued differential operator modulo $\mathcal{K}_\mani.$\\
In particular, the above quotient is naturally isomorphic to the Berezinian sheaf of the supermanifold, \emph{i.e.}
\bear
\mathcal{B}er (\mani) \cong \slantone{\mathcal{D}^{(m)}_{\mani} (\Omega^n_{\mani})}{\mathcal{K}_{\mani}}.
\eear
\end{theorem}
\begin{proof} We consider $\omega \otimes F \in \mathcal{D}^{(m)}_{\mani} (\Omega^n_{\mani})$ and we work in a coordinate system $x_i | \theta_\alpha $ over an open set $U$, such that $\mathcal{D}^{(m)}_{\mani} (\Omega^n_{\mani})$ has a basis as above in \eqref{GenFormOp}. Let us consider the following instances.
If it appears a term of the form $d\theta_{\alpha}^n$ for any $\alpha $ and any $n \geq 1$ in $\omega$, then the corresponding $\omega \otimes F$ goes to zero under $\iota^\ast$, and as such it is in $\mathcal{K}_\mani$. So this forces $\omega$ to be of the kind $dx_1\ldots dx_n $.
Now consider $\omega \otimes F$ to be of the kind $\omega \otimes F = dx_1 \ldots dx_n \otimes \partial_{x_I} \partial_{\theta_J}$ for some multi-indices $I$ and $J$ such that $|I| + |J| = m$. If $I \neq 0$ then $\omega \otimes F \in \mathcal{K}_{\mani} (U)$, indeed consider for example $dx_1 \ldots dx_n \otimes \partial_{\theta_2} \ldots \partial_{\theta_{m}} \partial_{x_1}$: the crucial case is that of $f$ of the form $f (x| \theta) = g_{\mathpzc{c}} (x) \theta_2 \ldots \theta_m$, with $g_{\mathpzc{c}}$ a compact supported function on $U$ to get $\omega \otimes F(f) = dx_1 \ldots dx_n \partial_{x_1} g_{\mathpzc{c}} (x) = d ( dx_2 \ldots dx_n g_{\mathpzc{c}}(x) )$, which implies $\omega \otimes F \in \mathcal{K}_{\mani}.$
This is enough to prove that the class $\mathcal{D} (x| \theta) \defeq [dx_1 \ldots dx_n \otimes \partial_{\theta_1} \ldots \partial_{\theta_m}]$ defines a generator for the above quotient sheaf. Also, this can indeed be identified with the Berezinian sheaf, upon checking that the class $\mathcal{D} (x |\theta)$ transform indeed as in \eqref{transf} under a change of local coordinates. This is a local check: a careful computation is postponed to the next subsection, in a slightly different context.
\end{proof}
{\remark As said above, the previous intrinsic construction of the Berezinian as the sheafification of a quotient pre-sheaf of differential operators valued into differential forms has the unquestionable merit of being relatively easy and, at the same time, as explained at the end of the second section of \cite{Ruiperez}, it makes the relationship with the Berezinian sheaf and the related integration theory apparent. A minor drawback of the construction is that it only holds true for real supermanifolds, as the existence of compactly supported functions is crucial to the above proof.}

\subsection{Berezinian Sheaf from Koszul Complex} Homological algebra comes in help to provide an intrinsic construction of the Berezinian sheaf on complex or algebraic supermanifolds, where compactly supported functions are not available and therefore the above quotient construction breaks down. As hinted in the introduction to this section, the idea of a suitable generalization to a supergeometric setting of the Koszul complex originally appeared in \cite{OP} and was subsequently in \cite{Manin}. Very recently an encompassing construction has been given in \cite{NojaRe}. \\
We let $\mathcal{E}$ be any locally-free sheaf of rank $p|q$ on a supermanifold $\mani$. We define
\begin{align}
& \mathcal{R} \defeq \bigoplus_{k \geq 0 } \mathcal{R}^{\mathcal{E}}_k \quad \mbox{with} \quad \mathcal{R}_{k}^{\mathcal{E}} \defeq \cat{S}^k \mathcal{E} \\
& \mathcal{R}^\Pi \defeq \bigoplus_{k\geq 0} \mathcal{R}_k^{\Pi \mathcal{E}} \quad \mbox{with} \quad \mathcal{R}_k^{\Pi\mathcal{E}} \defeq \cat{S}^k \Pi \mathcal{E},
\end{align}
and in turn we consider the following sheaf of $\stsheaf$-superalgebras given by the tensor product 
\bear
\mathcal{K}^{\mathcal{E}}_\bullet \defeq \bigoplus_{k \geq 0 } \mathcal{K}_{-k} = \mathcal{R} \otimes_{\stsheaf} \bigoplus_{k \geq 0 } \mathcal{R}^{\Pi}_k  = \mathcal{R} \otimes_{\stsheaf} \mathcal{R}^\Pi.
\eear
Further, let us consider two bases of local generators for $\mathcal{E}$ and $\Pi \mathcal{E}$ respectively given by $\{v_i | \chi_\alpha \}$ and a $\{ \pi \chi_\alpha | \pi v_i \}$. Using these we can define the following operator acting on $\mathcal{K}^\mathcal{E}_\bullet$:  
\bear \label{operatordelta}
\xymatrix@R=1.5pt{
\delta : \mathcal{K}^\mathcal{E}_\bullet = \mathcal{R} \otimes_{\stsheaf} \mathcal{R}^\Pi \ar[r] & \mathcal{K}^{\mathcal{E}}_\bullet = \mathcal{R} \otimes_{\stsheaf} \mathcal{R}^\Pi \\
\mathpzc{r} \otimes \mathpzc{r}^\Pi \ar@{|->}[r] & \delta (\mathpzc{r} \otimes \mathpzc{r}^\pi ) \defeq \left ( \sum_{i=1}^p v_i \otimes \partial_{\pi v_i} + \sum_{j=1}^q \chi_j \otimes \partial_{\pi \chi_j} \right ) (\mathpzc{r} \otimes \mathpzc{r}^\Pi), 
}
\eear
where $\mathpzc{r} \in \mathcal{R}$ and $\mathpzc{r}^\Pi \in \mathcal{R}^\Pi$. Since the derivations can be seen as the dual bases to the bases of $\mathcal{E}$ and $\Pi \mathcal{E}$ it is not hard to see that the above is globally well-defined and independent of the choice of local bases. Further, $\delta$ is homogenous of degree $-1$ with respect to the $\mathbb{Z}$-gradation of $\mathcal{K}_{\bullet}$ seen as a sheaf of $\mathcal{R}$-modules and it is \emph{odd} with respect to the $\mathbb{Z}_2$-graded structure on $\mathcal{K}_{\bullet}^\mathcal{E}$: more in particular it is not hard to show that it is nilpotent, \emph{i.e.} $\delta \circ \delta = 0 $, so that the pair $(\mathcal{K}_\bullet^{\mathcal{E}}, \delta)$ defines a differentially graded sheaf of $\mathcal{R}$-algebras. We can thus give the following definition.
\begin{definition}[Super Koszul Complex] Let $\mani$ be a real, complex or algebraic supermanifold. Given any locally-free sheaf of $\mathcal{O}_\mani$-modules $\mathcal{E}$, we call the pair $(\mathcal{K}^{\mathcal{E}}_\bullet, \delta)$ defined as above the \emph{super Koszul complex} associated to $\mathcal{E}$: 
\bear
\xymatrix{
\cdots \ar[r]^{ \delta \qquad } & \mathcal{R} \otimes \cat{S}^k \Pi \mathcal{V} \ar[r]^{\qquad \delta} & \cdots \ar[r]^{\delta \qquad } & \mathcal{R} \otimes \cat{S}^2 \Pi \mathcal{E} \ar[r]^{\; \delta} & \mathcal{R} \otimes \Pi \mathcal{E} \ar[r]^{\quad \delta} & \mathcal{R} \ar[r] & 0.  
}
\eear
\end{definition}
\noindent One of the main results of \cite{NojaRe} is concerned with the homology of this complex of sheaves. 
\begin{theorem}[Homology of the Super Koszul Complex] \label{homologykos} Let $\mani$ be a real, complex or algebraic supermanifold with structure sheaf $\mathcal{O}_\mani$, let $\mathcal{E}$ be a locally-free sheaf of $\mathcal{O}_\mani$-modules on $\mani$ and let $(\mathcal{K}_\bullet^{\mathcal{E}}, \delta)$ be the super Koszul complex associated to $\mathcal{E},$ defined as above. Then the super Koszul complex $(\mathcal{K}^\mathcal{E}_\bullet, \delta)$ is an exact resolution of $\mathcal{O}_\mani$ endowed with the structure of sheaf of $\mathcal{R}$-modules, \emph{i.e.}\ the homology $\mathcal{H}_i (\mathcal{K}^\mathcal{E}_\bullet, \delta)$ is concentrated in degree 0,
\bear
\mathcal{H}_i \left ( (\mathcal{K}^{\mathcal{E}}_\bullet, \delta )  \right ) \cong \left \{ \begin{array}{ccc}
\mathcal{O}_\mani   & &  i = 0\\
0  & & i \neq 0.
\end{array}
\right.
\eear
\end{theorem}
{\remark With reference to the above result, notice that $\mathcal{O}_\mani$ is indeed a $\mathcal{R}$-module thanks to the following short exact sequence of sheaves
\bear
\xymatrix{
0 \ar[r] & \mathcal{I}_{\mathcal{R}} \ar[r] & \mathcal{R} \ar[r] & \mathcal{O}_\mani \ar[r] & 0, 
}
\eear
where $\mathcal{I}_\mathcal{R} \defeq \bigoplus_{k\geq 1} \cat{S}^k \mathcal{E}$ is the sheaf of ideals of $\mathcal{R}$ generated by $\mathcal{E} \subset R$ so that $\mathcal{O}_{\mani} \cong \slantone{\mathcal{R}}{\mathcal{I}_\mathcal{R} \mathcal{R}}.$}

{\remark \noindent The \emph{Koszul resolution} of theorem \ref{homologykos} allows us to compute other derived functors in a supergeometric context. In particular, given the Koszul super complex $\mathcal{K}^{\mathcal{E}}_{\bullet}$ as above, one can introduce the \emph{dual} construction via the functor $\mathcal{H}om_{\mathcal{R}}( - , \mathcal{R})$, which yields the pair $(\mathcal{K}_\bullet^{\mathcal{E} \ast }, \delta^\ast ) \defeq (\mathcal{H}om_\mathcal{R} (\mathcal{K}^\mathcal{E}_\bullet, \mathcal{R}), \mathcal{H}om_\mathcal{R} (\delta, \mathcal{R}))$. Defining 
\bear
\mathcal{R}^{\Pi \ast}_k \defeq \bigoplus_{k \geq 0 } \mathcal{R}^{\pi \ast}_i \quad \mbox{with} \quad \mathcal{R}^{\Pi \ast}_k \defeq \cat{S}^k \Pi \mathcal{E}^\ast
\eear
the sheaf of $\mathcal{O}_{\mani}$-superalgebras generated by $\{ \partial_{\pi \chi_\alpha} |\partial_{\pi v_i} \}$ for $i = 1, \ldots, p$ and $\alpha = 1, \ldots, q$ in $\Pi \mathcal{E}^\ast$, it is easy to see that  
\bear
\mathcal{K}^{\mathcal{E} \ast}_{\bullet} \defeq \bigoplus_{k \geq 0} \mathcal{K}_{k}^{\mathcal{E}\ast} = \mathcal{R} \otimes_{\stsheaf} \bigoplus_{k \geq 0} \mathcal{R}_k^{\Pi \ast} = \mathcal{R} \otimes_{\stsheaf} \mathcal{R}^{\Pi \ast}.
\eear
The fundamental observation is that $\mathcal{K}^{\mathcal{E} \ast}_{\bullet} $ is acted by an operator $\delta^\ast$, whose definition is \emph{formally identical} to that of $\delta $ given above in \eqref{operatordelta}. But here $\delta^\ast$ has to be looked at as a \emph{multiplication operator} by the element $ \sum_{i=1}^p v_i \otimes \partial_{\pi v_i} + \sum_{j=1}^q \chi_j \otimes \partial_{\pi \chi_j} $ in $\mathcal{K}^{\mathcal{E}\ast}_\bullet$. This is enough to guarantee that $\delta^{\ast} \circ \delta^{\ast} = 0,$ as the element corresponding to $\delta^\ast$ is odd, so that we can introduce the following. }
\begin{definition}[Dual of the Super Koszul Complex] Let $\mani$ be a real, complex or algebraic supermanifold. Given any locally-free sheaf of $\mathcal{O}_\mani$-modules $\mathcal{E}$, we call the pair $(\mathcal{K}^{\mathcal{E}\ast}_\bullet, \delta^\ast)$ defined as above the \emph{dual of the super Koszul complex} associated to $\mathcal{E}$: 
\bear
\xymatrix{ 
0 \ar[r] & \mathcal{R} \ar[r]^{\delta^\ast \quad} & \mathcal{R} \otimes \Pi \mathcal{E}^\ast \ar[r]^{\delta^\ast } & \mathcal{R} \otimes \cat{S}^2 \Pi \mathcal{E}^\ast \ar[r]^{\qquad  \delta^\ast} & \ldots \ar[r]^{\delta^\ast\quad \; \;} & \mathcal{R} \otimes \cat{S}^k \Pi \mathcal{E}^\ast \ar[r]^{ \qquad \delta^\ast} & \ldots 
}
\eear
\end{definition}
\noindent Now, we aim at computing the (co)homology of the dual of the super Koszul complex. \\
Recalling that by definition we have $\mathcal{K}_{\bullet}^{\mathcal{E}\ast} \defeq \mathcal{H}om_{\mathcal{E}} (\mathcal{K}_{\bullet}^\mathcal{E}, \mathcal{R}), $ then
\bear
\mathcal{E}xt^{i}_{\mathcal{R}} (\mathcal{O}_\mani, \mathcal{R}) = 
 \mathcal{H}^i \left ( ( \mathcal{K}^{ \mathcal{E} \ast}_\bullet, \delta^\ast ) \right ).
\eear
The homology (sheaf) of the dual of the Koszul complex is computed in the following theorem, see \cite{NojaRe}.
\begin{theorem}[Homology of the dual of the Super Koszul Complex] \label{dualhomology} Let $\mani$ be a real, complex or algebraic supermanifold with structure sheaf $\mathcal{O}_\mani$, let $\mathcal{E}$ be a locally-free sheaf of $\mathcal{O}_\mani$-modules on $\mani$ and let $(\mathcal{K}_\bullet^{\mathcal{E}}, \delta)$ be the dual of the super Koszul complex associated to $\mathcal{E},$ defined as above. Then its homology is concentrated in degree $p$ and locally-generated over $\mathcal{O}_\mani$ by the class
\bear
\mathcal{E}xt^p_\mathcal{R} (\stsheaf, \mathcal{R}) \cong 
[  \chi_1 \ldots \chi_q \otimes \partial_{\pi v_1 } \ldots \partial_{\pi v_p} ] \cdot \mathcal{O}_\mani 
\eear
where $\chi_1 \ldots \chi_q \in \cat{\emph{S}}^q \mathcal{E}$ and $\partial_{\pi v_1} \ldots \partial_{\pi v_p} \in \cat{\emph{S}}^p \Pi \mathcal{E}^\ast.$
\end{theorem}
\begin{proof} It is immediate to observe that the element $\mathcal{D} \defeq \chi_1 \ldots \chi_q \otimes \partial_{\pi v_1 } \ldots \partial_{\pi v_p} \in \cat{S}^q \mathcal{E} \otimes \cat{S}^p \Pi \mathcal{E}^\ast$ belong to the kernel of $\delta^\ast.$ We can then observe that locally $
\mathcal{R} \otimes_{\mathcal{O}_\mani} \mathcal{R}^{\Pi \ast} $ is generated over $\mathcal{O}_\mani$ by the elements $ ( v_1, \ldots v_p, \partial_{\pi \chi_1}, \ldots, \partial_{\pi \chi_q} | \partial_{\pi v_1}, \ldots, \partial_{\pi v_p} , \chi_1, \ldots, \chi_q )$. Posing $N \defeq p + q$, we can redefine the generators as  
$ (s_1, \ldots, s_N) \defeq \left ( v_1, \ldots v_p, \partial_{\pi \chi_1}, \ldots, \partial_{\pi \chi_q} \right )$ and $ (\psi_1, \ldots, \psi_N) \defeq \left (  \partial_{\pi v_1}, \ldots, \partial_{\pi v_p} , \chi_1, \ldots, \chi_q \right )$, so that in particular one has that $ \delta^\ast =  \sum_{i=1}^N u_i \psi_i $ and $\mathcal{D} = \prod_{j=1}^N \psi_i $ and $\mathcal{R} \otimes \mathcal{R}^{\Pi \ast}$ becomes a sheaf of exterior algebras 
\bear
\left ( \mathcal{R} \otimes \mathcal{R}^{\Pi \ast} \right ) (U) \cong \bigwedge^\bullet_{\mathcal{O}(U)[s_1, \ldots, s_N]} (\psi_1, \ldots, \psi_N)
\eear
over the commutative ring $\mathcal{O} (U)[s_1, \ldots, s_n]$. This is the dual of the ordinary commutative Koszul complex, whose cohomology is concentrated in top-degree, in this case $N$, and generated over $\mathcal{O}_\mani$ by the element $\mathcal{D} =  \psi_1 \ldots \psi_N,$ see \cite{Eisenbud}. 
\end{proof}
{\remark The crucial point is now that the class singled out by the dual of the Koszul complex transforms by the multiplication by the Berezinian of an automorphism of the sheaf $\mathcal{E}$. More precisely, we prove the following result.}
\begin{theorem} Let $\mani$ be a real, complex, or algebraic supermanifold with structure sheaf $\mathcal{O}_\mani$, let $\mathcal{E}$ be a locally-free sheaf of $\mathcal{O}_\mani$-modules on $\mani$ and let $\varphi \in \mathcal{A}ut (\mathcal{E})$ be an automorphism of $\mathcal{E}$. Then the induced automorphism $\widehat \varphi \in \mathcal{A}ut (\mathcal{E}xt^p_\mathcal{R} (\stsheaf, \mathcal{R}))$ is given by the multiplication by the {inverse} of the Berezinian of the automorphism, \emph{i.e.}
\bear
\xymatrix@R=1.5pt{
\widehat{\varphi} : \mathcal{E}xt^p_\mathcal{R} (\stsheaf, \mathcal{R}) \ar[r] & \mathcal{E}xt^p_\mathcal{R} (\stsheaf, \mathcal{R}) \\
\mathcal{D} \ar@{|->}[r] & \mathcal{D} \cdot Ber (\varphi)^{-1} 
}
\eear
\end{theorem}
\begin{proof} Fixing a local system of generator for $\mathcal{E}$ given by $\{v_1, \ldots, v_p | \chi_1, \ldots, \chi_q \}$, then $\varphi \in Aut (\mathcal{E})$ is represented by a matrix $[M] \in GL(p|q, \mathcal{O}_{\mani} (U))$
\bear
[M (\varphi)]_{\alpha \beta} = \left ( 
\begin{array}{c|c}
A & B \\
\hline 
C & D
\end{array}
\right ) = \left ( 
\begin{array}{c|c}
a_{hi} & b_{hj} \\
\hline c_{ki} & d_{kj}
\end{array}
\right )
\eear
with $A, D$ even and $B, C$ odd submatrices. 
Now, if $\varphi_i$ for $i=1,2$ are automorphisms of $\mathcal{E}$, contravariant functoriality of the construction, see \cite{NojaRe}, implies that the product of two matrices $M (\varphi_1) \cdot M (\varphi_2) $ induces the product of two elements, we call them $ \widehat{\varphi}_{M (\varphi_2)} \cdot \widehat{\varphi}_{M (\varphi_1)} $, which acts as automorphisms of $\mathcal{E}xt^p_{\mathcal{R}} (\stsheaf,\mathcal{R}).$ It follows that we can use the standard decomposition
\bear
\left (
\begin{array}{c|c}
A & B \\
\hline 
C & D 
\end{array} 
\right )= \left ( 
\begin{array}{c|c}
1 & BD^{-1} \\
\hline 
0 & 1
\end{array}
\right )
\left ( 
\begin{array}{c|c}
A - BD^{-1}C & 0 \\
\hline 
0 & D
\end{array}
\right )
\left ( 
\begin{array}{c|c}
1 & 0 \\
\hline 
D^{-1}C & 1
\end{array}
\right )
\eear
to identify the action of the automorphisms $\mathcal{E}xt^p_{\mathcal{R}} (\stsheaf,\mathcal{R}) $ in terms of $M (\varphi)$. 
We consider separately the following cases:
\bear
(1): \;  M (\varphi) = \left (
\begin{array}{c|c}
A & 0 \\
\hline 
0 & D 
\end{array} 
\right ), \qquad
(2): \;  M (\varphi) = \left (
\begin{array}{c|c}
1 & 0 \\
\hline 
\ast & 1 
\end{array} 
\right ), \qquad
(3):  \;  M (\varphi ) = \left (
\begin{array}{c|c}
1 & * \\
\hline 
0 & 1 
\end{array} 
\right ).
\eear 
It is easy to see that only in the first case we have an induced transformation of the homology class $\mathcal{D} = [\chi_1 \ldots \chi_q \otimes \partial_{\pi v_1} \ldots \partial_{\pi v_p}]$ given by the multiplication by $ \det{(D)} \cdot \det{(A)}^{-1}$. In the two remaining cases the homology class $\mathcal{D}$ is invariant. It follows from the above decomposition that $\widehat \varphi_{M(\varphi)} = \det (D) \det (A- BD^{-1}C)^{-1}$, which is indeed $Ber([\varphi])^{-1}$ as claimed. 
\end{proof}
{\remark Notice that what plays a crucial role in the previous proof is the fact that \emph{$\mathcal{D}$ is a cohomology class}, so that those parts of the transformation on the representative which yield exact elements do not contribute to the induced automorphism. }
{\remark Clearly, it is enough to consider $\mathcal{E}^\ast \defeq \mathcal{H}om (\mathcal{E}, \mathcal{O}_{\mani})$ instead of $\mathcal{E}$ and the related homology of the dual of the super Koszul complex as to obtain the expected transformation
\bear
\xymatrix@R=1.5pt{
\widehat \varphi_{\mathcal{E}^\ast} : \mathcal{E}xt^p_{\scriptsize{\cat{S}}^\bullet \mathcal{E}^\ast} (\stsheaf, \cat{S}^\bullet \mathcal{E}^\ast) \ar[r] & \mathcal{E}xt^p_{\scriptsize{\cat{S}}^\bullet \mathcal{E}^\ast} (\stsheaf, \cat{S}^\bullet \mathcal{E}^\ast) \\
\mathcal{D}^\ast \ar@{|->}[r] & \mathcal{D}^\ast \cdot Ber (\varphi_\mathcal{E})  
}
\eear
where $\varphi_{\mathcal{E}}$ is an automorphism of $\mathcal{E}$. This suggests that given a locally-free sheaf $\mathcal{E}$ on a real, complex, or algebraic supermanifold, one can see the previous construction as a \emph{defining} one for the notion of \emph{Berezinian sheaf of $\mathcal{E}$}, by posing 
\bear
\mathcal{B}er(\mathcal{E}) \defeq \mathcal{E}xt^p_{\scriptsize{\cat{S}}^\bullet \mathcal{E}^\ast} (\stsheaf, \cat{S}^\bullet \mathcal{E}^\ast).
\eear
In particular, to make contact with the previous subsection, we give the following definition which agrees also with that of Manin \cite{Manin} and Witten \cite{Witten}.}
\begin{definition}[Berezinian of a Supermanifold] Let $\mani$ be a real, complex, or algebraic supermanifold with structure sheaf $\mathcal{O}_\mani$. We call the Berezinian sheaf of $\mani$ and we denote it by $\mathcal{B}er(\mani)$ the locally-free sheaf of $\mathcal{O}_\mani$-modules of rank $\delta_{0, n+m} | \delta_{1, n+m} $ defined by 
\bear
\mathcal{B}er (\mani) \defeq \mathcal{B}er ( \Omega^{1}_\mani), 
\eear
where $\mathcal{B}er (\Omega^{1}_\mani ) =  \mathcal{E}xt^p_{\scriptsize{\cat{S}}^\bullet (\Omega^1_{\mani})^\ast} (\stsheaf, \cat{S}^\bullet (\Omega^1_{\mani})^\ast)$. 
\end{definition}
{\remark Notice that, as above, the Berezinian sheaf of $\mani$ is locally generated by the class $[dx_1 \ldots dx_p \otimes \partial_{\theta_1} \ldots \partial_{\theta_q}]$ and that it can be equivalently defined as $\mathcal{B}er (\mani) \defeq \mathcal{B}er (\mathcal{T}_\mani)$, since if $A$ is an automorphism, then $Ber (A^{st} ) = Ber (A)$, $Ber (A^{-1}) = Ber (A)^{-1}$ and $Ber (\Pi A) = Ber (A)^{-1}$. }

\subsection{Berezinian Sheaf from Cohomology of Forms and Operators} The last construction of the Berezinian sheaf that we will discuss stands somewhat in between the previous ones, as it employs the sheaf $\mathcal{D}_\mani$ to \virgolette deform'' in a non-commutative fashion the previous Koszul complex construction \cite{CNR}. In particular, we consider the following tensor product of sheaves.
\begin{definition}[Universal de Rham Sheaf] Let $\mani$ be a real or complex supermanifold. We call the tensor product sheaf $\mathpzc{DR}_\mani \defeq \Omega^\bullet_\mani \otimes_{\stsheaf} \mathcal{D}_\mani$ the universal de Rham sheaf of $\mani$.
\end{definition}
{\remark It is easy to see that $\mathpzc{DR}_\mani$ is both $\mathbb{Z}$-graded and $\mathbb{Z}_2$-graded. Further, it is a sheaf of left $\Omega^\bullet_\mani$-modules - and hence also left $\mathcal{O}_\mani$-modules - and a sheaf of right $\mathcal{D}_\mani$-modules. Notice, by the way, that the structure of right $\mathcal{O}_\mani$-module induced by $\mathcal{D}_\mani$ does not coincide with that of left $\mathcal{O}_\mani$-module, since $\mathcal{D}_\mani$ is non-commutative. }
{\remark There is an obvious operator acting on $\mathpzc{DR} (\mani)$, whose action is given by
\bear
\xymatrix@R=1.5pt{
\mathscr{D} : \mathpzc{DR}_\mani \ar[r] & \mathpzc{DR}_\mani \\
\omega \otimes F \ar@{|->}[r] &  \mathscr{D} \left (\omega \otimes F \right ) \defeq  \sum_{a} (-1)^{|\omega| |x_a|} dx_a  \omega \otimes \partial_{x_a}  \circ F,
 }
\eear
for any $\omega \otimes F \in \mathpzc{DR}_\mani.$ It is not hard to prove that $\mathscr{D}$ is globally well-defined - as the $dx_a$'s and the $\partial_{x_a}$'s transform dually - and that for any $f \in \mathcal{O}_\mani$ one has that indeed $\mathscr{D} (\omega f \otimes F) = \mathscr{D} (\omega \otimes f F)$. Further, it is immediate to observe that the operator $\mathscr{D}$ is \emph{nilpotent}, since it can be seen as the multiplication by the odd element $\sum_a dx_a \otimes \partial_{x_a} \in \mathpzc{DR}_\mani$, \emph{i.e.}\ 
\bear
\mathscr{D} (\omega \otimes F) \defeq \left (\sum_a dx_a \otimes \partial_{x_a} \right ) \cdot ( \omega \otimes F).
\eear
We thus have that the pair $(\mathpzc{DR}^\bullet_\mani, \mathscr{D})$ defines a complex of sheaves, whose $\mathbb{Z}$-grading is induced by the one of $\Omega^\bullet_\mani$. The cohomology of this complex provides another construction of the Berezinian sheaf of $\mani$.}
\begin{theorem}[Cohomology of $\mathpzc{DR}_\mani^\bullet$] \label{Con3} Let $\mani$ be a real or complex supermanifold of dimension $p|q$. Then the homology of the complex $(\mathpzc{DR}^\bullet_\mani, \mathscr{D})$ is naturally isomorphic to the Berezinian sheaf of $\mani$, \emph{i.e.}\
\bear
\mathcal{H}^p \left ( (\mathpzc{DR}^\bullet_\mani, \mathscr{D} ) \right ) \cong \mathcal{B}er (\mani)
\eear
and $\mathcal{H}^i \left ( (\mathpzc{DR}^\bullet_\mani, \mathscr{D} ) \right ) \cong 0 $ for any $i \neq p$.
\end{theorem}
\begin{proof} We need to construct a homotopy for the operator $\mathscr{D}$. We work in a local chart $(U, x_a)$ so that the sheaf $\mathpzc{DR}_\mani \defeq \Omega^\bullet_{\mani} \otimes_{\stsheaf} \mathcal{D}_\mani$ is given by the sheaf of vector spaces generated by monomials having the form $ \omega \otimes F $, with $\omega = dx_I$ and $F = \partial_J f$ for multi-indices $I$ and $J$ and some function $f \in \stsheaf \lfloor_{U}$. We claim that the homotopy is given by the (local) operator 
\bear
\mathscr{H} (\omega \otimes F) \defeq \sum_a (-1)^{|x_a|( |\omega| + |\partial^J| + 1)}{\iota_{\pi \partial_{x_a}} }dx_I \otimes [\partial_J , x_a] f.
\eear
where the derivation $\iota_{\pi \partial_{x_a}} \defeq \partial_{dx_a}$ is the contraction with respect to the coordinate field $\pi \partial_a$, so that $\iota_{\pi \partial_{x_a}} (dx_b) = \delta_{ab}.$ A lengthy but not too hard computation gives
\begin{align}
( \mathscr{H} \mathscr{D} + \mathscr{D} \mathscr{H} ) (\omega \otimes F) & =  \left ( p + q + \mbox{deg}_0 (\omega) + \mbox{deg}_0 (\partial_J) - \mbox{deg}_1 (\omega) - \mbox{deg}_1 (\partial_J) \right ) (\omega \otimes F), 
\end{align}
where $\deg_0$ and $\deg_1$ are the degrees with respect to the even and odd generators. We have that $\mbox{deg}_1 (\omega) \leq p $ and $\mbox{deg}_1 (\partial_J) \leq q$, therefore the homotopy fails if and only if $\mbox{deg}_0 (\omega) = 0 = \mbox{deg}_0 (\partial_J) $ and $\mbox{deg}_1 (\omega) = p$, $\mbox{deg}_1 (\partial^J) = q:$ the monomial $ \omega \otimes F$ is given by $ dz_1 \ldots dz_p \otimes \partial_{\theta_1} \ldots \partial_{\theta_q}  f$
for $f \in \mathcal{O}_\mani \lfloor_{U}$.  This element is clearly in the kernel of $\mathscr{D}$ and it generates the Berezinian sheaf $\mathcal{B} er(\mani)$ as $\mathcal{O}_\mani$-module.
\end{proof}
{\remark Notice, once again, that this construction holds true in the smooth and holomorphic category, but also in the algebraic category.}

\section{Properties of the Berezinian Sheaf}

\noindent Having defined the Berezinian sheaf in the previous section, we are now interested in studying its properties. In the first subsection, in theorem \ref{bercan}, we will show how the Berezinian sheaf of a supermanifold $\mani$ is related to the canonical sheaf of its underlying manifold $\manir$: this will prove crucial in the definition of the Berezin integral. Further, in the second subsection, we will show that the Berezinian sheaf is a \emph{right} $\mathcal{D}_\mani$-module. This theory has been first developed by Penkov in the marvelous \cite{Penkov}, where \emph{Serre duality} - see also \cite{OP} - and \emph{Mebkhout duality} for complex supermanifolds are also proved. Later on, results in this direction appeared also in the book \cite{Manin}, however the relation with $\mathcal{D}_\mani$-module theory is left somewhat hidden. In this section we make this connection apparent, by spelling out all of its details and stressing differences and similarities with the ordinary commutative theory. Indeed, the presence of a right $\mathcal{D}_\mani$-module structure on the Berezinian sheaf is another striking analogy between the Berezinian sheaf and its commutative counterpart, the canonical sheaf on an ordinary manifold $X$, which carries as well the structure of right $\mathcal{D}_X$-module. More precisely we will see that in a similar fashion as in ordinary commutative theory, the right $\mathcal{D}_\mani$-module structure on $\mathcal{B}er (\mani)$ is related to the action of the Lie derivative on it.
{\remark As it should be clear from the previous section, it has to be noticed that the Berezinian sheaf is \emph{not} a sheaf of differential forms, and as such it does not appear in the de Rham complex $\Omega^\bullet_\mani$ of $\mani$. It follows that it is not trivial to define a notion of Lie derivative acting on sections of $\mathcal{B}er (\mani)$. Indeed, the so-called \emph{Cartan formula} $\mathcal{L}_V \omega = \{ d, \iota_V \} (\omega) $ which holds true for differential forms $\omega \in \Omega^i_X$ for any $i = 0, \ldots , \dim X$ and can be readily generalized to the de Rham complex of a supermanifold, does not apply to the Berezinian.  }

\subsection{Berezinian and Canonical Sheaf} In this first subsection we prove an easy, yet very important isomorphism, which establishes a crucial relation between the Berezinian sheaf $\mathcal{B}er (\mani) $of a supermanifold $\mani$ of dimension $p|q$ and the canonical sheaf $\Omega^p_{\manir}$ of the reduced space of $\manir$. We start by proving the following ancillary result, see for example \cite{Manin}.
\begin{lemma} \label{ferm} Let $\mani $ be a real or complex supermanifold. Then we have the following isomorphism of sheaves of $\stsheafred$-modules
\bear
\mathcal{F}_\mani \cong \Pi \left ( \slantone{\Omega^1_\mani}{\mathcal{J}_\mani \Omega^1_\mani} \right )_0  
\eear
where the subscript $0$ refers to the $\mathbb{Z}_2$-grading of the quotient sheaf $\Omega^1_\mani / \mathcal{J}_\mani \Omega^1_\mani.$
\end{lemma}
\begin{proof} First of all, notice that locally, a basis of $\mathcal{F}_\mani = \slanttwo{\mathcal{J}_\mani}{\mathcal{J}^2_\mani}$ is given by $\theta_\alpha \,\mbox{mod}\, \mathcal{J}_\mani^2$ for $\alpha = 1, \ldots, q$ where $q$ is the odd dimension of $\mani$. Further, locally, a basis of $\left ( \Omega^1_{\mani}/\mathcal{J}_\mani \Omega^1_\mani \right )_1 $ read $d\theta_\alpha \mbox{mod} \mathcal{J}_\mani \Omega^1_\mani$, again for $\alpha= 1, \ldots, q$ where $m$
is the odd dimension of $\mani$. 
We claim that the isomorphism reads $\theta^j \, \mbox{mod} \, \mathcal{J}_\mani^2 \mapsto d\theta^j \,\mbox{mod}\, \mathcal{J}_\mani \Omega^1_\mani$, and we prove that it is well-defined and independent of the charts. Indeed, let $x^\prime_a = z_i | \theta^\prime_\alpha $ be another local system of coordinates, then 
the transformation for the transformation of $\mathcal{F}_\mani = \mathcal{J}_\mani / \mathcal{J}_\mani^2$ one has that $\theta^\prime_\alpha \equiv \sum_{\beta} f_{\alpha \beta} (x) \theta_\beta \, \mbox{mod}\, \mathcal{J}^2_\mani.$ It follows that
\begin{align}
d \theta^\prime_\alpha & = \sum_b dx_b \frac{\partial \theta^\prime_\alpha}{\partial x^b}  + \sum_\beta d\theta_\beta \frac{\partial \theta^\prime_\alpha}{\partial \theta_\beta}  \nonumber \\ 
& = \sum_b dx_b  \frac{\partial}{\partial x_b } \left (  \sum_\gamma f_{\alpha \gamma}  (x) \theta_\gamma \, \mbox{mod} \, \mathcal{J}^2_\mani \right ) + \sum_\beta d\theta_\beta \frac{\partial}{\partial \theta_\beta} \left (  \sum_\gamma f_{\alpha \gamma} (x) \theta_\gamma \, \mbox{mod} \, \mathcal{J}^2_\mani\right )  \nonumber \\
& = \sum_{b, \gamma} dx_b \frac{\partial f_{\alpha \gamma} (x)}{\partial x_b} \theta_\gamma \, \mbox{mod}\, \mathcal{J}^2_\mani  + \sum_{\beta} d\theta_\beta f_{\alpha \beta} (x) \, \mbox{mod} \, \mathcal{J}^2_\mani  \nonumber \\
& \equiv \sum_\beta d\theta_\beta \, \mbox{mod} \, \left (  \mathcal{J}_\mani \Omega^1_\mani \right ) f_{\alpha \beta} (x) , 
\end{align}
since $\sum_{b, \gamma} dx_b \frac{\partial f_{\alpha \gamma} (x)}{\partial x_b} \theta_\gamma \, \mbox{mod}\, \mathcal{J}^2_\mani \equiv 0 \, \mbox{mod}\, \mathcal{J}_\mani \Omega^1_\mani.$ Reversing the parity of the local generators $d\theta_\beta$ concludes the proof. 
\end{proof}
\noindent Using the above lemma we prove the result we claimed at the beginning of the subsection.
\begin{theorem} \label{bercan} Let $\mani$ be a real or complex supermanifold and let $\mathcal{J}_\mani \subset \mathcal{O}_\mani$ be its nilpotent sheaf. Then there is a canonical isomorphism of sheaves of $\mathcal{O}_{\manir} \cong \mathcal{O}_\mani / \mathcal{J}_\mani$-modules 
\bear
\varphi: \mathcal{J}_\mani^q \mathcal{B}er (\mani) \stackrel{\cong}{\longrightarrow} \Omega^p_{\manir}, 
\eear 
where $\mathcal{B}er (\mani)$ is the Berezinian sheaf of $\mani$ and $\Omega^p_{\manir}$ is the canonical sheaf of $\manir.$ In local coordinates $x_a = z_i | \theta_\alpha$ the above isomorphism reads
\bear
\mathcal{D}(x) \theta_1 \ldots \theta_q f  \stackrel{\cong}{\longmapsto} dz_{1, \mathpzc{red}} \ldots dz_{p, \mathpzc{red}} f_{ \mathpzc{red}},
\eear
where $\mathcal{D}(x) \in \mathcal{B}er(\mani)$, $\theta_1\ldots \theta_q \in \mathcal{J}^q_\mani$, $dz_{1, \mathpzc{red}} \wedge \ldots \wedge dz_{p, \mathpzc{red}} f_{\mathpzc{red}} \in \Omega^p_{\manir}$ and where $f \in \mathcal{O}_\mani$ and $f_{\mathpzc{red}} = f \, \mbox{\emph{mod}}\, \mathcal{J}_\mani \in \mathcal{O}_{\mani_{\mathpzc{red}}}$.
\end{theorem}
\begin{proof} First of all we observe that $\mathcal{J}^q_\mani \mathcal{B}er(\mani) $ is obviously a sheaf of $\mathcal{O}_{\manir}$-modules, as $\mathcal{J}_\mani^q \cong \mathcal{J}^q_\mani / \mathcal{J}^{q+1}_\mani$ is a sheaf of $\mathcal{O}_{\manir}$-modules. Also, it is clear that if a generating section $\theta_\alpha \in \mathcal{J}_\mani / \mathcal{J}^2_\mani$ is such that $\theta^\prime_\alpha \equiv \sum_\beta f_{\alpha \beta} (z) \theta_\beta\, \mbox{mod} \, \mathcal{J}^2_\mani$ then 
\bear
\mathcal{J}^q_\mani \owns \theta_1 \ldots \theta_q = \det (f_{\alpha \beta}) \theta_1 \ldots \theta_q = \det (f_{\alpha \beta})_{\mbox{\emph{\tiny{red}}}} \theta_1 \ldots \theta_q,
\eear 
Now, considering local coordinates $x_a = z_i | \theta_\alpha$, for an index $a $ running on both even and odd coordinates, for a generic change of coordinates $x^\prime_a = z^\prime_i (x) | \theta_\alpha^\prime (x)$ one has that 
\begin{align}
\mathcal{D} (x^\prime ) \theta_1^\prime \ldots \theta^\prime_q & = \mathcal{D}(x) Ber \left ( \frac{\partial x^\prime}{\partial x} \right )_{\mathpzc{red}} \det (f_{\alpha \beta})_{\mathpzc{red}} \, \theta_1 \ldots \theta_q \nonumber \\
& = \mathcal{D}(x) \det \left (\frac{\partial z^\prime}{\partial z} \right )_{\mathpzc{red}} \det \left (\frac{\partial \theta^\prime}{\partial \theta} \right )^{-1}_{\mathpzc{red}} \det (f_{\alpha \beta})_{\mathpzc{red}} \, \theta_1 \ldots \theta_q. 
\end{align}
It follows from the previous lemma \ref{ferm} that $\det (\partial_{\theta} \theta^\prime)_{\mathpzc{red}} = \det (f_{\alpha \beta})_{\mathpzc{red}}$, which concludes the theorem since $\det (\partial_z z^\prime)_{\mathpzc{red}}$ is indeed the transformation of a generating section of the canonical sheaf $\Omega^p_{\manir} $ of the reduced manifold. \end{proof}
{\remark The above theorem holds true also if one considers \emph{compactly supported} sections instead: this will prove crucial to define a meaningful notion of integral on supermanifolds, as we shall see later on in this paper, see Section 7.}

\subsection{$\mathcal{D}_\mani$-modules and Connections} It is an easy yet fundamental result of $\mathcal{D}$-module theory that giving a $\mathcal{D}$-module structure on a sheaf corresponds to define a \emph{flat connection} on it. Nonetheless, as observed above, attention must be paid due to the non-commutativity of the sheaf $\mathcal{D}$ so that one has to distinguish between \emph{left} and \emph{right} action and thus \emph{left} and \emph{right} $\mathcal{D}$-module structures: accordingly, we will introduce \emph{left} and \emph{right} connections on sheaves and the related notion of \emph{flatness}.  \\
Let us start from \emph{left} $\mathcal{D}$-modules. In this case, the needed left action is induced by a \emph{left} connection: this is nothing but the standard notion of affine (or Koszul) connection on a sheaf introduced in ordinary differential geometry. In particular, we will work over a real or complex supermanifold, and thus the base field will be either $\mathbb{K} = \mathbb{R}$ or $\mathbb{K} = \mathbb{C}$. The notations employed mostly follow \cite{Manin}, in particular for any $f \in \mathcal{O}_\mani$ and $X \in \mathcal{T}_{\mani}$ we denote the commutator $[X, f] = X(f)$ as
\bear
[X, f] \defeq X \circ f - (-1)^{|X||f|} f \circ X,
\eear  
as to stress that we are considering the operator product $\circ$ in $\mathcal{D}_\mani$.
\begin{definition}[Left Connection on $\mathcal{E}$] Let $\mani$ be a real or complex supermanifold with structure sheaf $\mathcal{O}_\mani$ and let $\mathcal{E}$ be a sheaf of $\mathcal{O}_\mani$-modules. Then we say that a \emph{left} connection on $\mathcal{E}$ is a $\mathbb{K}$-bilinear morphism $\Delta_{\mathpzc{L}} : F^1 \mathcal{D}_\mani \otimes_{\mathbb{K}} \mathcal{E} \rightarrow \mathcal{E}$ such that the following are satisfied for any $f \in \mathcal{O}_\mani$, $X \in \mathcal{T}_\mani$ and $e \in \mathcal{E}:$
\begin{enumerate}
\item $\Delta_{\mathpzc{L}} (f \otimes e) = f e$,
\item $\Delta_\mathpzc{L} (f \circ X \otimes e) = f \Delta_{\mathpzc{L}} (X \otimes e), $
\item $\Delta_\mathpzc{L} (X \circ f \otimes e) = \Delta_\mathpzc{L} (X \otimes f e).$
\end{enumerate} 
In particular, we say that the left connection $\Delta_\mathpzc{L}$ is \emph{flat} if for any $e \in \mathcal{E}$ and $X, Y \in \mathcal{T}_\mani $ it satisfies
\bear
\Delta_\mathpzc{L} ([X, Y] \otimes e) = \Delta_\mathpzc{L} (X \otimes \Delta_\mathpzc{L} (Y \otimes e)) -  (-1)^{|X||Y|}\Delta_{\mathpzc{L}} (Y \otimes \Delta_\mathpzc{L} (X \otimes e)).
\eear
\end{definition}
{\remark It is to be noted that the previous definition is adapted to serve the needs of $\mathcal{D}$-module theory. In particular, notice that since $F^1 \mathcal{D}_\mani = \mathcal{O}_\mani \oplus \mathcal{T}_\mani$ generates $\mathcal{D}_\mani$, the definition has been extended as to include also the action of the structure sheaf $\mathcal{O}_\mani \subset F^1 \mathcal{D}_\mani$ on $\mathcal{E}$ in the first point. The second point is the usual $\mathcal{O}_\mani$-linearity in the first entry of the connection and the third point is the \emph{Leibniz rule}, since
\begin{align}
\Delta_{\mathpzc{L}} ( X \circ f \otimes e) & = \Delta_\mathpzc{L} \left ( \left ( X (f ) + (-1)^{|f| |X|}f \circ X\right )\otimes e \right )  \nonumber \\ 
& = \Delta_\mathpzc{L} (X(f) \otimes e)+ (-1)^{|f||X|}\Delta_{\mathpzc{L}} (f\circ X \otimes e)  \nonumber \\
& =  X(f)e + (-1)^{|f||X|} f \Delta_\mathpzc{L} (X \otimes e ).  
\end{align}
Finally, the condition of flatness, or \emph{integrability} for a connection can be rephrased by saying that the connection commutes with the operation of commutator, \emph{i.e.}
\bear
\Delta_{\mathpzc{L}} ([X, Y] \otimes  -) = [ \Delta_\mathpzc{L} (X\otimes - ), \Delta_\mathpzc{L} (Y \otimes - ) ],   
\eear
 for any $X, Y \in \mathcal{T}_\mani.$\\
We can now state the following crucial result.} 
\begin{theorem}[Left $\mathcal{D}_\mani$-Modules \& $\mathcal{O}_\mani$-Modules] \label{leftD} Let $\mani$ be a real or complex supermanifold and let $\mathcal{E}$ be a sheaf on $\mani$. Then $\mathcal{E}$ is a sheaf of \emph{left} $\mathcal{D}_\mani$-module on $\mani$ if and only if $\mathcal{E}$ is a sheaf of $\mathcal{O}_\mani$-modules endowed with a \emph{flat left connection}.
\end{theorem}
{\remark The proof of the above theorem is obvious and it simply amounts to check that \emph{associativity} of the left $\mathcal{D}_\mani$-action on $\mathcal{E}$ is reproduced by the flat connection. We will prove instead the analogous result in case of right $\mathcal{D}_\mani$-modules, which is equally simple but less ordinary, and might cause some confusion.   }
{\remark Notice also that the above theorem does \emph{not} require the sheaf $\mathcal{E}$ of $\mathcal{O}_\mani$-modules to be locally-free of finite rank, therefore one can look at a $\mathcal{D}_\mani$-module as a generalization of a vector bundle endowed with a flat connection.\\
\noindent The next corollary gives an obvious example of sheaf of left $\mathcal{D}_\mani$-modules. }
\begin{corollary}[Structure Sheaf $\mathcal{O}_\mani$] Let $\mani $ be a real or complex supermanifold and let $\mathcal{O}_\mani$ be its structure sheaf. Then $\mathcal{O}_\mani$ is a sheaf of \emph{left} $\mathcal{D}_\mani$-modules. 
\end{corollary}
\begin{proof} Obviously,  $\mathcal{O}_\mani$ can be endowed with a flat left connection, which is nothing but the {exterior derivative}, seen as a map $d : \mathcal{T}_\mani \otimes \mathcal{O}_\mani \rightarrow \mathcal{O}_\mani$ via the $X \otimes f \mapsto (df)(X)$ and where the action $\mathcal{O}_\mani \otimes \mathcal{O}_\mani \rightarrow \mathcal{O}_\mani$ is given by the superalgebra structure of $\mathcal{O}_\mani.$ Flatness is obvious.
\end{proof}
\noindent Let us now pass to the case of right $\mathcal{D}_\mani$-modules. In order to prove the analogous result of the theorem \ref{leftD} for right $\mathcal{D}_\mani$-modules we need to introduce a different kind of connection, which is to be related to a right action. As in \cite{Manin}, employing the same notation as above we have the following.
\begin{definition}[Right Connection on $\mathcal{E}$] Let $\mani$ be a real or complex supermanifold with structure sheaf $\mathcal{O}_\mani$ and let $\mathcal{E}$ be a sheaf of $\mathcal{O}_\mani$-modules. Then we say that a \emph{right} connection on $\mathcal{E}$ is a $\mathbb{K}$-bilinear morphism $\Delta_{\mathpzc{R}} : \mathcal{E} \otimes_{\mathbb{K}} F^1 \mathcal{D}_\mani  \rightarrow \mathcal{E}$ such that the following are satisfied for any $f \in \mathcal{O}_\mani$, $X \in \mathcal{T}_\mani$ and $e \in \mathcal{E}:$
\begin{enumerate}
\item $\Delta_\mathpzc{R} (e \otimes f) = e f $;
\item $\Delta_\mathpzc{R} (e \otimes X \circ f) =  \Delta_{\mathpzc{R}} (e \otimes X) f; $
\item $\Delta_\mathpzc{R} (e \otimes f \circ X) = \Delta_\mathpzc{R} (e f \otimes X),$
\end{enumerate} 
In particular, we say that the right connection $\Delta_\mathpzc{R}$ is \emph{flat} if for any $e \in \mathcal{E}$ and $X, Y \in \mathcal{T}_\mani $ it satisfies
\bear
\Delta_\mathpzc{R} ( e \otimes [X, Y] ) = \Delta_\mathpzc{R} ( \Delta_\mathpzc{R} (e \otimes X ) \otimes Y ) -  (-1)^{|X||Y|}\Delta_{\mathpzc{R}} ( \Delta_\mathpzc{R} (e \otimes Y ) \otimes X ).
\eear
\end{definition}
{\remark Notice that the third point in the definition is $\mathcal{O}_\mani$-linearity and that we have a \emph{modified} Leibniz rule, adapted to right structures. Indeed, one has
\begin{align}
\Delta_{\mathpzc{R}} (e \otimes X \circ f) & \defeq \Delta_\mathpzc{R} (e \otimes ( X(f) + (-1)^{|X||f|} f \circ X) )\nonumber \\ 
& = \Delta_\mathpzc{R} (e \otimes X(f)) + (-1)^{|X||f|}\Delta_\mathpzc{R} (e \otimes f \circ X)  \nonumber \\ 
& =  e X(f) + (-1)^{|X| |f|} \Delta_\mathpzc{R} (e f \otimes X),
\end{align}
so that the second property above can be rewritten as 
\bear \label{modified}
\Delta_\mathpzc{R} (e f \otimes X) = (-1)^{|X| |f|} ( \Delta_{\mathpzc{R}} (e \otimes X) f - e X(f) ).
\eear
Using right connections we can prove the following. }
\begin{theorem}[Right $\mathcal{D}_\mani$-Modules \& $\mathcal{O}_\mani$-Modules] \label{rightD} Let $\mani$ be a real or complex supermanifold and let $\mathcal{E}$ be a sheaf on $\mani$. Then $\mathcal{E}$ is a sheaf of \emph{right} $\mathcal{D}_\mani$-module on $\mani$ if and only if $\mathcal{E}$ is a sheaf of $\mathcal{O}_\mani$-modules endowed with a \emph{flat right connection}.
\end{theorem}
\begin{proof} The right $\mathcal{D}_\mani$-module structure on $\mathcal{E}$ corresponds to a right action 
\bear
\xymatrix@R=1.5pt{
\sigma_{\mathpzc{R}} : \mathcal{E} \times \mathcal{D}_\mani \ar[r] &  \mathcal{E}  \\
(e, F) \ar@{|->}[r] & e \cdot F \defeq \sigma_{\mathpzc{R}} (e, F)
}
\eear
In particular, associativity reads
\bear
\sigma_{\mathpzc{R}} (\sigma_{\mathpzc{R}} (e, D), H) = \sigma_{\mathpzc{R}} (e , D\circ H)
\eear
or analogously $( e \cdot D ) \cdot H = e \cdot (D \circ H)$ for $e \in \mathcal{E}$ and $D, H \in \mathcal{D}_\mani.$ Let us define $\sigma_{\mathpzc{R}} $ as acting by right multiplication on functions. Then for $X \in \mathcal{T}_\mani$ and $f \in \mathcal{O}_\mani$ we have that associativity reads
\bear
\sigma_{\mathpzc{R}} (e , f \circ X) = \sigma_{\mathpzc{R}} (\sigma_{\mathpzc{R}} (e, f) , X ) = \sigma_{\mathpzc{R}} (e f, X),
\eear
which is the last of the defining conditions for a right connection. Furthermore, we have
\begin{align}
\sigma_{\mathpzc{R}} (e, X(f)) & = \sigma_{\mathpzc{R}} (e, (X\circ f - (-1)^{|X||f|} f \circ X)) \nonumber \\
& = \sigma_{\mathpzc{R}} (e, X \circ f) - (-1)^{|X||f|} \sigma_{\mathpzc{R}} (ef , X ) \nonumber \\
& = \sigma_{\mathpzc{R}} (\sigma_{\mathpzc{R}}(e, X) , f) - (-1)^{|X||f|} \sigma_{\mathpzc{R}} (ef , X ) \nonumber \\
& = \sigma_{\mathpzc{R}} (e, X) f  - (-1)^{|X||f|} \sigma_{\mathpzc{R}} (ef , X ),
\end{align}
where we have used linearity and associativity. On the other hand, by definition $\sigma_{\mathpzc{R}} (e, X(f)) = e X(f)$, so that 
\bear
\sigma_{\mathpzc{R}} (ef , X ) = (-1)^{|X||f|} ( \sigma_{\mathpzc{R}} (e, X) f - e X(f)),
\eear
which is the modified Leibniz rule and it corresponds to the second defining condition of a right connection. Lastly, 
\begin{align}
\sigma_{\mathpzc{R}} (e, [X, Y]) & = \sigma_{\mathpzc{R}} (e, X\circ Y- (-1)^{|X||Y|} Y \circ X ) \nonumber \\
& = \sigma_{\mathpzc{R}} (\sigma_{\mathpzc{R}} (e, X ), Y ) - (-1)^{|X||Y|} \sigma_{\mathpzc{R}} (\sigma_{\mathpzc{R}} (e, Y), X)
\end{align}
which is the requirement of flatness. Notice that, conversely, starting from a right connection, it is enough to prove that it defines a right action on the generating $F^1 \mathcal{D}_\mani \cong \mathcal{O}_\mani \oplus \mathcal{T}_\mani,$ with $[X, f] = X(f)$ to have a right action of $\mathcal{D}_\mani$ via associativity. 
\end{proof}
\noindent Just like above for left $\mathcal{D}_\mani$-modules, the above theorem give an alternative characterization of right $\mathcal{D}_\mani$-modules. Notice, also, that the notion of right connection is not at all exotic as it might sound at first: namely working over an ordinary real or complex manifold $X$ it is easy to prove that - up to a sign - the \emph{Lie derivative} defines a right connection on the canonical sheaf $\omega_X \defeq \wedge^{\dim X} \mathcal{T}^\ast_X$, which is therefore a sheaf of \emph{right} $\mathcal{D}_X$-modules. 
\begin{lemma}[$\omega_X$ is a Right $\mathcal{D}_X$-module] \label{DrightX} Let $X$ be a real or complex manifold and let $\Omega^{\dim X}_X$ be its canonical sheaf. Then $\Omega^{\dim X}_X$ is a sheaf of right $\mathcal{D}_\mani$-module.
\end{lemma}
\noindent The proof of this theorem, together with a detailed discussion about the relation between the Lie derivative and the right $\mathcal{D}_X$-module structure of the canonical sheaf of an ordinary manifold, is deferred to the appendix, as to keep the focus on the case of supermanifolds in the main text. \\

\noindent In light of the previous section and the definition of the Berezinian sheaf of a supermanifold as the correct super-analog of the notion of canonical sheaf for an ordinary manifold, it is natural to ask if the $\mathcal{D}_\mani$-module property of lemma \ref{DrightX} goes through to the super setting and also the Berezinian sheaf is a sheaf of \emph{right} $\mathcal{D}_\mani$-modules. We will prove that this is indeed the case in the next section.

\subsection{Lie Derivative of $\mathcal{B}er (\mani)$ and Right $\mathcal{D}_\mani$-module Structure} \label{LieBerRight} Before we actually compute the action of the Lie derivative on a section of the Berezinian sheaf of a supermanifold, we start with some remarks concerning the relationship between left and right structures on a supermanifold. In particular, let $\mathcal{E}$ be a locally-free sheaf of \emph{left} $\mathcal{O}_\mani$-modules of rank $p|q$ which is generated over an open set $U$ by a set $\{ e_a \}_{a \in I}$ of $p$ even and $q$ odd generators. Then $\mathcal{E}$ is naturally also a locally-free sheaf of \emph{right} $\mathcal{O}_\mani$-modules simply taking into account the \emph{sign rule}, \emph{i.e.}\ given a local section $s_U \in \mathcal{E}$ over an open set $U$ such that $s_U = \sum_a {f^a} e_a $ for $f^a \in \mathcal{O}_\mani (U)$, then
\bear
s_U = \sum_a {f^a} e_a = \sum_a (-1)^{|e_a| |f^a| } e_a f^a.
\eear  
In particular, the tangent sheaf $\mathcal{T}_\mani$, realized as sheaf of derivations of the structure sheaf $\mathcal{D}er_{\mathbb{C}}(\mathcal{O}_\mani)$, can be seen as a sheaf of \emph{left} $\mathcal{O}_\mani$-modules. Accordingly, one has that a vector field acts as usual from the left, \emph{i.e.}\
\bear
\overset{\rightarrow}{X} (f) \defeq \sum_a X^a \overset{\rightarrow}{\partial_a} (f),
\eear 
where we have defined $\partial_a \defeq \frac{\partial}{\partial x_a} $, for $x_a$ an even or odd coordinate function. The action of the left derivation associated to $\overset{\rightarrow}{\partial_a} $ is defined as usual, \emph{i.e.}\ $\overset{\rightarrow}\partial_a (f) = \partial_a f$, and it is such that on coordinate functions one has $\overset{\rightarrow}{\partial_a} (x_b) = \delta_{ab}$. As seen above, $\overset{\rightarrow}{X}$ satisfies the Leibniz rule in the form 
\bear 
\overset{\rightarrow}{X} (fg) = \overset{\rightarrow}{X}(f) g + (-1)^{|X| |f|} f \overset{\rightarrow}{X} (g),
\eear
 for $f, g \in \mathcal{O}_\mani$. 
On the other hand, it makes sense to consider a right action of a vector field on a function, when $\mathcal{T}_\mani$ is seen as a sheaf of \emph{right} $\mathcal{O}_\mani$-modules. This is defined following the sign rule, by posing
\bear
(f) \overset{\leftarrow }{X} \defeq (-1)^{|X||f|} \overset{\rightarrow}{X} (f).
\eear
In particular, one sees that the action of the right derivation associated to $\overset{\leftarrow}{\partial_a}$ is the same as that of the left derivations on even coordinate functions and it is such that $(\theta_b) \overset{\leftarrow}{\partial_{\theta_a}} = - \overset{\rightarrow}{\partial_{\theta_a}} \theta_b$ on odd coordinate functions. 
Also, notice that in the case of right derivations the Leibniz rule reads 
\bear
 (fg) \overset{\leftarrow}{X} =  f ((g) \overset{\leftarrow}{X}) + (-1)^{|X||g|} ((f) \overset{\leftarrow}{X} )g.
\eear
for any $f , g \in \mathcal{O}_\mani.$
This agrees with the definition given via the sign rule, as indeed one finds
\begin{align}
\overset{\rightarrow}{X} (fg) & = (-1)^{|X| (|f| + |g|)}  (fg) \overset{\leftarrow}{X}. 
\end{align}
Having these considerations in mind, we look for an action of the Lie derivative on sections of the Berezinian sheaf of a supermanifold in a similar fashion as in \eqref{LieComm} for the canonical sheaf on an ordinary manifold. We remark, though, that this is not trivial. Indeed, sections of the Berezinian sheaf are not differential forms, and as such the usual Cartan calculus and Cartan homotopy formula for the Lie derivative does not apply.\\ 
In order to introduce a notion of Lie derivative we follow Deligne and Morgan, see \cite{Deligne} Chapter 3, Section 3.7. One first observes that an even vector field $X \in \mathcal{T}_X$ defines a flow $\varphi^t_X = \exp (tX)$
which is in general only defined in a neighborhood of $t = 0$ if $\mani$ is not compact. Then one defines the Lie derivative as follows - see \cite{Deligne}.
\begin{definition}[Lie Derivative] \label{LieDev} Let $\mani$ be a real or complex supermanifold, let $\mathcal{E}$ be a locally free sheaf on $\mani$ such that its sections carry a natural action of local diffeomorphisms and let $X \in \mathcal{T}_{\mani}$ be an even vector field. Then we define the Lie derivative of a section $\mathpzc{S} \in \mathcal{E}$ along $X$ as
\bear \label{LieFlow}
\mathcal{L}_{X} (\mathpzc{S}) \defeq \lim_{t \rightarrow 0} \frac{1}{t} \left ( \exp (tX)^\ast \mathpzc{S} - \mathpzc{S}\right ).
\eear 
Let now $X \in \mathcal{T}_\mani$ be an odd vector field and $k [\varepsilon]$ the ring over the field of real or complex numbers in one formal odd variable $\varepsilon.$ Then we define the Lie derivative of a section $\mathpzc{S}$ in $\mathcal{E}$ along $X$ by the condition that
\bear \label{conL}
\mathcal{L}_{\varepsilon X} (\mathpzc{S}) = \varepsilon \mathcal{L}_X (\mathpzc{S}),
\eear
where $\varepsilon X$ is an even vector field in the locally-free sheaf (of $\mathcal{O}_{\mani \times \mathfrak{spec}\, k[\varepsilon]}$-modules) $\mathcal{T}_{\mani \times \mathfrak{spec} \,k [\varepsilon]}$. 
\end{definition}
{\remark As in the ordinary context, this formalizes the idea of \emph{infinitesimal variation} induced on a certain section by the action of a certain vector field. It is possible to define the Lie derivative along an odd vector field using a notion of flow. In particular, given a real or complex supermanifold $\mani$ there exists a map
\bear
\varphi^{\varepsilon}_X: \mani \times \mathfrak{spec}\, k[\varepsilon]\rightarrow \mani
\eear
for $k$ the field of real or complex numbers and $\varepsilon $ an odd formal variable, such that $\varphi_X^{\varepsilon} \lfloor_{\mani_{\mathpzc{red}} \times \mathfrak{spec}\, k [\varepsilon]} \defeq id$ and such that 
$
(\varphi^\varepsilon_X)^\ast (f) \defeq f + \varepsilon X(f).
$
Then, for a section $f$, one defines the Lie derivative along $X$ as 
\bear
\mathcal{L}_X (f) \defeq \frac{d}{d\varepsilon} \left ((\varphi^{\varepsilon}_X)^\ast (f) \right ) |_{\varepsilon = 0},
\eear
so that one has $\mathcal{L}_X (f) = X(f)$. This is the same as the previous definition due to Deligne and Morgan, as indeed  $\varepsilon \mathcal{L}_X (f) = \mathcal{L}_{\varepsilon X} (f) = \varepsilon X (f)$ from \eqref{conL}. The case of differential forms is similar, while in the case of a vector field $Y$ one defines
\bear
\mathcal{L}_X (Y) \defeq \frac{d}{d\varepsilon} (\varphi^{- \varepsilon}_X)^{\ast} \left ( (\varphi^\varepsilon_{X } \right )_{\ast} Y) |_{\varepsilon = 0}, 
\eear
where $(\varphi^\varepsilon_{X})_{\ast} Y$ is the push-forward of $Y$.
Computing on a function $f$, one finds that 
\bear
\mathcal{L}_X (Y (f)) = \frac{d}{d\varepsilon} \left ( Y(f) + \varepsilon [X ,Y](f) \right ) |_{\varepsilon = 0} = [X, Y] (f), 
\eear
which agrees with Proposition 3.7.2 in \cite{Deligne} - that follows from the previous definition \ref{LieDev} -, and where the above bracket is a commutator or an anti-commutator depending on the parity of $Y$.
We now apply the definition \ref{LieDev} to recover an expression for the action of the Lie derivative on sections of the Berezinian sheaf. }
\begin{theorem}[Lie Derivative of Sections of Berezinian Sheaf] Let $\mani$ be a real or complex supermanifold. Let $\mathcal{D}$ be a section of the Berezianian sheaf $\mathcal{B}er ( \mani )$ and $X$ be a vector field in $\mathcal{T}_\mani$, such that they have trivializations $\mathcal{D} = \mathcal{D}(x) f(x)$ and $X = \sum_a X^a \partial_a $ in a certain local chart $(U, x_a)$. Then the Lie derivative $\mathcal{L}_X (\mathcal{D})$ defined as in \eqref{LieFlow} of $\mathcal{D} \in \mathcal{B}er (\mani)$ along the field $X \in \mathcal{T}_\mani$ is given in $U$ by
\bear
\mathcal{L}_X (\mathcal{D}) = (-1)^{|X| |\mathcal{D}|} \mathcal{D} (x) \sum_a (f X^a) \overset{\longleftarrow}{\partial_{a}},
\eear
for $a = 1, \ldots n |1, \ldots, m$ even and odd coordinates.
\end{theorem} 
\begin{proof} We use \eqref{LieFlow} and directly compute the variation of the sections $\mathcal{D} = \mathcal{D}(x) f(x)$ in the chart $U \subset X$ for an infinitesimal diffeormorphism controlled by a parameter $\varepsilon$, having parity $|\varepsilon| = |X|$, so that it makes sense to consider the even transformation $x^a \mapsto x^a + \varepsilon X^a.$ On the one hand one has 
\bear \label{sectionber}
\mathcal{D} (x^a + \varepsilon X^a) = \mathcal{D} (x^a) \mbox{Ber} ( (x^a + \varepsilon X^a ) \overset{\leftarrow}{\partial}_b) = \mathcal{D} (x^a) \mbox{Ber} ( \delta^a_b + \varepsilon \partial_b X^a ).  
\eear
The Berezinian of the coordinates transformation can be rewritten as
\bear \label{coorber}
\mbox{Ber} (\delta^a_b + \varepsilon \partial_b X^a) = 1 + t\,\mbox{Str} (\partial_b X^a) = 1 + \varepsilon \, \sum_a (-1)^{|x_a| |X_a|} \partial_a X^a,
\eear
where the sign is due to the super trace of the Jacobian matrix. Equivalently, acting from the left instead, one has
\bear
\mathcal{D} (x^a + \varepsilon X^a) = \mathcal{D} (x^a) \mbox{Ber} ( \overset{\rightarrow}{\partial}_b (x^a + \varepsilon X^a ) ) = \mathcal{D} (x^a) \mbox{Ber} ( \delta^a_b + \partial_b (\varepsilon  X^a )),  
\eear
so that this gives
\begin{align}
\mathcal{D} (x^a + \varepsilon X^a) & = \mathcal{D} (x^a) (1 + \mbox{Str} (\partial_b (\varepsilon  X^a))) \nonumber \\ 
&= \mathcal{D} (x^a) \left (1 + \sum_a (-1)^{|x_a|(|X| + |X^a|)} (\partial_a (\varepsilon  X^a)) \right) \nonumber \\ 
& = \mathcal{D} (x^a)\left  (1 + \varepsilon \sum_a (-1)^{|x_a||x_a| + |x_a| (|X^a| + |x_a|)}  \partial_a  X^a \right) \nonumber \\
& = \mathcal{D} (x^a)\left  (1 + \varepsilon \sum_a (-1)^{|x_a| |X^a|}  \partial_a X^a  \right )
\end{align}
where we have used that $|x_a| = |x_a||x_a|$ and that $|\varepsilon| = |X|.$ 
Considering the expansion of the local function, one finds
\bear
\mathcal{D}(x^a + \varepsilon X^a ) f(x^a + \varepsilon X^a) = \mathcal{D} (x^a) \left (1 + \varepsilon \, \sum_a (-1)^{|x_a| |X_a|} \partial_a X^a \right ) \left (f (x^a) + \varepsilon \sum_a X^a \partial_a f \right ).
\eear
In turn, this gives
\begin{align}
\mathcal{D}(x^a + \varepsilon X^a ) f(x^a + \varepsilon X^a)  & = \mathcal{D} (x^a) f(x^a) +  \nonumber  \\
& \; + \mathcal{D} (x) \varepsilon \sum_a  \left ( (-1)^{|x_a| |X_a|} (\partial_a X^a  ) f +  X^a (\partial_a f) \right ). 
\end{align}
We rearrange the summands inside the parentheses as follows: 
\begin{align}
 (-1)^{|x_a| |X_a|} (\partial_a X^a  ) f +  X^a (\partial_a f) & = (-1)^{|x_a||X^a| + |f| (|X|^a + |x_a|)} f (\partial_a X^a) + (-1)^{|X^a| (|x|_a + |f|)} (\partial_a f ) X^a \nonumber \\
& = (-1)^{|X^a| (|x_a| + |f|)} \partial_a (f X^a).
\end{align}
This leads to the following expression
\begin{align}
\mathcal{D}(x^a + \varepsilon X^a ) f(x^a + \varepsilon X^a) & = \mathcal{D}(x^a) f(x^a) + \mathcal{D}(x) \varepsilon \sum_a (-1)^{|X^a| (|x_a|) + |f|)} \partial_a (fX^a) \nonumber \\
& = \mathcal{D}(x^a) f(x^a) + \varepsilon \mathcal{D} (x) \sum_a (-1)^{|X^a| (|x_a| + |f|) + |\mathcal{D} (x)| (|X^a| + |x_a|)} \partial_a (fX^a) \nonumber \\
& = \mathcal{D}(x^a) f(x^a) + \varepsilon (-1)^{|\mathcal{D} (x)||X|} \mathcal{D} (x) \sum_a (-1)^{|X^a| (|x_a| + |f|)} \partial_a (fX^a) 
\end{align}
The only summand that matters in the computation of the Lie derivative \eqref{LieFlow} is the one that is linear in the parameter $\varepsilon$ and indeed we have that 
\bear
\mathcal{L}_X (\mathcal{D})= (-1)^{|\mathcal{D} (x)||X|} \mathcal{D} (x) \sum_a (-1)^{|X^a| (|x_a| + |f|)} \partial_a (fX^a).
\eear
Finally, notice that the right-hand side of the above expression for the Lie derivative can be re-written as
\begin{align}
\mathcal{L}_X (\mathcal{D}) & =  (-1)^{|\mathcal{D} (x)||X|} \mathcal{D} (x) \sum_a (-1)^{|X^a| (|x_a| + |f|) + |x_a| (|f| + |X^a|)}  (fX^a) \overset{\leftarrow}{\partial}_a \nonumber \\
& = (-1)^{|\mathcal{D}| |X|}  \mathcal{D}(x) \sum_a  (fX^a) \overset{\leftarrow}{\partial}_a 
\end{align}
taking into account that $|\mathcal{D}| = |\mathcal{D}(x)|+ |f|$, thus concluding the proof. 
\end{proof}
{\remark As already observed, the resemblance with the expression \eqref{LieComm} for the Lie derivative of the canonical sheaf is apparent. Nonetheless, in order to get the supergeometric analog of lemma \ref{LieLemma} it is useful to introduce the following re-definition of the Lie derivative to get rid of an inconvenient sign
\bear
\mathfrak{L}_X (\mathcal{D}) \defeq (-1)^{| \mathcal{D} | | X | }\mathcal{L}_X(\mathcal{D}).
\eear
\noindent In local coordinates, for sections $\mathcal{D} = \mathcal{D} (x)f$ and $X = \sum_a X^a \partial_a$, one has the following action
\bear \label{DressedLie}
\mathfrak{L}_X (\mathcal{D}) = \mathcal{D} (x) \sum_a (f X^a) \overset{\leftarrow}{\partial}_a = \mathcal{D} (x) \sum_a (-1)^{|x_a| (|f|  + |X^a|)}\partial_a (f X^a).
\eear
\noindent We can thus finally prove the following theorem.}
\begin{theorem}[$\mathcal{B}er (\mani)$ is a Right $\mathcal{D}_\mani$-module] \label{BerRightTheo} Let $\mani$ be a real or complex supermanifold and let $\Delta_{\mathpzc{R}}^{\mathcal{B}er} : \mathcal{B}er (\mani) \otimes_{\mathbb{K}} F^{1}\mathcal{D}_\mani \rightarrow \mathcal{B}er (\mani)$ be defined as 
\begin{align}
& \Delta^{\mathcal{B}er}_\mathpzc{R} (\mathcal{D} \otimes f) \defeq \mathcal{D} f, \\ 
& \Delta_\mathpzc{R}^{\mathcal{B}er} (\mathcal{D} \otimes X) \defeq - \mathfrak{L}_X (\mathcal{D} ), 
\end{align}
for any $\mathcal{D} \in \mathcal{B}er (\mani)$, $f \in \stsheaf \subset F^1 \mathcal{D}_\mani$ and $X \in \mathcal{T}_\mani \subset F^1 \mathcal{D}_\mani.$
Then the followings hold true.
\begin{enumerate}[leftmargin=*]
\item $\Delta_\mathpzc{R}^{\mathcal{B}er}$ defines a flat right connections on the Berezinian sheaf $\mathcal{B}er (\mani )$. 
\item In any system of coordinates $\Delta_\mathpzc{R}^{\mathcal{B}er}$ is the unique right connection on $\mathcal{B}er (\mani)$ satisfying 
\bear
\Delta^{\mathcal{B}er}_\mathpzc{R} (\mathcal{D} (x) \otimes \partial_a ) \defeq - \mathfrak{L}_{\partial_a} (\mathcal{D} (x))= 0 \quad \forall \, a = 1, \ldots, n | 1, \ldots, m,
\eear  
where $\mathcal{D}(x)$ is a local generating section of $\mathcal{B}er (\mani)$ and $\partial_a \in \mathcal{T}_\mani$ is a coordinate vector field.\\
\end{enumerate}
In particular $\Delta_{\mathpzc{R}}^{\mathcal{B}er}$ endows $\mathcal{B}er(\mani)$ with a \emph{right} $\mathcal{D}_\mani$-module structure.
\end{theorem}
\begin{proof} We start using \eqref{DressedLie} to show that the axioms of a right connection hold true. 
The only non trivial verification to carry out is that $\Delta^{\mathcal{B}er}_\mathpzc{R} (\mathcal{D} \otimes X \circ f) = \Delta_\mathpzc{R}^{\mathcal{B}er} (\mathcal{D} \otimes X) f$. We recall that $
X \circ f = X (f) + (-1)^{|X||f|} f \circ X.$ Upon using this, one computes  
\begin{align}
& \Delta^{\mathcal{B}er}_{\mathpzc{R}} \left (\mathcal{D}(x) g \otimes \sum_a X^a \partial_a \circ f \right ) = \Delta^{\mathcal{B}er}_\mathpzc{R} \left (\mathcal{D} (x) g \otimes \sum_a X^a \partial_a f \right ) + \Delta^{\mathcal{B}er}_{\mathpzc{R}} \left (\mathcal{D}(x) g \otimes \sum_a (-1)^{|f| |X|} f X^a \partial_a \right ) \nonumber \\
& = \mathcal{D} (x)  \sum_a g X^a \partial_a f - \mathcal{D} (x) \sum_a (-1)^{|f| (|X^a| + |x_a|) + |x_a| (|X^a| + |f| + |g|)} \partial_a (gfX^a) \nonumber \\
& = \mathcal{D} (x) \sum_a g X^a \partial_a f - \mathcal{D} (x) \sum_a (-1)^{|x_a| (|X^a| + |g|) } \left ( \partial_a (gX^a) f + (-1)^{|x_a | (|X^a| + |g|)} gX^a \partial_a f \right ) \nonumber \\
& = - \mathcal{D} (x) \sum_a (-1)^{|x_a| (|X^a| + |g|) }  \partial_a (gX^a) f,
\end{align}
upon observing that the first term and the last term cancel pairwise in the semi-last equality above. This gives
\bear
\Delta^{\mathcal{B}er}_\mathpzc{R} (\mathcal{D}(x) g \otimes \sum_a X^a \partial_a ) f = - \mathcal{D} (x) \sum_a (-1)^{|x_a| (|X^a| + |g|) }  \partial_a (gX^a) f,
\eear
thus proving that $\Delta^{\mathcal{B}er}_\mathpzc{R} (\mathcal{D} \otimes X \circ f) = \Delta_\mathpzc{R}^{\mathcal{B}er} (\mathcal{D} \otimes X) f.$\\
Also, it follows trivially from equation \eqref{DressedLie} that the action of $\Delta^{\mathcal{B}er}_\mathpzc{R}$ satisfies the third defining property of a right connection, while the first one is just the right multiplication.\\
Let us now prove that the right connection defined by $\Delta_\mathpzc{R}^{\mathcal{B}er}$ is flat, \emph{i.e.}\ it satisfies
\bear \label{superflat}
\Delta^{\mathcal{B}er}_{\mathpzc{R}} (\mathcal{D} \otimes [X, Y]) = \Delta^{\mathcal{B}er}_{\mathpzc{R}} ( \Delta^{\mathcal{B}er}_\mathpzc{R} (\mathcal{D} \otimes X ) \otimes Y ) - (-1)^{|X||Y|} \Delta^{\mathcal{B}er}_\mathpzc{R} (\Delta^{\mathcal{B}er}_\mathpzc{R} (\mathcal{D} \otimes Y) \otimes X),
\eear
for any pair of fields $X, Y \in \mathcal{T}_\mani.$ On the one hand, recalling that the supercommutator $[\cdot , \cdot]$ is given by  
\bear
[X, Y] \defeq \left [\sum_a X^a \partial_a ,\; \sum_b Y^b \partial_b \right ] = \sum_a \left ( X^b \partial_b Y^a - (-1)^{|X||Y|}Y^b \partial_b X^a \right ) \partial_a, 
\eear
one has
\begin{align}
\Delta_\mathpzc{R} (\mathcal{D} \otimes [X, Y]) = & - \mathcal{D} (x) \sum_{a, b} \Big [  (-1)^{|x_a | (|f| + |X| + |Y^a|)} \partial_a (f X^b \partial_b (Y^a)) + \nonumber \\ 
& - (-1)^{ |X| |Y| +|x_a | (|f| + |Y| + |X^a|)} \partial_a (f Y^b \partial_b (X^a))  \Big ],
\end{align}
where we have used that $
| X^b \partial_b Y^a| = |X| + |Y^a|$ and $ | Y^b \partial_b X^a| = |Y| + |X^a|.$
On the other hand, one has
\begin{align}
\Delta^{\mathcal{B}er}_\mathpzc{R} ( \Delta^{\mathcal{B}er}_\mathpzc{R} (\mathcal{D} \otimes X ) \otimes Y ) & = \Delta_\mathpzc{R}^{\mathcal{B}er} \left ( - \mathcal{D}(x) \sum_a (-1)^{|x_a| (|f| + |X^a|)} \partial_{a} (f X^a)\otimes \sum_b Y^b \partial_b \right ) \nonumber \\
& = + \mathcal{D} (x) \sum_{a,b} (-1)^{|x_a| (|f| + |X^a|) + |x_b| ( |f| + |X| + |Y^b| )} \partial_b ( \partial_a (f X^a ) Y^b).
\end{align}
Now note that
$
\partial_a (f X^a ) Y^b = \partial_a (f X^a Y^b) - (-1)^{|x_a| (|f| + |X^a|)} fX^a (\partial_a Y^b),
$
so that plugging this into the above one
\begin{align}
\Delta^{\mathcal{B}er}_\mathpzc{R} ( \Delta^{\mathcal{B}er}_\mathpzc{R} (\mathcal{D} \otimes X ) \otimes Y ) & = + \mathcal{D} (x) \sum_{a,b} (-1)^{|x_a| (|f| + |X^a|) + |x_b| ( |f| + |X| + |Y^b| )} \partial_b  \partial_a (f X^a  Y^b) + \nonumber \\
& - \mathcal{D} (x) \sum_{a, b} (-1)^{|x_b| ( |f| + |X| + |Y^b| )} \partial_b (f X^a \partial_a (Y^b)).
\end{align}
The same holds true for the other part, that is one finds
\begin{align}
\Delta^{\mathcal{B}er}_\mathpzc{R} ( \Delta^{\mathcal{B}er}_\mathpzc{R} (\mathcal{D} \otimes Y ) \otimes X ) & = \Delta_\mathpzc{R}^{\mathcal{B}er} \left ( - \mathcal{D}(x) \sum_a (-1)^{|x_b| (|f| + |Y^b|)} \partial_{b} (f Y^b)\otimes \sum_a X^a \partial_a \right ) \nonumber \\
& = + \mathcal{D} (x) \sum_{a,b} (-1)^{|x_b| (|f| + |Y^b|) + |x_a| ( |f| + |Y| + |X^a| )} \partial_a ( \partial_b (f Y^b ) X^a).
\end{align}
It follows that, by Leibniz rule, one gets
\begin{align}
\Delta^{\mathcal{B}er}_\mathpzc{R} ( \Delta^{\mathcal{B}er}_\mathpzc{R} (\mathcal{D} \otimes Y ) \otimes X ) & = + \mathcal{D} (x) \sum_{a,b} (-1)^{|x_b| (|f| + |Y^b|) + |x_a| ( |f| + |Y| + |X^a| )} \partial_a  \partial_b (f Y^b  X^a) + \nonumber \\
& - \mathcal{D} (x) \sum_{a,b} (-1)^{ |x_a| ( |f| + |Y| + |X^a| )} \partial_a  (f Y^b  \partial_b  (X^a)).
\end{align}
Let us now consider the full expression
\begin{align}
& \Delta^{\mathcal{B}er}_\mathpzc{R} ( \Delta^{\mathcal{B}er}_\mathpzc{R} (\mathcal{D} \otimes X ) \otimes Y ) - (-1)^{|X||Y|} \Delta^{\mathcal{B}er}_\mathpzc{R} (\Delta^{\mathcal{B}er}_\mathpzc{R} (\mathcal{D} \otimes Y) \otimes X)  =  \nonumber \\
& = - \mathcal{D} (x) \sum_{a, b} \left ( (-1)^{|x_a| ( |f| + |X| + |Y^b| )} \partial_a (f X^b \partial_b (Y^a)) 
- (-1)^{|X||Y|+  |x_a| ( |f| + |Y| + |X^a| )} \partial_a  (f Y^b  \partial_b  (X^a)) \right ) + \nonumber \\
& \quad + \mathcal{D} (x) \sum_{a,b} \Big ( (-1)^{|x_b| (|f| + |X^b|) + |x_a| ( |f| + |X| + |Y^a| )} \partial_a  \partial_b (f X^b  Y^a) + \nonumber \\ 
& \quad - (-1)^{|X||Y| + |x_b| (|f| + |Y^b|) + |x_a| ( |f| + |Y| + |X^a| )} \partial_a  \partial_b (f Y^b  X^a)  \Big ). \nonumber 
\end{align}
We need the last two terms to cancel pairwise. We observe that
\bear
\partial_a  \partial_b (f Y^b  X^a) = (-1)^{|x_a| |x_b| + |X^a | |Y^b|}\partial_b \partial_a ( fX^a Y^a). 
\eear
Renaming the indexes in the last addendum one finds that 
\begin{align} \label{zero}
 \mathcal{D} (x) & \sum_{a,b}  \partial_a \partial_b (f X^b Y^a )  \Big ( (-1)^{|x_b| (|f| + |X^b|) + |x_a| ( |f| + |X| + |Y^a| )} + \nonumber \\
 & \quad - (-1)^{|X||Y|+ |x_a| (|f| + |Y^a|) + |x_b| ( |f| + |Y| + |X^b| ) + |x_b||x_a| + |X^b||Y^a|}  \Big ).
\end{align}
This leads us to consider the following equality
\begin{align}
|x_b| (|f| + |X^b|) & + |x_a| ( |f| + |X| + |Y^a| ) = \nonumber \\
&= |X| || Y| + x_b| ( |f| + |Y| + |X^b| ) + |x_a| (|f| + |Y^a|) + |x_b||x_a| + |X^b||Y^a|,
\end{align}
which is easily verified by rearranging the right-hand side. 
This tells that term in \eqref{zero} is identically zero, thus proving flatness of $\Delta_\mathpzc{R}^{\mathcal{B}er}$ and concluding the proof of the first point.\\
\noindent For the second point, uniqueness of $\Delta_\mathpzc{R}^{\mathcal{B}er} $ follows from the fact that $\mathcal{D}(x)$ is a local generator for $\mathcal{B}er (\mani)$ and that the $\partial_a$'s give a system of local generators for $\mathcal{T}_\mani$. Also, choosing a system of coordinates, 
$\Delta_\mathpzc{R}^{\mathcal{B}er} (\mathcal{D}(x) \otimes \partial_a) = - \mathfrak{L}_{\partial_a} (\mathcal{D}(x)) = 0$
is readily verified simply applying \eqref{DressedLie}. That this remains zero in any system of coordinates $x^\prime = x^\prime (x)$ is a (lengthy) local check, which is carried out in \cite{Manin}.\end{proof}

{\remark The above result establishes that there exists a right action of vector fields on the Berezinian sheaf of a supermanifold given - up to a sign - by the Lie derivative 
\bear
\xymatrix@R=1.5pt{
\mathcal{B}er (\mani) \otimes \mathcal{T}_\mani \ar[r] & \mathcal{B}er (\mani) \\
\mathcal{D} \otimes X \ar@{|->}[r] & \mathcal{D} \cdot X \defeq - (-1)^{|\mathcal{D}||X| }\mathcal{L}_X (\mathcal{D}),
}
\eear
so that, explicitly, using \eqref{DressedLie}, one has 
\bear \label{ActionLie}
\mathcal{D} \cdot X = - \mathcal{D}(x)\sum_a (-1)^{|x_a| (|X^a| + |f|)}  \partial_a (fX^a),
\eear
if $\mathcal{D}$ is trivialized as $\mathcal{D} (x) f$ for a function $f$. It is worth noticing that the above construction can be obtained in a rather more neat and economic way in a purely algebraic fashion, by using the third construction of the Berezinian sheaf that we have discussed in the previous section in theorem \ref{Con3}, which introduces the sheaf $\mathcal{D}_\mani$ from the very beginning. In particular, one proves the following.}
\begin{corollary}[$\mathcal{B}er (\mani)$ is a Right $\mathcal{D}_\mani$-Module - Cohomological Version] \label{BerRightCor} Let $\mani$ be a real or complex supermanifold. Then right action 
\bear
\mathcal{H}^p \left ( (\mathpzc{DR}^\bullet_\mani, \mathscr{D} ) \right ) \otimes_{\stsheaf } \mathcal{D}_\mani \longrightarrow \mathcal{H}^p \left ( (\mathpzc{DR}^\bullet_\mani, \mathscr{D} ) \right )
\eear
is uniquely characterized by the condition $\mathcal{D}(x) \cdot \partial_a = 0 $ for any $a$ even and odd, where $\mathcal{D}(x) \in \mathcal{H}^p \left ( (\mathpzc{DR}^\bullet_\mani, \mathscr{D} ) \right ) \cong \mathcal{B}er (\mani)$ is a generating section of $\mathcal{B}er(\mani)$. More in particular, the action is given (up to an overall sign) by the Lie derivative on $\mathcal{B}er (\mani).$
\end{corollary}
\begin{proof} With reference to Theorem \ref{Con3}, one has that in cohomology $\mathcal{D}(x) \cdot \partial_a = [dz_1 \ldots dz_p \otimes \partial_{\theta_1} \otimes \ldots \partial_{\theta_q} \partial_a ] = 0 $ for any $a$, which characterizes the right action of $\mathcal{D}_\mani$ on $\mathcal{B}er (\mani).$ \\
In particular, one sees that considering a section $\mathcal{D} = dz_1 \ldots dz_p \otimes \partial_{\theta_1}  \ldots \partial_{\theta_q}  f \in \mathcal{B}er(\mani)$ and a vector fields $X = \sum_a X^a \partial_a $, one finds 
\begin{align}
\mathcal{D} \cdot X & = \mathcal{D}(x) \sum_a (-1)^{|x_a| (|X^a| + |f|)} \left ( - \partial_a (fX^a) + \partial_a \cdot f X^a \right ) \nonumber \\
& = -\mathcal{D}(x) \sum_a (-1)^{|x_a| (|X^a| + |f|)} \partial_a (f X^a),
\end{align} 
where we have used that $\mathcal{D}(x) \cdot \partial_{x_a} = 0$ is zero in the cohomology. This matches \eqref{ActionLie}.
\end{proof}
\noindent This shows the relevance of the construction of the Berezinian via the total de Rham algebra $\Omega^\bullet \otimes_{\mathcal{O}_{\mani}} \mathcal{D}_\mani.$ 
{\remark Finally, it is worth observing the similarities between these results and constructions and those elucidated in appendix \ref{app2} regarding the right $\mathcal{D}_X$-module structure and its relation with the Lie derivative for the canonical sheaf of an ordinary manifold $X$.}

\section{Integral Forms and Spencer Cohomology of Supermanifolds}

\noindent In the previous section we have shown that, given a real or complex supermanifold $\mani$, there exists a flat right connection which endows $\mathcal{B}er (\mani)$ with the structure of a right $\mathcal{D}_\mani$-module. In analogy with the ordinary case of the canonical sheaf of a real or complex manifold, we want to study the structure of the so-called \emph{Spencer complex} related to $\mathcal{B}er(\mani)$ \cite{Sabbah}. This is particularly relevant in the context of supermanifolds since the Berezinian sheaf does not appear in the de Rham complex - this is instead the case for the canonical sheaf on ordinary manifolds - thus giving rise to a genuinely new complex. Also, as we shall see shortly, sections of the Berezinian sheaf are the objects to look at for a meaningful notion of integration over a supermanifold. We start with the main definition \cite{BL1, BL2, Manin, Deligne}.
\begin{definition}[Integral Forms] Let $\mani$ be a real or complex supermanifold of dimension $p|q$. We call integral forms of degree $p-i$ the sections of the sheaves 
\bear
\Sigma^{p-i}_\mani \defeq \mathcal{B}er (\mani) \otimes_{\stsheaf} \cat{S}^i \Pi \mathcal{T}_\mani 
\eear
for any $i\geq 0$. 
\end{definition}
\noindent Notice that, equivalently, one can define integral forms of degree $p-i$ to be sections of the sheaves $\Sigma^{p-i}_\mani \defeq \mathcal{H}om_{\mathcal{O}_\mani} (\Omega^{i}_\mani, \mathcal{B}er (\mani)).$
{\remark The degree is assigned so that sections of the Berezinian sheaf $\Sigma_\mani^p = \mathcal{B}er (\mani)$ are \emph{top-integral forms} in degree $p$, which equals the even dimension of the supermanifold, mimicking what happens for the canonical sheaf in the ordinary setting. Further, it is to be stressed that integral forms can have any \emph{negative} degree: this is in some sense \virgolette dual'' to what happens in the case of differential forms on a supermanifold, as seen above.}\\

\noindent We are now interested in structuring the sheaves of integral forms into a complex: to this end, we need to introduce a suitable differential. This is where the (flat) right connection $\Delta^{\mathcal{B}er}_{\mathpzc{R}} : \mathcal{B}er (\mani) \otimes \mathcal{T}_\mani \rightarrow \mathcal{B}er (\mani)$ discussed in the previous section serves its purposes \cite{Manin, Penkov}. In particular, we define a morphism of sheaves as follows.
\bear \label{delta1}
\xymatrix@R=1.5pt{
\delta : \Sigma^{p-1} = \mathcal{B}er (\mani) \otimes \Pi \mathcal{T}_\mani \ar[r] & \Sigma^p = \mathcal{B}er (\mani) \\ 
\mathcal{D} \otimes \pi X \ar@{|->}[r] & \delta (\mathcal{D} \otimes \pi X) \defeq (-1)^{|\mathcal{D}| + |X|} \Delta_{\mathpzc{R}}^{\mathcal{B}er} (\mathcal{D}\otimes X).
}
\eear 
\noindent By the previous section this is well-defined and it does not depend on the choice of local coordinates. Also, notice that it is $\mathcal{O}_\mani$-linear since 
\bear
\delta (\mathcal{D}f \otimes \pi X ) = (-1)^{|D| + |f| + |X|} \Delta_{\mathpzc{R}}^{\mathcal{B}er} (\mathcal{D} f \otimes X) = (-1)^{|D| + |f| + |X|} \Delta_{\mathpzc{R}}^{\mathcal{B}er} (\mathcal{D}  \otimes f X) = \delta (\mathcal{D} \otimes f \pi X),
\eear
by the property of the right connection $\Delta^{\mathcal{B}er}_\mathpzc{R}$. Moreover, since $\Delta_{\mathpzc{R}}^{\mathcal{B}er} (- \otimes X) = - \mathfrak{L}_X (-) $, the above can also be rewritten as
\begin{align}
\delta (\mathcal{D} \otimes \pi X) & = (-1)^{|\mathcal{D}| + |X| +1 } \mathfrak{L}_X (\mathcal{D}) =  (-1)^{|\mathcal{D}| + |\pi X|} \mathfrak{L}_X (\mathcal{D}), 
\end{align}
which stress the dependence from the parity of the $\Pi$-vector field $\pi X \in \Pi \mathcal{T}_\mani.$ Using the explicit expression of the Lie derivative \eqref{ActionLie}, in local coordinates, for $\mathcal{D} = \mathcal{D}(x) f$ with $f$ a section of the structure sheaf and $\pi X = \sum_a X^a \pi \partial_a$ the action of $\delta$ is therefore given by
\begin{align} \label{actiondelta}
\delta (\mathcal{D} \otimes \pi X) & = (-1)^{|X| + |\mathcal{D}|} \Delta^{\mathcal{B}er}_{\mathpzc{R}} (\mathcal{D} \otimes X)  = (-1)^{|X| + |\mathcal{D}| +1} \mathfrak{L}_X (\mathcal{D}) \nonumber \\
& = (-1)^{|X| + |\mathcal{D}| +1} \mathcal{D}(x) \sum_a (-1)^{|x_a| (|f| + |X^a|)} \partial_a (f X^a),
\end{align}
where we have used \eqref{ActionLie}.
{\remark In order to extend the above morphism to integral forms of any degree, in a similar fashion as in ordinary differential geometry, we introduce the \emph{contraction operator} $\iota_\omega : \cat{S}^\bullet \Pi \mathcal{T}_\mani \rightarrow \cat{S}^\bullet\Pi \mathcal{T}_\mani$ where $\omega \in \Omega^1_\mani$, which is a derivation of the supersymmetric algebra $\cat{S}^\bullet \Pi \mathcal{T}_\mani$ characterized by the properties that $\iota_{\omega } (f) = 0$ for any $f  \in \stsheaf $ and $\iota_{\omega} (\pi X) = \omega( \pi X)$ for any $\pi X \in \Pi \mathcal{T}_\mani$. In particular, on the bases of $\Omega^1_\mani$ and $\Pi \mathcal{T}_\mani$ we set $\iota_{dx_a} (\pi \partial_{b}) = \delta_{ab}$. We will employ the notation $\iota_{dx_a} = \partial_{\pi \partial_a}$ as to stress that $\iota_{dx_a}$ acts indeed as a derivation. In a certain trivialization, for sections $\omega \in \Omega^1_\mani$ and $\pi X \in \Pi \mathcal{T}_\mani$ one has that
\bear
\iota_\omega (\pi X) = \sum_{a} \omega_a \partial_{\pi \partial_{a}} \left ( \sum_b X^b \pi \partial_{b} \right ) = \sum_{a}(-1)^{(|x_a | + 1)|X^a|} \omega_a X^a.
\eear
Using this, we prove the following lemma. }
\begin{lemma}[Representation Lemma] Let $\mani$ be a real or complex supermanifold. Then the action of $\delta : \mathcal{B}er (\mani) \otimes \Pi \mathcal{T}_\mani \rightarrow \mathcal{B}er (\mani)$ defined as in \eqref{delta1} can be given the following operator representation
\bear \label{represent}
\delta = \sum_a (-1)^{|x_a| +1 } \mathcal{L}_{\partial_{a}} \otimes \iota_{dx_a},
\eear
where $\mathcal{L}_{\partial_{a}}$ is the Lie derivative with respect to the coordinate vector field $\partial_{x_a}$ acting on sections of the Berezinian sheaf $\mathcal{B}er (\mani) $ and $\iota_{dx_a}$ is the contraction with respect to the coordinate 1-form $dx_a$ acting on sections of the $\Pi$-tangent sheaf $\Pi \mathcal{T}_\mani.$
\end{lemma}
\begin{proof} Let us choose a chart with $\pi X = \sum_a X^a \pi \partial_a$ and $\mathcal{D} = \mathcal{D} (x) f$. Then one computes 
\begin{align}
\left ( \sum_a (-1)^{|x_a| + 1} \mathcal{L}_{\partial_a} \otimes \iota_{dx_a} \right ) &  \left (\mathcal{D}(x) f \otimes \sum_b X^b \pi \partial_b \right )  \nonumber \\ 
& = \sum_{a,b} \mathcal{L}_{\partial_a} (\mathcal{D}(x)f X^b) (-1)^{(|\mathcal{D}| + |X^b|+1)(|x_a| + 1)} \otimes  \iota_{dx_a}(\pi \partial_b) \nonumber \\
& = \mathcal{D}(x)\sum_a (-1)^{(|\mathcal{D}| + |X^a| + 1)(|x_a| + 1)} (-1)^{|x_a| (|\mathcal{D}| + |X^a|)} (-1)^{|x_a|(|f| + |X^a|)}\partial_a ( fX^a) \nonumber \\
& = \mathcal{D}(x) \sum_a (-1)^{|\mathcal{D}| + |X| + 1} (-1)^{|x_a| (|f| + |X^a|)} \partial_a (f X^a),
\end{align}
which matches $\delta (\mathcal{D} \otimes \pi X) $ as in \eqref{actiondelta}.
\end{proof}
{\remark The above lemma is convenient to extend the action of $\delta$ to integral forms of any degree. More precisely, it can be proved that $\delta$ extends to a \emph{derivation} $\delta : \Sigma^{p-i}_\mani \rightarrow \Sigma_\mani^{p-i+1}$ for any $i \geq 1$, where $p$ is the even dimension of the supermanifold, in the sense that 
\bear \label{leibint}
\delta (\mathcal{D} \otimes \pi X^{(n)} \pi X^{(m)}) = \delta (\mathcal{D} \otimes \pi X^{(n)} )\,  \pi X^{(m)} + (-1)^{|\pi X^{(n)}| |\pi X^{(m)}|} \delta (\mathcal{D} \otimes \pi X^{(m)} ) \, \pi X^{(n)}
\eear 
for any $\pi X^{(n)} \in \cat{S}^n \Pi \mathcal{T}_\mani$ and $\pi X^{(m)} \in \cat{S}^m \Pi \mathcal{T}_\mani$ for any $m, n \geq 1,$ and where we have the supersymmetric product understood for notational convenience. More in particular, using \eqref{represent}, one computes
\begin{align}
\delta (\mathcal{D} \otimes \pi X \, \pi Y) & 
 =  \left ( \sum_{a,b} (-1)^{(|\mathcal{D}| + |X^b| +1)(|x_a | +1)} \mathcal{L}_{\partial_a}  \left ( \mathcal{D} X^a \right ) \otimes \iota_{dx_a} (\pi \partial_b) \right ) \left ( \sum_c Y^c \pi \partial_c \right ) \, + \nonumber \\
& \quad + (-1)^{|\pi X| |\pi X|} \left ( \sum_{a,c} (-1)^{ (|\mathcal{D}| + |Y^c| +1 )(|x_a | +1)} \mathcal{L}_{\partial_a}  \left ( \mathcal{D} Y^a \right ) \otimes \iota_{dx_a} (\pi \partial_c) \right ) \left ( \sum_b X^b \pi \partial_b \right ) \nonumber \\ 
& =  \left ( \sum_{a} (-1)^{(|\mathcal{D}| + |X^a| +1)(|x_a | +1)} \mathcal{L}_{\partial_a}  \left ( \mathcal{D} X^a \right ) \right )  \left ( \sum_c Y^c \pi \partial_c \right ) \, + \nonumber \\
&  \quad + (-1)^{|\pi X| |\pi X|} \left (\sum_{a} (-1)^{ (|\mathcal{D}| + |Y^c| +1)(|x_a | +1)} \mathcal{L}_{\partial_a}  \left ( \mathcal{D} Y^a \right ) \right )  \left (  \sum_b X^b \pi \partial_b  \right ) \nonumber \\ 
& = \delta (\mathcal{D}  \otimes \pi X)  \, \pi Y + (-1)^{|\pi X| |\pi Y|} \delta (\mathcal{D}  \otimes \pi Y)  \, \pi X.
\end{align}
This proves that Leibniz formula in the form \eqref{leibint} holds true, so that $\delta : \Sigma_\mani^{p-1} \rightarrow \Sigma_\mani^p$ is indeed a derivation. Working by induction on the degree of the integral forms one gets to the conclusion. Notice that this shows that $\delta : \Sigma_\mani^{p-1} \rightarrow \Sigma_\mani^{p} $ indeed extends to a derivation $\delta : \Sigma_\mani^{p-i} \rightarrow \Sigma_\mani^{p-i+1}$ for any $i \geq 0$ and that it is globally well-defined, \emph{i.e.}\ it does not depends on the choice of coordinates. We can thus prove the following lemma.}
\begin{lemma} \label{dgsm} The pair $(\Sigma^\bullet_\mani, \delta)$ defines a differential graded supermodule (DGsM).
\end{lemma}
\begin{proof} We are only left to prove that $\delta$ is nilpotent, \emph{i.e.} $\delta^2 = 0$. First we note that, for any $a$ even and odd, one has that $\mathcal{L}_{\partial_a}$ and $\iota_{dx_a}$ have opposite parity so that $\delta$ is odd. Let us prove that $\mathcal{L}_{\partial_a}$ and $\iota_{dx^a}$ satisfy the \emph{same} commutation relations, in particular they (super)commute pairwise. On the one hand, clearly $[\iota_{dx^a}, \iota_{dx^b}] = 0$. On the other hand
\begin{align}
[\mathcal{L}_{\partial_a}, \mathcal{L}_{\partial_b}] \mathcal{D} & = \mathcal{L}_{\partial_a} \mathcal{L}_{\partial_b} (\mathcal{D}) - (-1)^{|x_a||x_b|} \mathcal{L}_{\partial_b} \mathcal{L}_{\partial_a} (\mathcal{D}) \nonumber \\
& = (-1)^{|\mathcal{D}||x_b| + 1} \mathcal{L}_{\partial_a} (\Delta^{\mathcal{B}er}_\mathpzc{R}(\mathcal{D} \otimes \partial_b)) - (-1)^{|x_a||x_b|} (-1)^{|\mathcal{D}||x_a| + 1} \mathcal{L}_{\partial_b}(\Delta^{\mathcal{B}er}_\mathpzc{R}(\mathcal{D} \otimes \partial_b)) \nonumber \\
& = (-1)^{|\mathcal{D}|(|x_a| + |x_b|) + |x_a||x_b|} \left ( \Delta^{\mathcal{B}er}_\mathpzc{R} (\mathcal{D}_R^{\mathcal{B}er}(\mathcal{D} \otimes \partial_b) \otimes \partial_a ) - (-1)^{|x_a||x_b|}  ( \Delta^{\mathcal{B}er}_\mathpzc{R} (\mathcal{D}_\mathpzc{R}^{\mathcal{B}er}(\mathcal{D} \otimes \partial_a) \otimes \partial_b )\right ) \nonumber \\
& = (-1)^{|\mathcal{D}|(|x_a| + |x_b|) + |x_a||x_b|} \Delta^{\mathcal{B}er}_{\mathpzc{R}} (\mathcal{D}\otimes [\partial_a , \partial_b]) = 0,
 \end{align}
where we have used the flatness of the right connection in the semi-last equality and that $[\partial_a , \partial_b] = 0$. It follows from Lemma \ref{lemmaA1} that $\delta$ is nilpotent. 
\end{proof}
\noindent Thanks to the previous lemma \ref{dgsm} we can now give the following definition.
\begin{definition}[Spencer Complex / Complex of Integral Forms of $\mani$] We call the differential graded supermodules $(\Sigma_\mani^\bullet, \delta)$ the Spencer complex of $\mani$ or complex of integral forms of $\mani$ 
\bear
\xymatrix{
\ldots \ar[r]  & \Sigma^{p-n}_\mani \ar[r]^{\delta } & \ldots \ar[r]^{\delta} & \Sigma_\mani^{p-1} \ar[r]^{ \; d} & \Sigma^p_\mani \ar[r]^{ \; d} &0,
}
\eear
where $p$ is the even dimension of $\mani$.
\end{definition}
{\remark It is important to stress the following difference between differential and integral forms - which emerges also from the previous lemma, if compared to the analogous lemma \ref{dgsa} for differential forms. Integral forms are not structured into a sheaf of \emph{superalgebras}, \emph{i.e.}\ it does not make sense to multiply two integral forms. Indeed, if on the one hand the supersymmetric product of two $\Pi$-polyfields yields a $\Pi$-polyfields since $\cat{S}^\bullet \Pi \mathcal{T}_\mani$ is a sheaf of superalgebras, the multiplication of two sections of the Berezinian sheaf $\mathcal{B}er(\mani)$ is not a section of the Berezinian sheaf, but instead a section of $\mathcal{B}er(\mani)^{\otimes 2}$, which never appears in the definition of the sheaves $\Sigma^{p-i}_\mani$.} 
{\remark Also, it is crucial to observe the difference between the \emph{de Rham complex} / complex of differential forms and the \emph{Spencer complex} / complex of integral forms on a supermanifold $\mani$. Whereas the first one is not bounded from \emph{above}, the second one is not bounded from \emph{below} instead. 
\bear
\xymatrix@R=10pt{
& 0 \ar[r] & \Omega^0_\mani  \ar[r]\ar@{.}[d] & \Omega^1_{\mani}  \ar[r]\ar@{.}[d] & \ldots \ar[r]  & \Omega^n_{\mani} \ar[r]\ar@{.}[d] & \Omega^{n+1}_\mani \ar[r] & \ldots \\
\ldots \ar[r] & \Sigma_\mani^{-1} \ar[r] & \Sigma^0_\mani \ar[r] & \Sigma^1_\mani \ar[r] & \ldots \ar[r] & \Sigma^p_\mani \ar[r] & 0.
}
\eear
\noindent We now state the analog of the Poincaré lemma in the context of integral forms. This appears without proof in \cite{Manin}. The following proof is adapted from the very recent \cite{CNR}. }
\begin{theorem}[Poincaré Lemma for Integral Forms] \label{PoincLemmInt} Let $\mani$ be a real supermanifold and let $(\Sigma^\bullet_\mani, \delta)$ be the Spencer complex of $\mani$. Then one has 
\bear
H^k_{\delta} (\Sigma^\bullet_\mani) \cong \left \{ \begin{array}{lll}
\mathbb{R}_\mani & & k = 0\\
0 & & k >0.
\end{array}
\right.
\eear
where $\mathbb{R}_\mani$ is the sheaf of locally constant function on $\mani$. In particular, $H^0_{\delta} (\Sigma^{p-\bullet}_\mani)$ is generated by the section $\sigma_0 = \mathcal{D}(x) \, \theta_1 \ldots \theta_q \otimes \pi \partial_{x_1} \ldots \pi \partial_{x_p}$, where $\mathcal{D}(x)$ is a generating section of the Berezinian sheaf and $\pi \partial_{x_1} \ldots \pi \partial_{x_p} $ is the totally anti-symmetric $\Pi$-polyvector field in $\cat{\emph{S}}^p \Pi \mathcal{T}_\mani.$ 
\end{theorem}
\begin{proof} We show for any $k \neq 0$ a homotopy 
$h^k : \mathcal{B}er (\mani) \otimes \cat{S}^{p - k} \Pi \mathcal{T}_\mani  \rightarrow \mathcal{B}er (\mani) \otimes \cat{S}^{p-k-1} \Pi \mathcal{T}_\mani$ for $\delta$, \emph{i.e.}\ a map such that
\bear 
h^{k+1} \circ \delta^{k} + \delta^{k-1} \circ h^k = id_{\mathcal{B}er (\mani) \otimes \cat{\scriptsize{S}}^{p-k} \Pi \mathcal{T}_\mani},
\eear where we have specified the degree and where the maps go as follows
\bear
\xymatrix{
\cdots \ar[r] &  \mathcal{B}er (\mani) \otimes \cat{S}^{p - k +1} \Pi \mathcal{T}_\mani \ar[r] &  \mathcal{B}er (\mani) \otimes \cat{S}^{p - k} \Pi \mathcal{T}_\mani \ar@{-->}[d]_{id}\ar[dl]_{h^k} \ar[r]^{\delta^k} &  \mathcal{B}er (\mani) \otimes \cat{S}^{p - k-1} \Pi \mathcal{T}_\mani \ar[dl]^{h^{k+1}} \ar[r] & \cdots \\
\cdots \ar[r] &  \mathcal{B}er (\mani) \otimes \cat{S}^{p - k+1} \Pi \mathcal{T}_\mani \ar[r]_{\delta^{k-1}} &  \mathcal{B}er (\mani) \otimes \cat{S}^{p - k} \Pi \mathcal{T}_\mani \ar[r] &  \mathcal{B}er (\mani) \otimes \cat{S}^{p - k-1} \Pi \mathcal{T}_\mani \ar[r] & \cdots. \\
}
\eear
Working locally on $\mani$, we define $x_a \defeq z_1, \ldots, z_p | \theta_1, \ldots, \theta_q$ the even and odd local coordinates and we introduce a parameter $t \in [0,1]$. Working as in the above Poincaré lemma for differential forms, we consider the map $(t , x_a ) \stackrel{G}{\longmapsto} t x_a,$ so that in turn one has  $f(x_a) \stackrel{G^\ast}{\longmapsto} f(tx_a)$ on sections of the structure sheaf $\stsheaf.$
Writing $G$ as a family of maps as $G_t : \mani \rightarrow \mani$, we write in turn $G^\ast_t : \mathcal{O}_\mani \rightarrow \mathcal{O}_\mani$ for the corresponding map on $\mathcal{O}_\mani.$ \\
We claim that the homotopy is given by
\bear
h^k (\mathcal{D}(x) f \otimes \pi X) \defeq (-1)^{|f| + |\pi X|} \mathcal{D}(x) \, \sum_b (-1)^{|f| (|x_b| + 1)} \left ( \int_0^1 dt \, t^{K_\mathfrak{s}} x_b G^\ast_t f \right ) \otimes \pi \partial_{b} (\pi X), 
\eear
where $\mathcal{D}(x)$ is a generating section of the Berezinian, $f $ is a section of the structure sheaf and $\pi X$ is a $\Pi$-polyfield of degree $k$ in the form $\pi X = \pi \partial^I$ for some multi-index $|I| = k$ and $K_\mathfrak{s}$ is a constant, dependent on the integral form $\mathfrak{s} = \mathcal{D} (x) f \otimes \pi X$ chosen which will be fixed in what follows.\\
On the one hand we have
\begin{align} \label{firstrow}
h^{k+1} \circ \delta^k (\mathcal{D}(x) f \otimes \pi X) 
& = \mathcal{D}(x) \sum_{a,b} (-1)^{(|x_a| + |x_b|)(|f| + |x_a|)} \left ( \int_0^1 dt \, t^{K_{\delta \mathfrak{s}}} x_b G^\ast_t ( \partial_a f ) \right ) \otimes \pi \partial_b \cdot  \partial_{\pi \partial_a} \pi X.
\end{align}
On the other hand, the action of $\delta^{k-1}\circ h^k$ is more complicated, namely made out of four summands:
\begin{align}
\delta^{k-1}\circ h^k  (\mathcal{D}(x) f \otimes \pi X) 
& = + \mathcal{D}(x) \sum_a \int^1_0 dt \, t^{K_\mathfrak{s}} G^\ast_t f \otimes \pi X  \label{one} \\
& \quad + \mathcal{D}(x) \sum_a (-1)^{|x_a|} \int_0^1 dt \, t^{K_\mathfrak{s}} x_b \partial_{ a} G^\ast_t f \otimes \pi X  \label{two} \\
& \quad + \mathcal{D}(x) \sum_a (-1)^{|x_a| + 1} \int_0^1 dt \, t^{K_\mathfrak{s}} G^\ast_t f \otimes \pi \partial_a \cdot \partial_{\pi \partial_a} \pi X \label{three} \\ 
& \quad - \mathcal{D}(x) \sum_{a,b} (-1)^{(|f| + |x_a|) (|x_a| + |x_b|)} \int^1_0 dt \, t^{K_\mathfrak{s}} x_b {\partial}_a (G^\ast_t f) \otimes \pi \partial_b \cdot \partial_{\pi \partial_a } \pi X. \label{four}
\end{align}
For the last line \eqref{four} to cancel \eqref{firstrow} we need $K_{\delta \mathfrak{s}} = Q_{\mathfrak{s}} + 1,$ by chain-rule.
For the summand \eqref{one} it is immediate to observe 
\bear
\mathcal{D}(x) \sum_a \int^1_0 dt \, t^{K_\mathfrak{s}} G^\ast_t f \otimes \pi X = (p+q)  \mathcal{D}(x) \left (\int^1_0 dt \, t^{K_\mathfrak{s}} G^\ast_t f \right ) \otimes \pi X.  
\eear
For the summand \eqref{two}, assuming without loss of generality that $f$ is homogeneous of degree $\deg_{\theta} (f)$ in the theta's we have
\begin{align}
\mathcal{D}(x) \sum_a (-1)^{|x_a|} \int_0^1 dt \, t^{K_\mathfrak{s}} x_b & \partial_{ a} G^\ast_t f \otimes \pi X   = 
 \mathcal{D}(x) f \otimes \pi X - \delta_{K_\mathfrak{s} + 1 + \deg_{\theta} (f), 0 } \, ( \mathcal{D}(x) f (0) \otimes \pi X )+ \nonumber \\ 
& \quad - \left ( K_\mathfrak{s} + 1 + 2 \deg_{\theta} (f) \right ) \mathcal{D}(x) \left ( \int_0^1 dt\, t^{K_\mathfrak{s}} G^\ast_t f \right ) \otimes \pi X,
\end{align}
upon integration by parts.
For the summand \eqref{three}, we define $\deg_{\pi \partial_\theta} (\pi X)$ and $\deg_{\pi \partial_z} (\pi X)$ the degree of $\pi X$ in its even and odd monomials. We have
\begin{align}
\mathcal{D}(x) \sum_a (-1)^{|x_a| + 1} \int_0^1 dt \, &t^{K_\mathfrak{s}}  G^\ast_t f \otimes  \pi \partial_a \cdot \partial_{\pi \partial_a} \pi X =  \nonumber \\
& = \left (\deg_{\pi \partial_\theta} (\pi X) - \deg_{\pi \partial_z} (\pi X) \right ) \mathcal{D}(x) \left ( \int_0^1 dt\, t^{K_\mathfrak{s} } G^{\ast}_t f \right ) \otimes \pi X.  
\end{align}
Altogether one gets
\begin{align}
(\delta^{k-1} \circ & h^k + h^{k+1}  \circ \delta^k )(\mathcal{D}(x) f \otimes \pi X)   = \mathcal{D}(x) f \otimes \pi X - \delta_{K_\mathfrak{s} +1 + \deg_{\theta} (f), 0} \,\varphi \, f(0) \otimes \pi X + \nonumber \\
& + \left (p+q + \deg_{\pi \partial_\theta} (\pi X) - \deg_{\pi \partial_z } (\pi X) - 2\deg_{\theta} (f) - K_\mathfrak{s} - 1 \right ) \mathcal{D}(x) \int_0^1 dt\, t^K_{\mathfrak{s}} G^{\ast}_t f \otimes \pi X.
\end{align}
This implies that in order to have a homotopy $K_\mathfrak{s}$ must be such that 
\bear
K_\mathfrak{s} = p+q + \deg_{\pi \partial_\theta} (\pi X ) - \deg_{\pi \partial_z } (\pi X) - 2 \deg_{\theta} (f) -1,
\eear
so that one is led to consider 
\begin{align}
(\delta^{k-1} \circ & h^k + h^{k+1}  \circ \delta^k )(\mathcal{D}(x) f \otimes \pi X)  = \mathcal{D}(x) f \otimes \pi X  + \nonumber\\
& \quad - \delta_{(p+q + \deg_{\pi \partial_\theta} (\pi X) - \deg_{\pi \partial_z } (\pi X) - \deg_{\theta} (f)) , 0} \,\mathcal{D}(x) \, f(0|\theta) \otimes \pi X.
\end{align}
It is easy to see that the above homotopy fails if $\deg_{\pi \partial_\theta} (F) = 0 $, $\deg_{\pi \partial_z } (F) = p$ and $\deg_{\theta} (f) = q$, so that one identifies the generator of the cohomology in the element $\sigma_0 \defeq  \mathcal{D}(x) \theta_1 \ldots \theta_q \otimes \pi \partial_{z_1} \ldots \pi \partial_{z_p}  $, which is indeed obviously closed under the action of $\delta$ by inspection.
\end{proof}
\noindent This result allows us to compute the cohomology of integral forms and, in turn, to make contact with the cohomology of differential forms. Namely, we give the following definition. 
\begin{definition}[Spencer Cohomology of $\mani$] Let $\mani$ be a real supermanifold. Then we define the Spencer cohomology of $\mani$ to be the cohomology of the \emph{global sections} of the (sheaf of) differentially graded supermodules $(\Sigma^\bullet_{\mani}, \delta)$, \emph{i.e.}
\bear
H_{\mathpzc{Sp}}^k ( \mani) \defeq H^k_\delta (\check{H}^0_{\check\delta} (\Sigma^{\bullet}_\mani) ), 
\eear
where $\check{H}^0_{\check \delta }(\Sigma^\bullet_{\mani})$ is $0$-\v{C}ech cohomology group of $\Sigma^{\bullet}_\mani$, \emph{i.e.}\ the global sections or $\Sigma^{\bullet}_\mani$.
\end{definition}
\noindent Having proved lemma \ref{PoincLemmInt}, the Spencer cohomology of $\mani$ is easily computed in exactly the same fashion as in theorem \ref{qiso}.
\begin{theorem}[Quasi-Isomorphism II] \label{qiso2} Let $\mani$ be a real supermanifold and let $\manir$ be its reduced manifold. Then the Spencer complex of $\mani$ is \emph{quasi-isomorphic} to the de Rham complex of $\manir.$
In particular, one has that 
\bear
H^\bullet_{\mathpzc{Sp}} (\mani) \cong H^\bullet_{\mathpzc{dR}} (\manir).
\eear
\end{theorem} 
\begin{proof} Once again, the theorem is an easy consequence of the \emph{\v{C}ech-to-de Rham spectral sequence} for the double complex $(\Sigma^{\bullet}_{\mani}, \check \delta, d)$, where $\check \delta$ is the \v{C}ech differential and $\delta$ is the integral forms differential. On the one hand, {generalized} Mayer-Vietoris short exact sequence (hence the existence of a partition of unity) and Poincaré lemma yield $H^\bullet_{\mathpzc{dR}} (\manir) \cong \check{H}^\bullet (|\manir|, \mathbb{R}_\mani)$ in the ordinary setting, on the other hand lemma \ref{PoincLemmInt} gives that $H^\bullet_{\mathpzc{Sp}} (\mani) \cong \check{H}^\bullet (|\manir|, \mathbb{R}_\mani )$ in the supergeometric setting, where $\check{H}^\bullet (|\manir|, \mathbb{R}_{\mani})$ is the \v{C}ech cohomology of the sheaf of locally-constant functions $\mathbb{R}_\mani$. It follows that   
\bear
H^\bullet_{\mathpzc{Sp}} (\mani) \cong \check{H}^\bullet (|\manir|, \mathbb{R}_\mani) \cong H^\bullet_{\mathpzc{dR}} (\manir),
\eear
thus concluding the proof.\end{proof} 
\noindent Just like in the case of differential forms, we have in particular the Poincaré lemma for integral forms for the model supermanifold $\mathbb{R}^{p|q}$. 
\begin{theorem}[Poincaré Lemma for $\mathbb{R}^{p|q}$] The Spencer cohomology of the supermanifold $\mathbb{R}^{p|q}$ is given by
\bear
H_{\mathpzc{Sp}}^k (\mathbb{R}^{p|q}) \cong \left \{ \begin{array}{lll}
\mathbb{R} & & k = 0\\
0 & & k >0.
\end{array}
\right.
\eear
\end{theorem}
\begin{proof} Follows immediately from the above \ref{PoincLemmInt} and \ref{qiso2}.
\end{proof}
\noindent The above theorem \ref{qiso2} and theorem \ref{qiso} have an obvious yet important corollary: the cohomology of differential and integral forms compute exactly the same invariants, \emph{i.e.} the de Rham cohomology of the reduced space of the supermanifold. 
\begin{corollary}[$H^\bullet_{\mathpzc{Sp}} (\mani) \cong H^\bullet_{\mathpzc{dR}} (\mani)$] \label{geniso}Let $\mani$ be a real supermanifold. Then the de Rham cohomology of differential forms of $\mani$ is naturally isomorphic to the Spencer cohomology of integral forms of $\mani$, \emph{i.e.}\
\bear
H^\bullet_{\mathpzc{Sp}} (\mani) \cong H^\bullet_{\mathpzc{dR}} (\mani).
\eear
\end{corollary}
\begin{proof} It follows immediately from theorem \ref{qiso} and theorem \ref{qiso2}
\end{proof}
{\remark \label{intcompact} An analogous result hold true for the \emph{compactly suppported} Spencer cohomology. One finds that 
\bear
H^\bullet_{\mathpzc{Sp}, \mathpzc{c}} (\mani) \cong H^\bullet_{\mathpzc{dR}, \mathpzc{c}} (\mani),
\eear 
where $H^\bullet_{\mathpzc{dR}, \mathpzc{c}} (\manir)$ is in turn isomorphic to $H^\bullet_{\mathpzc{dR}, \mathpzc{c}} (\manir).$ In particular, the compactly supported Poincaré lemma for integral forms on $\mathbb{R}^{p|q}$ reads
\bear
H_{\mathpzc{Sp}, \mathpzc{c}}^k (\mathbb{R}^{p|q}) \cong \left \{ \begin{array}{lll}
\mathbb{R} & & k = p\\
0 & & k \neq p.
\end{array}
\right.
\eear
A representative is given by $\mathcal{D}(x) \theta_1 \ldots \theta_q \mathpzc{B}_{\mathpzc{c}} (z_1, \ldots, z_p)$ for a compactly supported bump function $\mathpzc{B}_\mathpzc{c}$ which integrate to one on the reduced space. Compactly supported integral forms will play a crucial role in the next section, when integration on supermanifold will be introduced. }

{\remark As an addendum to the above remark \ref{cohomorem}, the previous corollary \ref{geniso} says that, despite both the complex of differential and integral forms are not bounded either from above or below, their cohomology can indeed be non-zero (and isomorphic) only in the framed part of the diagram below - where the degree of differential forms matches the degree of integral forms - from zero to the even dimension of the supermanifold.
\bear
\xymatrix@R=10pt{
& 0 \ar[r] & \Omega^0_\mani (\mani)  \ar[r]\ar@{.}[d] & \Omega^1_{\mani} (\mani) \ar[r]\ar@{.}[d] & \ldots \ar[r]  & \Omega^n_{\mani} (\mani) \ar[r]\ar@{.}[d] & \Omega^{n+1}_\mani (\mani) \ar[r] & \ldots \\
\ldots \ar[r] & \Sigma_\mani^{-1}(\mani) \ar[r] & \Sigma^0_\mani (\mani) \ar[r] & \Sigma^1_\mani (\mani) \ar[r] & \ldots \ar[r] & \Sigma^p_\mani (\mani) \ar[r] & 0.
\save "1,3"."2,6"*[F]\frm{}
\restore 
}
\eear
Once again, here we have denoted $\Omega^k_\mani(\mani)$ and $\Sigma_\mani^k(\mani)$ the global section of the sheaves of differential and integral forms.}

\section{Berezin Integral and Stokes' Theorem on Supermanifolds}

\noindent In this section we introduce the notion of Berezin integral for real supermanifolds \cite{Berezin, BL1, BL2, Voronov, Witten}, following the philosophy and exposition given in \cite{Manin}, that underlines the role of integral forms and their cohomology as introduced above. In particular, we denote with 
\bear
\mathcal{B}er_{\mathpzc{c}} (\mani) \defeq \Gamma_\mathpzc{c} (\mani, \mathcal{B}er (\mani)),
\eear
the $\mathbb{R}$-module of the \emph{compactly supported} sections of the Berezinian on $\mani$ and, more in general, with 
\bear
\Sigma_{\mani, \mathpzc{c}} \defeq \Gamma_{\mathpzc{c}} (\mani, \Sigma^i_{\mani})
\eear 
the $\mathbb{R}$-module of the \emph{compactly supported} integral forms of degree $i\leq p$ on $\mani$, with $\dim \mani = p|q$. Note that $\Sigma_{\mani, \mathpzc{c}}^p = \mathcal{B}er_{\mathpzc{c}} (\mani),$ in the convention previously set. Further, we assume our supermanifold always has a \emph{finite good cover}. We start with the following preparatory lemma, see \cite{Manin}.
\begin{lemma} \label{prep} Let $\mani$ be a real supermanifold of dimension $p|q$. Then one has the following isomorphism of $\mathbb{R}$-modules  
\bear
\mathcal{B}er_{\mathpzc{c}}(\mani) \cong \mathcal{J}_\mani^q \mathcal{B}er_{\mathpzc{c}} (\mani) + {\delta} (\Sigma^{p-1}_{\mani, \mathpzc{c}}).
\eear
More in particular the intersection $ \mathcal{J}_\mani^q \mathcal{B}er_{\mathpzc{c}} (\mani) \cap {\delta} (\Sigma^{p-1}_{\mani, \mathpzc{c}})$ is such that 
\bear
\varphi ( \mathcal{J}_\mani^q \mathcal{B}er_{\mathpzc{c}} (\mani) \cap {\delta} (\Sigma^{p-1}_{\mani, \mathpzc{c}}) ) = d (\Omega^{p-1}_{\mani, \mathpzc{c}}),
\eear
where $\varphi : \mathcal{J}^q_\mani \mathcal{B}er_{\mathpzc{c}}(\mani) \stackrel{\cong}{\longrightarrow} \Omega^p_{\manir, \mathpzc{c}} $ is the isomorphism of Theorem \ref{bercan}.
\end{lemma}
\begin{proof} Let $\mathcal{D}_\mathpzc{c} \in \mathcal{B}er_\mathpzc{c} (\mani) = \Sigma_{\mani, \mathpzc{c}}^p$ be a section of the Berezinian sheaf having compact support. Then, using a partition of unity $\{ \rho_\kappa \}_{\kappa \in I}$ for $\mani$, we represent $\mathcal{D}$ as a (locally) finite sum
\bear \label{decomp}
\mathcal{D}_{\mathpzc{c}} = \sum_{i = 1}^n \mathcal{D}_{\mathpzc{c}}^{(i)}
\eear
where $\mbox{supp} (\mathcal{D}^{(i)}_{\mathpzc{c}})$ is compact and it is contained in a certain open set $U^{(i)}$ locally described by the coordinate system $x_a = z_1, \ldots, z_p | \theta_1, \ldots, \theta_q$, where the dependence of $i$ is understood. Then one can in turn decompose $\mathcal{D}^{(i)}_{\mathpzc{c}}$ as follows 
\bear
\mathcal{D}^{(i)}_{\mathpzc{c}} (x) = \mathcal{D}^{(i)}_{\mathpzc{c}, 0} (x) + \mathcal{D}^{(i)}_{\mathpzc{c}, 1} (x) \theta_1 \ldots \theta_q,
\eear 
such that every monomial in the $\theta$-expansion of $\mathcal{D}_{\mathpzc{c}, 0} (x)$ contains at most $q-1$ theta's, \emph{i.e.} it is of the form 
\bear
\mathcal{D}^{\underline \epsilon }_{\mathpzc{c}, 0} (x) = \widetilde{\mathcal{D}} (x) \theta_1^{\epsilon_1} \ldots \theta_q^{\epsilon_q} f_{\underline \epsilon, \mathpzc{c}}  (x)
\eear 
for $\epsilon = (\epsilon_1, \ldots, \epsilon_q) $ with $\epsilon_i \in \{ 0, 1\}$ and $|\underline \epsilon| < q$ and $\widetilde{\mathcal{D}} (x) = [dz_1 \ldots dz_p \otimes \partial_{\theta_1}\ldots \partial_{\theta_q}]$ a generating section. It is easy to see that any such section is in the image of $\delta^{p-1}: \Sigma_{\mani, \mathpzc{c}}^{p-1} \rightarrow \Sigma_{\mani, \mathpzc{c}}^{p}$. Indeed, let for example be $\epsilon_j = 0$, then, by definition of $\delta$, one sees that up to sign
\bear
\delta (\widetilde{\mathcal{D}} (x) \theta_1^{\epsilon_1} \ldots \theta_j \ldots \theta_q^{\epsilon_q} f_{\underline \epsilon, \mathpzc{c}}  (x) \otimes \pi \partial_{\pi \theta_j}) = \widetilde{\mathcal{D}} (x) \theta_1^{\epsilon_1} \ldots \theta_{j-1}^{\epsilon_{j-1}} \theta_{j+1}^{\epsilon_{j+1}} \ldots \theta_q^{\epsilon_q} f_{\underline \epsilon, \mathpzc{c}}  (x).
\eear
Notice that since every real supermanifold is split, the $\theta$-degree of the expansion is invariant. It follows that 
\bear
\sum_{i} \mathcal{D}^{(i) }_{\mathpzc{c}, 0} \in \delta (\Sigma_{\mani, \mathpzc{c}}^{p-1}),
\eear
and hence the difference $\mathcal{D}_{\mathpzc{c}} - \sum_{i} \mathcal{D}^{(i) }_{\mathpzc{c}, 0}$ lies in $\mathcal{J}^q_\mani \mathcal{B}er_{\mathpzc{c}}(\mani),$ which proves the first statement. \\
For the second statement, let first be $\mathcal{D}_{\mathpzc{c}} \in \mathcal{J}^q_\mani \mathcal{B}er_{\mathpzc{c}} (\mani) \cap \delta (\Sigma^{p-1}_{\mani, \mathpzc{c}})$, with $\mathcal{D}_{\mathpzc{c}} = \delta \omega_{\mathpzc{c}}$ for some $\omega_{\mathpzc{c}} \in \Sigma^{p-1}_{\mani, \mathpzc{c}}$. Let us decompose $\omega_\mathpzc{c} \in \Sigma^{p-1}_{\mathpzc{c}}$ as in \eqref{decomp}, 
\bear
\omega_{\mathpzc{c}} = \sum_{i} \omega^{(i)}_{\mathpzc{c}}. 
\eear
for $\mbox{supp}(\omega^{(i)}_{\mathpzc{c}})$ compact and contained in some open set $U^{(i)}$ locally described by the coordinate system $x_a = z_1, \ldots, z_p | \theta_1, \ldots, \theta_q$, where again we have left the dependence of $i$ understood. Now, without loss of generality, any $\omega^{(i)}_{\mathpzc{c}}$ can be taken of the form 
\bear
\omega^{(i)}_{\mathpzc{c}} = \mathcal{D}(x) \theta_1 \ldots \theta_q f_{\mathpzc{c}}(z) \otimes \sum_a \pi \partial_{x_a} 
\eear
with $f_{\mathpzc{c}}$ depending on the even coordinates only, since we have seen that any monomial which does not have all the theta's is the image of $\delta $, and since $\delta^2 = 0$ it will not contribute to $\mathcal{D}_{\mathpzc{c}} \in \Sigma^p_{\mani, \mathpzc{c}}$. Further there should be at least one of the $\pi \partial_{z_j}$'s, since otherwise either $\delta \omega^{(i)}_{\mathpzc{c}} \notin \mathcal{J}^q_\mani \mathcal{B}er_\mathpzc{c}(\mani)$ or $\delta \omega^{(i)}_\mathpzc{c} = 0$ in $\mathcal{J}^q_\mani \mathcal{B}er_\mathpzc{c} (\mani)$. But then, one concludes that $\delta \omega^{(i)}_\mathpzc{c} $ should be of the form 
\bear
\delta (\omega^{(i)}_\mathpzc{c}) = \sum_{j=1}^p \mathcal{D}(x) \theta_1 \ldots \theta_q \partial_{z_j} f_{\mathpzc{c}}(z),
\eear
so that if follows 
\bear
\varphi  ( \delta (\omega^{(i)}_\mathpzc{c})  ) = \sum_{j=1}^p dz_1 \ldots dz_p \partial_{z_j} f_{\mathpzc{c}} (z) \in d(\Omega^p_{\manir, \mathpzc{c}}),
\eear
proving that $\varphi (\mathcal{J}_\mani^q \mathcal{B}er_{\mathpzc{c}} (\mani) \cap {\delta} (\Sigma^{p-1}_{\mani, \mathpzc{c}}) \subset d (\Omega^p_{\manir, \mathpzc{c}})$.\\
Viceversa, let us take $\mathcal{D}_\mathpzc{c} \in \mathcal{J}_\mani \mathcal{B}er_{\mathpzc{c}} (\mani)$ such that $\varphi (\mathcal{D}_{\mathpzc{c}}) = d \eta_{\mathpzc{\tiny{red, c}}} $, for $\eta_{\mathpzc{red, c}} \in \Omega^{p-1}_{\manir, \mathpzc{c}}$ and let us prove that $\mathcal{D}_\mathpzc{c} \in \delta (\Sigma^{p-1}_{\mani, \mathpzc{c}})$. To this end, once again we decompose $\omega_{\mathpzc{red, {c}}} $ as follows
\bear
\eta_{\mathpzc{red, {c}}} = \sum_{i} \eta^{(i)}_{\mathpzc{red,{c}}}, 
\eear
as above, where now $\mbox{supp}(\eta^{(i)}_{\mathpzc{red, {c}}})$ is compact and contained in some open set $U^{(i)}$ locally described by the coordinate system $ z_1, \ldots, z_p $ for $\manir,$ so that in these coordinates one has
\bear
\eta^{(i)}_{\mathpzc{red, {c}}}  = \sum_k dz_1 \ldots \widehat{dz_k} \ldots dz_p f^k_{\mathpzc{c}} (z)  
\eear
Accordingly, one can lift $\eta^{(i)}_{\mathpzc{red, {c}}}$ to the integral form in $\Sigma^{p-1}_{\mani, \mathpzc{c}}$ given by
\bear
\omega^{(i)}_{\mathpzc{c}} = \sum_k \mathcal{D}(x) \theta_1 \ldots \theta_q f^k_{\mathpzc{c}} (z) \otimes \pi \partial_{z_k}
\eear
so that up to sign one gets $\varphi (\delta \omega^{(i)}_{\mathpzc{c}}) = d \eta_{\mathpzc{red, {c}}}^{(i)}.$ Notice that the support of $\eta^{(i)}_{\mathpzc{red, {c}}} $ is the same as the one of $\omega^{(i)}_{\mathpzc{c}}$ and that, with abuse of notation, we have denoted the even local coordinates in the same way on $\manir$ and $\mani$. Summing over the index $i$, defining $\omega_{\mathpzc{c}} \defeq \sum_i \omega^{(i)}_{\mathpzc{c}} $, one has that 
\bear
\varphi (\mathcal{D}_\mathpzc{c} - \delta \omega_{\mathpzc{c}}) = \varphi (\mathcal{D}_\mathpzc{c}) - \varphi (\delta \omega_{\mathpzc{c}}) = \varphi (\mathcal{D}_\mathpzc{c}) - d \eta_{\mathpzc{red, {c}}} = 0,
\eear
by hypothesis. Let us set $\widetilde{\mathcal{D}_{\mathpzc{c}}} \defeq \mathcal{D}_\mathpzc{c} - \delta \omega_{\mathpzc{c}}$, choose a partition of unity $\rho_j $ subordinate to the above open sets,  so that it commutes with the isomorphism $\varphi$ and posing so that posing $\widetilde{\mathcal{D}}_{\mathpzc{c}}^{(i)} \defeq \rho_i \widetilde{\mathcal{D}_\mathpzc{c}}$ one gets $\varphi (\widetilde{\mathcal{D}}^{(i)}_\mathpzc{c}) = 0$. Then, it follows that in the above domain $\widetilde{\mathcal{D}}^{(i)}_\mathpzc{c}$ does not contain all of the theta's, otherwise one would get $\varphi (\widetilde{\mathcal{D}}^{(i)}_\mathpzc{c}) \neq 0$. In turn, this implies that $\widetilde{\mathcal{D}}^{(i)}_\mathpzc{c}$ is in the image of $\delta$, \emph{i.e.} there exists $\widetilde{\omega}_{\mathpzc{c}}^{(i)} \in \Sigma_{\mani, \mathpzc{c}}^{p-1}$ such that $\delta (\widetilde{\omega}_{\mathpzc{c}}^{(i)}) = \widetilde{\mathcal{D}}^{(i)}_\mathpzc{c}$ and the support of $\widetilde{\omega}_{\mathpzc{c}}^{(i)}$ is at most the same as the one of $\widetilde{\mathcal{D}}^{(i)}_\mathpzc{c}$. Posing $\widetilde{\omega_\mathpzc{c}} \defeq \sum_i \widetilde{\omega}^{(i)}_\mathpzc{c}$ and summing over $i$, recalling that $\widetilde{\mathcal{D}_{\mathpzc{c}}} = \mathcal{D}_\mathpzc{c} - \delta \omega_{\mathpzc{c}}$, we have $\mathcal{D}_\mathpzc{c} = \delta ( \widetilde \omega_{\mathpzc{c}} + \omega_{\mathpzc{c}})$, proving that $\mathcal{D}_{\mathpzc{c}} \in \delta (\Sigma^{p-1}_{\mani, \mathpzc{c}})$.
\end{proof}
\noindent The previous lemma has the following immediate consequence, which we state as a theorem.
\begin{theorem} \label{isoint} Let $\mani$ be a real supermanifold. Then the following natural isomorphism of $\mathbb{R}$-modules holds true
\bear
\slantone{\mathcal{B}er_{\mathpzc{c}} (\mani)}{\mbox{\emph{Im}} (\delta^{p-1})} \cong \slantone{\Omega_{\manir, \mathpzc{c}}^p}{\mbox{\emph{Im}} (d^{p-1})}.
\eear
where $\delta^{p-1} : \Sigma^{p-1}_{\mani, \mathpzc{c}} \rightarrow \Sigma^p_{\mani, \mathpzc{c}}$ and $d^{p-1}: \Omega^{p-1}_{\mani, \mathpzc{c}} \rightarrow \Omega^{p}_{\mani, \mathpzc{c}}.$
\end{theorem} 
\begin{proof} The isomorphism is induced by the map $\varphi : \mathcal{J}^q_\mani \mathcal{B}er_{\mathpzc{c}} (\mani) \rightarrow \Omega^p_{\mani, \mathpzc{c}}$ and it follows immediately from lemma \ref{prep} since $\mathcal{B}er_{\mathpzc{c}}(\mani)$ decomposes as $\mathcal{B}er_{\mathpzc{c}}(\mani) \cong \mathcal{J}_\mani^q \mathcal{B}er_{\mathpzc{c}} (\mani) + {\delta} (\Sigma^{p-1}_{\mani, \mathpzc{c}})$ and $\varphi ( \mathcal{J}_\mani^q \mathcal{B}er_{\mathpzc{c}} (\mani) \cap {\delta} (\Sigma^{p-1}_{\mani, \mathpzc{c}}) = d (\Omega^{p-1}_{\mani, \mathpzc{c}})$.
\end{proof}
\noindent We can thus finally define the Berezin integral of a section of $\mathcal{B}er_\mathpzc{c}(\mani)$. 
\begin{definition}[Berezin Integral] \label{berezinint} Let $\mani$ be a real supermanifold of dimension $p|q$ such that $\manir$ is oriented and let $x_a = z_1, \ldots , z_p | \theta_1, \ldots, \theta_q$ is a local system of coordinates on an open set $U$. Let $\mathcal{D}_{\mathpzc{c}} \in \mathcal{B}er_\mathpzc{c}(\mani)$ be a compactly supported section of the Berezinian sheaf which reads 
\bear
\mathcal{D}_\mathpzc{c} = \sum_{\underline \epsilon} \mathcal{D} (x) \theta_1^{\epsilon_1} \ldots \theta_q^{\epsilon_q} f^{\underline \epsilon}_\mathpzc{c} (z_1, \ldots, z_p),
\eear
with $\underline \epsilon = (\epsilon_1, \ldots, \epsilon_q)$ for $\epsilon_j = \{0,1 \}$ in the above local coordinates in $(U, x_a)$. Then we define the Berezin integral of $\mathcal{D}_\mathpzc{c}$ as the map 
\bear
\int_\mani : \mathcal{B}er_{\mathpzc{c}} (\mani) \longrightarrow \mathbb{R} \\
\eear
given in the coordinate domain $(U, x_a)$ by 
\bear
\int_{U} \mathcal{D}_\mathpzc{c} \defeq \int_{U_{\mathpzc{red}}} dz_1 \ldots dz_p f^{1\ldots 1}_{\mathpzc{c}} (z_1, \ldots, z_p),
\eear
and we extend the definition to all $\mani$ by additivity via a partition of unity.
\end{definition}
{\remark The above definition is well-given. Indeed, it does not depend on the choice of local coordinates as a consequence of the previous lemma \ref{prep} and theorem \ref{isoint}. Indeed, more invariantly, the Berezin integral is the map given by composition of the isomorphism of theorem \ref{isoint} with the ordinary integral, \emph{i.e.}
\bear \label{compomap}
\xymatrix{
{\mathcal{B}er_{\mathpzc{c}} (\mani)}/{\delta (\Sigma^{p-1}_{\mani, \mathpzc{c}})} \ar[r]^{\widetilde{\varphi}} & {\Omega_{\manir, \mathpzc{c}}^p}/{d(\Omega^{p-1}_{\manir, \mathpzc{c}})} \ar[r]^{\qquad \quad \int} & \mathbb{R},
}\eear
which induces the isomorphism $H^p_{\mathpzc{dR}, \mathpzc{c}} (\manir) \cong \mathbb{R}$ in compactly supported de Rham cohomology. More in particular, the above constructions allow us to prove an analog of \emph{Stokes theorem} for supermanifolds. }
\begin{theorem}[Stokes Theorem for Supermanifolds] \label{stokesthm} Let $\mani$ be a real supermanifold of dimension $p|q$ with $\manir$ oriented and let $\mathcal{D}_{\mathpzc{c}} \in \mathcal{B}er_{\mathpzc{c}} (\mani)$. Then the following are true.
\begin{enumerate}[leftmargin=*] \item There exists $\omega_\mathpzc{c} \in \Sigma^{p-1}_{\mani, \mathpzc{c}}$ such that $\mathcal{D}_\mathpzc{c} = \delta \omega_\mathpzc{c}  $ if and only if 
\bear
\int_\mani \mathcal{D}_\mathpzc{c}  = 0.
\eear
In other words a compactly supported integral form of degree $p$ is exact, \emph{i.e.} $[\mathcal{D}_\mathpzc{c}] \equiv 0 \in H^p_{\mathpzc{Sp}, \mathpzc{c}} (\mani)$ if and only if it has vanishing Berezin integral. 
\item If $\mani$ is connected, the Berezin integral defines an isomorphism  
\bear
\int_\mani : H^p_{\mathpzc{Sp},\mathpzc{c}} (\mani) \stackrel{\cong}{\longrightarrow} \mathbb{R}.
\eear
In particular, a representative of $H^p_{\mathpzc{Sp}, \mathpzc{c}} (\mani)$ is given by $\sigma_\mathpzc{c} \defeq \mathcal{D}(x) \theta_1 \ldots \theta_q \mathpzc{B}_\mathpzc{c} (z_1,\ldots, z_p)$, where $\mathpzc{B}_\mathpzc{c}$ is any bump function which integrate to one on $\manir$. 
\end{enumerate}
\end{theorem}
\begin{proof} It is enough to use \eqref{compomap}. More in particular, on the one hand we have observed that if $\mathcal{D}_\mathpzc{c} \notin \mathcal{J}^q_\mani \mathcal{B}er_\mathpzc{c} (\mani)$, \emph{i.e.} does not have all the theta's, then it is in the image of $\delta$ and, by definition, its Berezin integral yields zero. On the other hand, if $\mathcal{D}_\mathpzc{c} \in \mathcal{J}^q_\mani \mathcal{B}er_{\mathpzc{c}} (\mani) \cap \delta (\Sigma^{p-1}_{\mani, \mathpzc{c}})$, \emph{i.e.}\ it has all of the theta's and it is in the image of $\delta$, then if one has $\mathcal{D}_\mathpzc{c} = \mathcal{D} (x) \theta_1 \ldots \theta_q f_\mathpzc{c}(z)$ then $f_\mathpzc{c} (z)$ is a divergence, and $\mathcal{D}_\mathpzc{c}$ gets mapped to an element in $d(\Omega^p_{\manir, \mathpzc{c}})$, which integrates to zero by ordinary Stokes theorem. The second point follows immediately from \eqref{compomap} and the previous theorem \ref{isoint}. 
\end{proof}
{\remark Explicitly one has the following isomorphism between local representatives
\bear
\xymatrix{
H^p_{\mathpzc{Sp}, \mathpzc{c}} (\mani)   \owns \mathcal{D} (x) \theta_1 \ldots \theta_q \mathpzc{B}_\mathpzc{c} (z_1, \ldots z_p) \ar@{|->}[r]^{ \tilde \varphi \quad } &  dz_1 \wedge \ldots \wedge dz_p \mathpzc{B}_\mathpzc{c} (z_1, \ldots z_p) \in H^p_{\mathpzc{dR}, \mathpzc{c}} (\manir).
} 
\eear
where again $\mathpzc{B}_\mathpzc{c} (z_1,\ldots, z_p)$ is a bump function which integrate to one over $\manir $. More in general, as in remark \ref{intcompact}, it can be proved that the \emph{compactly supported} Spencer cohomology of integral forms is isomorphic to the compactly supported de Rham cohomology, so that one has 
$
H^\bullet_{\mathpzc{Sp}, \mathpzc{c}} (\mani) \cong H^\bullet_{\mathpzc{dR}, \mathpzc{c}} (\manir).
$}
{\remark We will not deal with the subtle case of supermanifolds \emph{with boundaries} and the related Stokes' Theorem. More on this can be found in \cite{Manin} and \cite{Voronov, Witten}.}

\subsection{Supersymmetry and the Berezin Integral} The foremost application of the theory of integration on supermanifolds is related to high-energy physics, in particular with modern \emph{supersymmetric field theories} \cite{CDF, Deligne}. Very roughly speaking, physical \emph{elementary particles} are described mathematically via irreducible (projective and unitary) representations of the Poincaré group, whose Casimir invariants are in turn related to the \emph{mass} and the \emph{spin} (or \emph{helicity} in the zero-mass case) of the particles. In this context, a \emph{supersymmetry} is a physical symmetry that relates particles characterized by integer spins (\emph{bosons}, physically describing \virgolette interactions'') to particles characterized by half-integer spins (\emph{fermions}, physically describing \virgolette matter''). These symmetries are building pillars of the most far-reaching theories in contemporary physics, such as string theory. \\
Briefly, a physical theory on an unspecified space(time) $\mathpzc{M}$, where $\mani $ is an ordinary manifold, is described by an action functional $\mathcal{A} : \mathscr{F}_{\mani} \rightarrow \mathbb{R}$ where $\mathscr{F}_{\mathpzc{M}}$ is the \emph{space of fields} $\varphi^i$ of the theory. One usually writes the action as an integral over $\mani$   
\bear \label{actionfis}
\mathcal{A} \defeq \int_{\mathpzc{M}} \mathscr{L} (\varphi^i) 
\eear  
where $\mathscr{L} $ is the so-called \emph{Lagrangian density} of the theory, which - for an ordinary space-time manifold $\mathpzc{M}$ - is a (compactly supported) section of the canonical sheaf of $\mathpzc{M}$, \emph{i.e.} $\mathscr{L} \in \Omega^{\dim \mathpzc{M}}_{\mani, \mathpzc{c}}$, so that the integral makes sense (we consider $\mathpzc{M}$ to be oriented). We say that the physical theory described by $\mathcal{A}$ is \emph{invariant} under the (infinitesimal) transformation generated by a vector field $\mathcal{X} \in \mathcal{T}_\mathpzc{X}$ if 
\bear \label{invariance}
\delta_\mathcal{X} \mathcal{A} \defeq \int_{\mathpzc{M}} \mathcal{L}_{\mathcal{X}} \mathscr{L} = 0,
\eear
where $\mathcal{L}_{\mathcal{X}} \mathscr{L}$ is the Lie derivative of $\mathscr{L}$ with respect to the field $\mathcal{X}$. In this case one says that the transformation is a symmetry of the theory and the field $\mathcal{X}$ is the generator of the symmetry. \\
The simplest way to make a physical theory manifestly invariant under a given transformation is to construct the theory in a space whose isometry group contains such a transformation: supermanifolds do this job for supersymmetry transformations. More precisely, supersymmetric field theories can be made into manifestly invariant theories under supersymmetry if they can be constructed as theories on particular supermanifolds called \emph{superspacetimes}, which are defined as \emph{homogenous superspaces for the action of Poincaré Lie supergroups} and whose reduced manifolds are ordinary physical spacetimes (\emph{e.g.}\ the Minkowski spacetime $\mathbb{R}^{1, D-1}$ in $D$ dimensions) \cite{AC, DeligneFreed, Freed, Varadarajan}. \\
Superspacetimes are constructed out of three pieces of data, 
\begin{enumerate}[leftmargin=*]
\item a \emph{real quadratic vector space} $(V, \mathpzc{Q})$ of dimension $\dim (V) = D$, with $\mathpzc{Q}: V \rightarrow \mathbb{R}$ whose associated symmetric bilinear form $\mathpzc{B} : V \times V \rightarrow \mathbb{R}$ has signature $(1, D-1)$;
\item a \emph{real spinorial representation} $\mathcal{S} : {\cat{Spin}} (V ) \rightarrow \cat{Aut} (\mathcal{S})$ of dimension $\dim (\mathcal{S}) = q$ of the group $\cat{Spin} (V)$.
\item a $\cat{Spin}(V)$-equivariant symmetric non-zero bilinear map $\gamma : \mathcal{S} \times \mathcal{S} \rightarrow V$.
\end{enumerate}
Notice that the map $\gamma : \mathcal{S} \times \mathcal{S} \rightarrow V$ always exists in the given setting \cite{Varadarajan}. 
In the above data, $V$ is called (abelian) \emph{translation algebra}, as it appears as the summand in the Poincaré Lie algebra   
$
\mathfrak{Iso} (V) \defeq V \rtimes \mathfrak{so} (V)
$
corresponding to spacetime translations. The related Poincaré Lie superalgebra $\mathfrak{SIso} (V) = \mathfrak{g}_0 \oplus \mathfrak{g}_1$ is constructed out of the direct product of vector spaces 
\bear
\mathfrak{SIso} (V) = (V \rtimes \mathfrak{so} (V)) \oplus \mathcal{S},
\eear
where $\mathfrak{g}_0 \defeq V \rtimes \mathfrak{so} (V)$ and $\mathfrak{g}_1 \defeq \mathcal{S}$. Here $\mathcal{S}$ is looked at as a $\mathfrak{g}_0$-module with a trivial action 
\bear \label{actionV}
V \cdot \mathcal{S} \defeq [v, \mathpzc{s}] = 0 
\eear
for any $v \in V$ and $\mathpzc{s} \in \mathcal{S}$. Defining further 
\bear
[\mathpzc{s}_1, \mathpzc{s}_2] \defeq \gamma (\mathpzc{s}_1, \mathpzc{s}_2)
\eear  
for any $\mathpzc{s}_1, \mathpzc{s}_2 \in \mathcal{S}$, one obtains a Lie superalgebra, indeed $[\mathpzc{s}, [\mathpzc{s}, \mathpzc{s}]] $ is zero since $[\mathpzc{s}, [\mathpzc{s}, \mathpzc{s}]] = [\mathpzc{s}, \gamma (\mathpzc{s}, \mathpzc{s})]  = 0$ by \eqref{actionV}. Letting $\mathcal{S}$ be irreducible with basis given by $Q_a$ for $a = 1, \ldots, q$ and $P_\mu$ with $\mu = 1, \ldots D$ the standard basis of $V \cong \mathbb{R}^{1, D-1}$, one can write - upon a suitable normalization of $\gamma$ -:
\bear \label{susyrel}
[Q_a, Q_b] = \sum_{\mu = 1}^D \gamma^\mu_{ab} P_\mu,
\eear
where $[Q_a, Q_b] \defeq \gamma (Q_a, Q_b)$ so that $\gamma^\mu_{ab} = \gamma^\mu_{ba}$: this is the way the crucial commutation relation of Poincaré Lie superalgebra appears in physics. \\
The Poincaré Lie supergroup $\mathbb{S}\mbox{\cat{Iso}} (V)$ is obtained exponentiating this construction, and the related superspacetime is the quotient supermanifold (or \emph{coset superspace}, as most frequently called in physics) obtained by modding out the Lie group $\mbox{\cat{SO}} (V)$ of Lorentz transformations \cite{Varadarajan},
\bear
\mathscr{M} \defeq \slantone{\mathbb{S}\cat{Iso} (V)}{\mbox{\cat{SO}} (V)}.
\eear
 Equally, one can notice that $V \oplus \mathcal{S}$ is also a Lie superalgebra, with $\mathfrak{g}_0 = V$, the ordinary translation algebra and $\mathfrak{g}_1 = \mathcal{S}$: this is called \emph{translation superalgebra} and the superspacetime is given by its related Lie supergroup.  
Notice that $\dim ( \mathscr{M} ) = \dim( V) | \dim (\mathcal{S} ) = D | q$, and if $(x^\mu | \theta^a)$ and $(y^\mu | \psi^a) $ are coordinates for $\mathscr{M}$, then the group law $(z | \lambda) \defeq (x | \theta ) \cdot (y |\psi) $ reads 
\begin{align}
& z^\mu = x^\mu \cdot y^\mu = x^\mu + y^\mu - \frac{1}{2} \sum_{a,b } \gamma^\mu_{ab} \theta^a \psi^b, \qquad \lambda^a = \theta^a + \psi^a. 
\end{align}
Correspondingly, a set of even left-invariant vector fields is given by $\{ \partial_\mu \}_{\mu = 1, \ldots, D}$, while a set of \emph{odd left-invariant vector fields} is given by $\{ \mathcal{Q}_a\}_{a = 1, \ldots, q}$, with 
\bear \label{susygen}
\mathcal{Q}_a \defeq \frac{\partial}{\partial \theta^a} + \frac{1}{2} \sum_{\mu, b} \gamma^\mu_{ab} \theta^b \frac{\partial}{\partial x^\mu}
\eear
and it is not hard to verify that $[\mathcal{Q}_a, \mathcal{Q}_b] = \sum_{\mu = 1}^D \gamma^\mu_{ab} \partial_\mu$, just like the above \eqref{susyrel}, upon identifying $P_\mu \defeq \partial_\mu$, as customary, since the $\partial_\mu$'s generate space-time translations. In view of this, \emph{supersymmetry transformations are generated by the vector fields $\mathcal{Q}_a$}.\\ 
In light of the previous sections, one can generalize the definition \eqref{actionfis} and \eqref{invariance} on a superspacetime. The action of the physical theory is now given by an integral on the superspacetime 
\bear \label{superaction}
\mathcal{A}^{\mathpzc{s}} \defeq \int_{\mathscr{M}} \mathscr{L}^\mathpzc{s} (\varphi^i),
\eear
where now the Lagrangian density is a section of the compactly supported Berezinian sheaf of the superspacetime $\mathscr{M}$, \emph{i.e.} $\mathscr{L}^\mathpzc{s} \in \mathcal{B}er_{\mathpzc{c}} (\mathscr{M})$. It can be trivialized as
\bear
\mathscr{L}^\mathpzc{s} (\varphi^i) = \mathcal{D}(x | \theta) \, \Phi (\varphi^i (x), \theta^a),   
\eear
with $\mathcal{D} (x|\theta) = [dx_1, \ldots, dx_D \otimes \partial_{\theta_1} \ldots \partial_{\theta_q}]$ a generating section for the Berezinian and where $\Phi (\varphi^i, \theta^a) \in \mathcal{O}_\mathscr{M, \mathpzc{c}}$ is a so-called \emph{superfield}, containing the original physical fields $\varphi^i \in \mathcal{O}_{\mathpzc{M}, \mathpzc{c}}$ - plus auxiliary fields - in its component expansion. The component fields transform one into another under supersymmetry, forming a so-called \emph{supersymmetry multiplet}. 
\begin{lemma}[Supersymmetry Invariance] Let $\mathcal{A}^\mathpzc{s} $ be an action on a connected superspacetime $\mathscr{M}$ as in \eqref{superaction} and let $\mathcal{Q}$ be any supersymmetry generator of the form \eqref{susygen}. Then 
\bear
\delta_{\mathcal{Q}} \mathcal{A}^\mathpzc{s} = \int_\mathscr{M} \mathcal{L}_\mathcal{Q} ( \mathscr{L}^\mathpzc{s} ) = 0.
\eear
In particular, the action is invariant under supersymmetry.
\end{lemma}
\begin{proof} The result follows from the action of the Lie derivative on sections of the Berezinian sheaf as in equation \eqref{ActionLie} and Stokes theorem \ref{stokesthm} for supermanifolds. \\
More precisely, the part of the Lie derivative with respect to the odd coordinate vector fields $\partial_{\theta_a} \in \mathcal{Q}_a$ integrate to zero since $\mathcal{L}_{\partial_a} (\mathscr{L}^\mathpzc{s} ) \notin \mathcal{J}^q_\mathscr{M} \mathcal{B}er_\mathpzc{c}(\mathscr{M})$ and, as such, it is $\delta$-exact as an integral form. Similarly, the part of the Lie derivative with respect to the odd vector field $\sum_{\mu, b} \gamma^\mu_{ab} \theta^b \frac{\partial}{\partial x^\mu}\in \mathcal{Q}$ yields a divergence, which is an element in $\mathcal{J}^q_\mathscr{M} \mathcal{B}er_\mathpzc{c} (\mathscr{M}) \cap \delta (\Sigma^{D-1}_{\mathscr{M, \mathpzc{c}}})$ that integrate to zero again by Stokes theorem \ref{stokesthm}. 
\end{proof}
{\remark In general, verifying that a theory is indeed invariant under supersymmetry is not a trivial matter, and it often requires going through delicate and lengthy calculations. The above lemma shows that upon using an adequate and mathematically aware formalism - based on supermanifolds, integral forms, and the related integration theory -, supersymmetry invariance becomes apparent and practically no checks are required. Nonetheless, it is fair to say that writing a superspace action is in general not an easy task and indeed there exist theories of great physical interest for which action on superspacetime is not known \cite{DeligneFreed, Freed}. }

{\remark More in general, Lagrangian densities on superspacetimes might also involve \emph{covariant derivatives}, \emph{i.e.}\ differential operators acting on superfields which commutes with the supersymmetry generators: this is the case for example of the kinetic term (\emph{field strength}) of supersymmetric gauge theories. Geometrically, these can be constructed as \emph{right} invariant vector fields on the superspacetime $\mathscr{M}$, in opposition with $\mathcal{Q}$ being left invariant \cite{DeligneFreed, Freed, Varadarajan}.  }

\section{Poincaré Duality on Supermanifolds} 
\noindent Having available a notion of integration on supermanifolds via the Berezin integral, we can prove the analog of {Poincaré duality} for supermanifolds. Whereas Poincaré duality on ordinary manifolds yields a perfect pairing between the Rham cohomology groups of the manifolds, we will see instead that on a supermanifold Poincaré duality defines a perfect pairing between the cohomology of two different complexes, that of differential and that of integral forms: this is rooted in the peculiar geometry of forms and, in turn, integration theory on supermanifolds. We start with some preliminary remarks.
{\remark As discussed in the first section around equation \eqref{obpair}, there is an obvious pairing $\Pi \mathcal{T}_\mani \times \Omega^1_\mani \rightarrow \stsheaf$, given by the contraction of $\Pi$-vector fields and forms on $\mani.$ This can be extended to higher supersymmetric powers of $\Pi \mathcal{T}_\mani$ and $\Omega^1_\mani$, 
\bear \label{pairn}
\xymatrix@R=1.5pt{
\langle \cdot, \cdot \rangle :  \cat{S}^n \Pi \mathcal{T}_\mani \times \Omega^m_{\mani} \ar[r] & \cat{S}^{n-m} \Pi \mathcal{T}_\mani \\
}
\eear
for any $n, m \geq 1$ with $n\geq m$ so that associativity is met in the form $\langle \langle \pi X^{(n)} , \omega_1\rangle, \omega_2 \rangle = \langle \pi X^{(n)} , \omega_1 \omega_2 \rangle $, for any $\pi X^{(n)} \in \cat{S}^n \Pi \mathcal{T}_\mani$ and $\omega_1, \omega_2 \in \Omega^\bullet_\mani$ whose sum of degrees does not exceed $n$.  This pairing induces a \emph{right multiplication} between integral forms in $\Sigma^{p-n}_\mani = \mathcal{B}er(\mani) \otimes \cat{S}^n \Pi \mathcal{T}_\mani$ and differential forms in $\Omega^m_\mani$, that we write as
\bear
\xymatrix@R=1.5pt{
\cdot  : \Sigma^{p-n}_\mani \times \Omega^m_{\mani} \ar[r] & \Sigma^{p-(n-m)}_\mani, \\
( \sigma, \omega ) \ar@{|->}[r] &  \sigma \cdot \omega  
}
\eear
for $n, m \geq 0$, $n\geq m$ and $n+m$ and where $p$ is the even dimension of $\mani,$ so that $\Sigma^p_\mani = \mathcal{B}er(\mani).$ Notice that, at this stage, this does not define a structure of right $\Omega^\bullet_\mani$-module on integral forms, since the multiplication is defined only for certain degrees. 
The crucial observation is that the above multiplication is compatible with the differential, \emph{i.e.}\ we have a Leibniz rule in the form 
\bear \label{leibnizstrange}
\delta (\sigma \cdot \omega) = \delta \sigma \cdot \omega + (-1)^{|\sigma|} \sigma \cdot d\omega,
\eear
for any $\sigma \in \Sigma^{p-n}_\mani $ and $\omega \in \Omega^m_\mani$ and where $d$ is the de Rham differential. This can be proved in local coordinates by expanding the tensor $\pi X^{(n)}$ over a base of $\cat{S}^n \Pi \mathcal{T}_\mani$ of supersymmetric products of the $\pi \partial_{x_a}$'s and use \eqref{pairn}. This is a lengthy computation, where the bookkeeping of the signs involved plays a crucial role: we check it in the case $\sigma \in \Sigma^{p-2}_\mani$ and $\omega \in \Omega^1_\mani$, for $\sigma = \sum_{a,b } \mathcal{D}(x) f \otimes X_{ab} \pi \partial_{a} \, \pi \partial_{b} $ and $\omega = \sum_c g_c dx_c$. One has 
\begin{align} \label{ex2}
& \delta (\sigma) \cdot \omega +  (-1)^{|\sigma|} \sigma \cdot d \omega =  \nonumber \\
& \sum_{a, b} (-1)^{|x_a|+1 + |X_{ab}| + |\mathcal{D}(x)| + |f|  + |g_b| (|x_b| + 1)} \mathcal{D}(x) \bigg ((-1)^{|x_a| (|f| + |X_{ab}|)} \partial_a (f X_{ab}) g_b + f X_{ab} \partial_a g^b \bigg ) \pm  (a \leftrightarrow b)    \nonumber \\
& \sum_{a, b} (-1)^{|x_a|+1 + |X_{ab}| + |\mathcal{D}(x)| + |f|  + |g_b| (|x_b| + 1) + |x_a| (|f| + |X_{ab}|)} \mathcal{D}(x)  \partial_a (f X_{ab} g_b)\pm  (a \leftrightarrow b).
\end{align}
On the other hand, one computes directly that 
\bear
 \delta (\sigma \cdot \omega) = \sum_{a, b} (-1)^{|x_a|+1 + |X_{ab}| + |\mathcal{D}(x)| + |f|  + |g_b| (|x_b| + 1) + |x_a| (|f| + |X_{ab}|)} \mathcal{D}(x)  \partial_a (f X_{ab} g_b) \pm  (a \leftrightarrow b), 
\eear
matching \eqref{ex2}. The previous Leibniz rule \eqref{leibnizstrange} is the key to proving Poincaré duality for supermanifolds. }
\begin{theorem}[Poincaré Duality for Supermanifolds] \label{PD} Let $\mani$ be a real supermanifold of dimension $p|q$ with $\manir$ oriented. Then, for any $n \geq 0$, the Berezin integral defines a perfect pairing in cohomology 
\bear
\xymatrix@R=1.5pt{
H^{p-n}_{\mathpzc{Sp}, \mathpzc{c}} (\mani) \times H^n_{\mathpzc{dR}} (\mani) \ar[r] & \mathbb{R}  \\
( [\sigma_\mathpzc{c}], [\omega] ) \ar@{|->}[r] & \int_\mani \sigma \cdot \omega
}
\eear
In particular, there is a natural isomorphism 
\bear
(H^{p-n}_{\mathpzc{Sp}, \mathpzc{c}} (\mani))^{\ast} \cong  H^n_{\mathpzc{dR}} (\mani).
\eear
\end{theorem}
\begin{proof} First of all let us observe that the map is well-defined. Indeed if $\sigma_\mathpzc{c} \in \Sigma_{\mani, \mathpzc{c}}^{p-n} $ and $\omega \in \Omega^n_\mani$ are two representatives, then $\sigma_{\mathpzc{c}} \cdot \omega \in \Sigma^p_{\mani, \mathpzc{c}} = \mathcal{B}er_\mathpzc{c}(\mani)$, so that it can indeed be integrated in the Berezin sense. In other words, the map is given by the following composition: 
\bear \label{compodual}
\xymatrix@R=1.5pt{
\Sigma_{\mani, \mathpzc{c}}^{p-n} \times  \Omega^n_\mani \ar[r] & \Sigma^p_{\mani, \mathpzc{c}} \ar[r] & \mathbb{R} \\
(\sigma_\mathpzc{c}, \omega ) \ar@{|->}[r] & \sigma_\mathpzc{c} \cdot \omega \ar@{|->}[r] & \int \sigma \cdot \omega.
}
\eear
Further, let $\sigma_\mathpzc{c}$ and $\omega$ be \emph{closed}, \emph{i.e.}\ $\delta \sigma = 0 $ and $d \omega = 0$. It follows immediately from the  \eqref{leibnizstrange} that $\sigma_\mathpzc{c} \cdot \omega$ is closed, \emph{i.e.} $\delta (\sigma \cdot \omega) = 0.$ \\
Now let $\sigma_\mathpzc{c}$ be exact, \emph{i.e.}\ $\sigma_\mathpzc{c} = \delta \eta_\mathpzc{c}$ for some $\eta_\mathpzc{c} \in \Sigma^{p-(n-1)}_{\mani, \mathpzc{c}}$. Then, using again \eqref{leibnizstrange} one has that
\bear
\sigma_{\mathpzc{c}} \cdot \omega = \delta \eta_\mathpzc{c} \cdot \omega = \delta (\eta_\mathpzc{c} \cdot d\omega),
\eear  
proving that $\sigma_\mathpzc{c} \cdot \omega $ is exact. If instead $\omega \in \Omega^n_\mani$ is exact, \emph{i.e.}\ $\omega = d \gamma$ for some $\gamma \in \Omega^{n-1}_\mani$, then by \eqref{leibnizstrange} one has
\bear
\sigma_\mathpzc{c} \cdot \omega = \sigma_\mathpzc{c} \cdot d\gamma = (-1)^{|\sigma_\mathpzc{c}|} \delta (\delta \sigma_\mathpzc{c} \cdot \gamma), 
\eear
proving that $\sigma_\mathpzc{c} \cdot \omega$ is exact. We thus have that the composition of maps \eqref{compodual} descends to cohomology to give
\bear \label{compocoho}
\xymatrix@R=1.5pt{
 H^{p-n}_{\mathpzc{Sp}, \mathpzc{c}} (\mani) \times H^n_{\mathpzc{dR}} (\mani) \ar[r] &  H^{p}_{\mathpzc{Sp}, \mathpzc{c}} (\mani) \ar[r] & \mathbb{R} \\
([\sigma_\mathpzc{c}], [\omega] ) \ar@{|->}[r] & [\sigma_\mathpzc{c} \cdot \omega] = [\sigma_\mathpzc{c}] \cdot [\omega] \ar@{|->}[r] & \int \sigma \cdot \omega.
}
\eear
Now let us work by induction on the cardinality of the good covering. Namely, let us start proving the result for the pivotal case of a covering of cardinality one. This corresponds to the Poincaré lemma for $\mathbb{R}^{p|q}$, which follows straightforwardly from the Poincaré lemmas for differential forms and compactly supported integral forms, \emph{i.e.}\
\bear
H^k_{\mathpzc{dR}} (\mathbb{R}^{p|q}) \cong
\left \{
\begin{array}{cc}
\mathbb{R} &  \quad k=0 \\
0 & \quad  k \geq 0,
\end{array}
\right. 
\qquad
H^k_{\mathpzc{Sp}, \mathpzc{c}} (\mathbb{R}^{p|q}) \cong
\left \{
\begin{array}{cc}
\mathbb{R} &  \quad k=p \\
0 & \quad  k \neq p.
\end{array}
\right. 
\eear
The only non-zero pairing reads
\bear \label{compocoho}
\xymatrix@R=1.5pt{
 H^{p}_{\mathpzc{Sp}, \mathpzc{c}} (\mathbb{R}^{p|q}) \times H^0_{\mathpzc{dR}} (\mathbb{R}^{p|q}) \ar[r] &  H^{p}_{\mathpzc{Sp}, \mathpzc{c}} (\mani) \ar[r] & \mathbb{R} \\
([\mathcal{D}(x) \theta_1 \ldots \theta_q \mathpzc{B}_\mathpzc{c}], [1] ) \ar@{|->}[r] & [\mathcal{D}(x) \theta_1 \ldots \theta_q \mathpzc{B}_\mathpzc{c}] \cdot [\omega] \ar@{|->}[r] & \int_{\mathbb{R}^{p|q}} \mathcal{D}(x) \theta_1 \ldots \theta_q \mathpzc{B}_\mathpzc{c}.
}
\eear
Upon choosing $\mathpzc{B}_\mathpzc{c} (z_1, \ldots, z_p)$ a bump functions which integrates to one on $\mathbb{R}^p$, one gets
\bear
\int_{\mathbb{R}^{p|q}} \mathcal{D}(x) \theta_1 \ldots \theta_q \mathpzc{B}_\mathpzc{c} = \int_{\mathbb{R}^p} dz_1 \ldots dz_p \mathpzc{B}_\mathpzc{c} (z_1, \ldots, z_p) = 1,
\eear
concluding the proof of the isomorphism $(H^{p}_{\mathpzc{Sp}, \mathpzc{c}} (\mathbb{R}^{p|q}))^{\ast} \cong H^0_{\mathpzc{dR}} (\mathbb{R}^{p|q}).$ \\
Now, if $\mani$ is covered by two open sets $U$ and $V$ so that with abuse of notation we can write $\mani =  U \cup V$ as supermanifolds, thanks to the Mayer-Vietoris sequences for ordinary and compactly supported cohomology 
we have the following (sign-)commutative diagram \cite{BottTu}
\begin{align}
\xymatrix@C=5pt{
\ldots \ar[r] & H^{p-k+1}_{\mathpzc{Sp}} (U)^\ast \oplus H^{p-k+1}_{\mathpzc{Sp}, \mathpzc{c}} (V)^\ast \ar[d]_{\alpha^{k-1}_{U + V}} \ar[r] & H^{p-k+1}_{\mathpzc{Sp}, \mathpzc{c}} (U \cap V)^\ast \ar[d]_{\alpha^{k-1}_{U \cap V}} \ar[r] & H^{p-k}_{\mathpzc{Sp}, \mathpzc{c}} (U \cup V)^\ast \ar[r] \ar[d]_{\alpha^k_{U \cup V}} &  H^{p-k}_{\mathpzc{Sp}, \mathpzc{c}} (U)^\ast \oplus  H^{p-k}_{\mathpzc{Sp}, \mathpzc{c}} (V)^\ast \ar[d]_{\alpha^k_{U + V}} \ar[r] & \ldots  \\ 
\ldots \ar[r] & H^{k-1}_{\mathpzc{dR}} (U) \oplus H^{k-1}_{\mathpzc{dR}} (V) \ar[r] & H^{k-1}_{\mathpzc{dR}} (U \cap V) \ar[r] & H^k_{\mathpzc{dR}} (U \cup V) \ar[r] &  H^k_{\mathpzc{dR}} (U) \oplus H^k_{\mathpzc{dR}} (V) \ar[r] & \ldots 
}
\end{align}
Notice that the above is induced by the pairing of the two Mayer-Vietoris long exact sequences, so that one has 
\bear
\xymatrix{
& \mathbb{R} & & \\
\ldots & \ar[l]  H^{p-k+1}_{\mathpzc{Sp}, \mathpzc{c}} (U \cap V) \ar[u]^{\int_{U \cap V}} \ar@{}[d]_{\otimes} & \ar[l]_{\; \; \delta}  H^{p-k}_{\mathpzc{Sp}, \mathpzc{c}} (U \cup V) \ar@{}[d]_\otimes & \ar[l]  \ldots \\
\ldots & \ar[l]  H^{k-1}_{\mathpzc{dR}} (U \cap V) \ar[r]^d  & H^k_{\mathpzc{dR}} (U \cup V) \ar[r]\ar[d]^{\int_{U \cup V}} & \ldots \\
&  & \mathbb{R} }
\eear
in a way such that (sign-)commutativity means that for any $\sigma_{\mathpzc{c}} \in H^{p-k}_{\mathpzc{Sp}, \mathpzc{c}} (U \cup V)$ and $\omega \in H^{k-1}_{\mathpzc{dR}} (U \cap V)$ one has
\bear
\int_{U \cap V} \delta \sigma_{\mathpzc{c}} \cdot \omega = \pm \int_{U \cup V} \sigma_{\mathpzc{c}} \cdot d\omega.
\eear 
By Poincaré duality for ${\mathbb{R}^{p|q}}$ the maps $\alpha^{k-1}_{U+V}, \alpha^{k}_{U+V}$ and $\alpha_{U \cap V}^{k}, \alpha^k_{U \cap V}$ are isomorphisms. It follows from \emph{five lemma} that also $\alpha^{k}_{U \cap V}$ is an isomorphism. The proof is then concluded by working by induction on the cardinality of the covering. Namely, suppose Poincaré duality holds for supermanifolds admitting a covering by $n$ opens sets at most and consider a manifold with a covering of $n+1$ open sets, $\{ U_1, \ldots, U_{n+1}\}$: then by hypothesis Poincaré duality holds for the supermanifolds $( \bigcup_{i=1}^{n} U_i )\cap U_{n+1}$, $U_{n+1}$ and $\bigcup_{i=1}^{n} U_i$, so reasoning exactly as above one conclude that Poincaré duality holds true also for $\bigcup_{i=1}^{n+1} U_i$. 
\end{proof}

{\remark The same is true also switching the compact support from integral to differential forms, \emph{i.e.} we have a perfect pairing
\bear
\xymatrix@R=1.5pt{
H^{p-n}_{\mathpzc{Sp}} (\mani) \times H^n_{\mathpzc{dR}, \mathpzc{c}} (\mani) \ar[r] & \mathbb{R}  \\
( [\sigma], [\omega_\mathpzc{c}] ) \ar@{|->}[r] & \int_\mani \sigma \cdot \omega_\mathpzc{c}.
}
\eear
Notice that in this context Poincaré duality for the supermanifold $\mathbb{R}^{p|q}$ depends once again on the Poincaré lemmas. The non-trivial part of the pairing is given by $H^0_{\mathpzc{Sp}} (\mathbb{R}^{p|q}) \times H^p_{\mathpzc{dR}, \mathpzc{c}} (\mathbb{R}^{p|q}) \rightarrow \mathbb{R}$, and the representative are $\sigma = \mathcal{D}(x) \theta_1\ldots \theta_q \otimes \pi \partial_{z_1} \ldots \pi \partial_{z_p} \in H^0_{\mathpzc{Sp}} (\mathbb{R}^{p|q}) $ and $\omega_{\mathpzc{c}} = dz_1 \ldots dz_p \mathpzc{B}_\mathpzc{c} (z_1, \ldots, z_p) \in H^n_{\mathpzc{dR},\mathpzc{c}} (\mathbb{R}^{p|q})$ respectively, where $\mathpzc{B}_\mathpzc{c} (z_1, \ldots, z_p)$ is a bump function which integrate to one on $\mathbb{R}^p$. The pairing integral then reads
\begin{align}
\int_{\mathbb{R}^{p|q}} \sigma \cdot \omega_\mathpzc{c} & = \int_{\mathbb{R}^{p|q}} \mathcal{D}(x) \theta_1\ldots \theta_q \otimes \pi \partial_{z_1} \ldots \pi \partial_{z_p} \cdot (dz_1 \ldots dz_p ) \mathpzc{B}_\mathpzc{c} (z_1, \ldots, z_p) \nonumber \\ 
& = \int_{\mathbb{R}^{p|q}} \mathcal{D}(x) \theta_1 \ldots \theta_q \mathpzc{B}_\mathpzc{c} (z_1,\ldots, z_p) \nonumber \\
& = \int_{\mathbb{R}^p} dz_1 \ldots dz_p \mathpzc{B}_{\mathpzc{c}} (z_1, \ldots z_p) = 1,
\end{align}
where we have used that $ \pi \partial_{z_1} \ldots \pi \partial_{z_p} \cdot (dz_1 \ldots dz_p ) = 1$ and the very definition of the Berezin integral. }

\section{Different Perspectives: Forms and Integration on Total Space}

\noindent Given a supermanifold $\mani$, there exists a different point of view on the theory of forms and integration on $\mani$, which highlights the role of the supermanifold associated to the (parity-shifted) tangent bundle $\Pi \mathcal{T}_\mani $ of $\mani$. Even if this point of view has been introduced since the early days by Bernstein and Leites \cite{BL2}, Voronov and Zorich \cite{VZ1}, Gaiduk, Khudaverdian, and Schwarz \cite{Gaiduk}, it came back in the spotlight only recently. Indeed, inspired by the work of Belopolsky \cite{Belo2}, Witten has provided a terse review of this formalism, putting forward new applications to string theory \cite{Witten}. Further, Castellani, Catenacci, Grassi, and collaborators have worked extensively in this direction, providing many examples of applications to physics, and in particular proving its effectiveness in the context of supergravity \cite{CCG}-\cite{CGP}.

\noindent In this section we will discuss this new perspective on forms and integration on supermanifolds. The reader is warned, though, that the present section has a somewhat different flavor compared to the preceding ones. Indeed, the material presented in the previous sections is mathematically well-established, and as such, it is amenable to the unified and rigorous treatment we tried to give in the previous pages. On the other hand, the issues discussed in the present section constitute an active mathematical research area, recently spawned from physical ideas and intuitions \cite{Belo2, Witten, WittenRiem, WittenSuper}. As such, this section should be seen as a first step taken by the author toward a rigorous mathematization of the exciting ideas advocated by Belopolsky and Witten in a string-theoretic framework within the context of a long-term research program. For this reason, the reader cannot expect the same level of completeness and the systematization which characterized the previous part of this paper.  \\ 
More in detail, we start providing rigorous mathematical definitions, see definition \ref{TM}, making precise notions which have been so far discussed up to a \virgolette{physics standard of rigor}'' only. These allow us to describe the local and global geometry of supermanifold associated to the parity-shifted tangent bundle of a supermanifold. Differential forms are then characterized globally via an extension of vector bundles, lemma \ref{ExtO}, which in turn allows one proving a \virgolette triviality property'' for the Berezinian sheaf, lemma \ref{canvol}. This is the starting point for the definition of an integration theory on these superspaces. The second part of the section is dedicated to a specific class of (integrable) \virgolette functions'' that can be defined on these supermanifolds, namely those having a \emph{distributional dependence} of Dirac-delta type of the fiber coordinates, definition \ref{deltaforms}. These are the objects advocated by Belopolski and Witten for string theory applications. We first give a mathematically precise meaning to their formal properties by endowing these \virgolette delta forms'' with several different module structures, and we show that they carry a degree. Finally, upon postulating a reasonable transformation property, we show that delta forms - once organized into locally-free sheaves of modules by their degree - are isomorphic to integral forms as previously introduced in this paper. As a consequence, the related integration theories are equivalent. This result fills a gap in the literature, making clear an otherwise rather obscure connection.

\begin{definition}[The Supermanifold $\Pi {T}\mani $] \label{TM} Let $\mani$ be a smooth supermanifold of dimension $p|q$ and let $\Pi \mathcal{T}_\mani$ be its parity-shifted tangent sheaf. We call $\Pi {T}\mani$ the $p+q|p+q$-dimensional supermanifold given as a ringed space by the pair $(|(\Pi {T}\mani)_{\mathpzc{red}}|, \mathcal{O}_{\Pi T\mani})$. The topological space $|(\Pi{T}\mani)_{\mathpzc{red}}|$ is given by the total space of the smooth vector bundle of rank $q$ over $\manir$ given by $ \pi_{\mathpzc{red}}: (\Pi {T}\mani)_{\mathpzc{red}} {\rightarrow} \manir$, endowed with its canonical topology. The sheaf $\mathcal{O}_{\Pi T \mani}$ is defined to be the sheaf associated to the presheaf defined by 
\bear
U \times V \longmapsto \mathcal{O}_{\Pi T \mani} (U\times V) \defeq \mathcal{O}_\mani (U) \hat{\otimes}_{\mathbb{R}} \mathcal{O}_{\mathbb{R}^{q|p}} (V),
\eear
for $U \times V \subset \pi^{-1}_{\mathpzc{red}} (U)$ an open set in $\pi^{-1}_{\mathpzc{red}} (U) = U \times \mathbb{R}^{q} \subset |(\Pi T\mani)_{\mathpzc{red}}|$, and where sections of $\mathcal{O}_\mani$ are smooth functions on the base supermanifold $\mani$ and sections of $\mathcal{O}_{\mathbb{R}^{q|p}}$ are smooth functions of the even and odd fiber coordinates. Finally, note that the tensor product is an appropriate completion of the ordinary algebraic tensor product.

\end{definition}
\begin{remark} In the above definition the \virgolette canonical topology'' on the total space of a vector bundle is defined locally via the product topology, and then gluing along the transition functions: the related quotient topology is the desired topology on the total space of the vector bundle.  
\end{remark}
\begin{remark} \label{remTM} If we let $(U, x_a)$ be a local chart for the $p|q$-dimensional manifold $\mani$, where the index $a$ runs on even and odd local coordinates, then $(\pi_{\mathpzc{red}}^{-1} (U), x_a, \mathpzc{F}_a)$ is a chart for $\Pi {T}\mani$ with fiber coordinates given by
\bear
\mathpzc{F}_a \defeq dx_a. 
\eear
This means that we impose on the $\mathpzc{F}_a$'s the same transformation rules of the $dx_a$'s. More precisely, if we let $(U, x_a)$ and $(U^\prime, x^\prime_a)$ be two charts on $\mani$ with $U \cap U^\prime \neq \emptyset$ and we let $(\pi^{-1}_{\mathpzc{red}}(U), x_a, dx_a)$ and $(\pi^{-1}_{\mathpzc{red}} (U^\prime), x^\prime_a, dx^\prime_a)$ the corresponding charts on $\Pi T \mani$, then the transition functions in the intersection $\pi^{-1}_{\mathpzc{red}} (U) \cap \pi^{-1}_{\mathpzc{red}} (U^\prime)$ read
\bear \label{transTM}
x_a = x_a (x^\prime), \qquad dx_a = dx^\prime_b \cdot \left ( \frac{\partial x_a}{\partial x^\prime_b} \right ).
\eear
\end{remark}
{\remark \label{pseudorem} The above local description of the real supermanifold $\Pi {T}\mani$ via charts allows to represent local sections of the sheaf $\mathcal{O}_{\Pi T\mani}$ over an open set in terms of the local coordinates $x_a$ and $\mathpzc{F}_a = dx_a$ as a function $f(x_a , \mathpzc{F}_a) = f (x_a, dx_a)$. 
Notice that, as in Definition \ref{TM}, in the smooth category also \emph{generalized} and \emph{transcendental} functions are allowed. In particular, it makes sense to consider - for example - a transcendental dependence on the \emph{even} fiber coordinates $\mathpzc{F}_a$'s. For example, considering the $0|1$-dimensional real supermanifold $\mathbb{R}^{0|1}$ described by an odd coordinate $\theta$, it makes sense for the $1|1$-dimensional supermanifold $\Pi {T}\mathbb{R}^{0|1}$ to consider sections of its structure sheaf $\mathcal{O}_{\Pi T \mathbb{R}^{0|1}}$ of the form 
\bear
f (\theta , d\theta) = \exp \left ( d\theta \right ), \qquad g (\theta, d\theta ) = \log (d\theta), \quad h(\theta, d\theta) = \sin (d\theta), 
\eear
where $d\theta = \mathpzc{F}$ is the even coordinate of the $1|1$ dimensional supermanifold $\cat{T}\mathbb{R}^{0|1}$. \\
In this respect, it is crucial to note that asking that sections of the structure sheaf $\mathcal{O}_{\Pi {T}\mani}$ are restricted to have only \emph{polynomial} dependence on the \emph{even} fiber coordinates $\mathpzc{F}_a$ is equivalent to set $\mathcal{O}_{\Pi {T}\mani} = \Omega^\bullet_\mani$ (with its $\mathcal{O}_\mani$-module structure), indeed, over an open set $U \times V \subset \pi^{-1}_{\mathpzc{red}} (U)$ one would have 
\bear
\mathcal{O}_{\Pi {T} \mani} (U \times V) \defeq \mathcal{O}_\mani (U) \otimes_{\mathbb{R}} \mathbb{R}[\mathpzc{F}_1, \ldots \mathpzc{F}_{p+q}] = \mathcal{O}_\mani (U) \otimes_{\mathbb{R}} \mathbb{R}[dx_1, \ldots, dx_{p+q}], \label{xF2}
\eear
This means that \virgolette functions'' on $\Pi {T}\mani$ with polynomial dependence of their fiber coordinates are ordinary differential forms on $\mani.$ \\ 
More in general, though, differential forms are just a subclass - actually a subalgebra - of sections of the structure sheaf of $\Pi {T} \mani$, \emph{i.e.}\ $\Omega^\bullet_\mani \subset \mathcal{O}_{\Pi {T} \mani}$. This remark justifies the following definition, which assigns a specific name to sections of $\mathcal{O}_{\Pi {T} \mani}$, see \cite{BL2, Manin, Witten}.}
\begin{definition}[Pseudodifferential Forms] Let $\mani$ be a real supermanifold and let $\Pi {T} \mani$ be defined as above. A section of $\mathcal{O}_{\Pi {T}\mani}$ is said to be a pseudodifferential form on $\mani$, or a pseudoform on $\mani$ for short. 
\end{definition}

\noindent Let us now study the geometry of the sheaf of 1-forms $\Omega^1_{\Pi {T} \mani}$ of the supermanifold $\Pi T \mani$. This is a locally-free sheaf of $\mathcal{O}_{\Pi {T} \mani}$-modules of rank  $p+q | p+q$, and over an open set $U \times V \subset \pi^{-1}_{\mathpzc{red}} (U)$ in $\Pi T\mani$ one has 
\bear
U \times V \longmapsto \Omega^{1}_{\Pi T\mani} (U \times V) = \mathcal{O}_{\Pi T \mani} (U \times V) \cdot \{ dx_1, \ldots dx_{p+q}, d\mathpzc{F}_1, \ldots, d\mathpzc{F}_{p+q} \},
\eear 
where the $d\mathpzc{F}_a$'s are the local 1-forms arising from the fiber coordinates $\mathpzc{F}_a$'s. Notice that, since we defined $\mathpzc{F}_a = dx_a$, one could write $d \mathpzc{F}_a = d(dx_a)$ - see remark \ref{WittenNot} later on in this section about this peculiar notation, which might cause some confusion. 
\begin{remark} \label{transtanM}
The transition functions of $\Omega^1_{\Pi {T} \mani}$ are easily characterized thanks to the previous remark \ref{remTM}, in particular equation \eqref{transTM}. More precisely, if we let
$(dx_a, d\mathpzc{F}_a)$ and $(dx^\prime_a, d\mathpzc{F}^\prime_a)$ be two local bases of $\Omega^1_{\Pi T \mani}$ on the open sets $\pi^{-1}_{\mathpzc{red}} (U)$ and $\pi^{-1}_{\mathpzc{red}}(U^\prime)$ in $\Pi T\mani$ with $U \cap U^\prime \neq \emptyset $, then the transition functions of $\Omega^1_{\Pi {T}\mani}$ on the intersection $\pi^{-1}_{\mathpzc{red}} (U) \cap \pi^{-1}_{\mathpzc{red}} (U^\prime)$ read
\bear \label{tfcot1}
dx_a = dx^\prime_b \cdot \left ( \frac{\partial x_a}{\partial x^\prime_b} \right )
\eear
\bear \label{tfcot2}
d\mathpzc{F}_a = d\mathpzc{F}^\prime_b \cdot \left ( \frac{\partial x_a }{\partial x^\prime_b} \right) + (-1)^{|x^\prime_b| +1} \mathpzc{F}^\prime_b  \cdot d \left ( \frac{\partial x_a}{\partial x^\prime_b} \right ). 
\eear
The transition functions of the $dx_a$'s in equation \eqref{tfcot1} are clear. For the transition functions of the $d\mathpzc{F}_a$'s in equation \eqref{tfcot2} one first observes that
\bear \label{summandtrans}
d \mathpzc{F}_a = dx^\prime_b \cdot \left ( \frac{\partial \mathpzc{F}_a}{\partial x^\prime_b} \right ) + d\mathpzc{F}^\prime_b \cdot  \left ( \frac{\partial \mathpzc{F}_a}{\partial \mathpzc{F}^\prime_b} \right ). 
\eear
The first summand in \eqref{summandtrans} reads
\begin{align}
dx^\prime_b \cdot \left ( \frac{\partial \mathpzc{F}_a }{\partial x^\prime_b} \right ) & = dx^\prime_b \frac{\partial}{\partial x^\prime_b} \left ( \mathpzc{F}^\prime_c \frac{\partial x_a}{\partial x^\prime_c}\right ) = (-1)^{|x^\prime_b| + 1} \mathpzc{F}^\prime_b \cdot d \left ( \frac{\partial x_a }{\partial x^\prime_b } \right).
\end{align}
The second summands in \eqref{summandtrans} read 
\begin{align}
d\mathpzc{F}^\prime_b \cdot \left ( \frac{\partial \mathpzc{F}_a}{\partial \mathpzc{F}^\prime_b} \right ) & =  d\mathpzc{F}^\prime_b \frac{\partial}{\partial \mathpzc{F}^\prime_b} \left ( \mathpzc{F}^\prime_c \frac{\partial x_a}{\partial{x^\prime_c}} \right ) = d\mathpzc{F}^\prime_b \cdot \left (\frac{\partial x_a}{\partial x^\prime_b} \right ). \nonumber 
\end{align}
\end{remark}
\noindent The specific form of the transition functions of $\Omega^1_{\Pi T \mani}$ leads to the following lemma which describes the global geometry of $\Omega^1_{\Pi {T} \mani}$. 
\begin{lemma}[$\Omega^1_{\tiny{\cat{T}} \mani}$ as Extension of Vector Bundles] \label{ExtO} Let $ \Pi {T}\mani$ be defined as above. Then the canonical exact sequence 
\bear
\xymatrix{
0 \ar[r] & \pi^\ast \Omega^1_\mani \ar[r] & \Omega^1_{\Pi T \mani} \ar[r] & \Omega^1_{\Pi {T} \mani / \mani} \ar[r] &  0
}
\eear
induces the isomorphism of locally-free sheaves $\Omega^1_{\Pi{T} \mani / \mani} \cong \pi^\ast \mathcal{T}^\ast_\mani$. In particular, $ \Omega^1_{\Pi{T} \mani}$ is defined as the extension of locally-free sheaves
\bear \label{extomega}
\xymatrix{
0 \ar[r] & \pi^\ast \Omega^1_\mani \ar[r] & \Omega^1_{\Pi T \mani} \ar[r] & \pi^\ast \mathcal{T}^\ast_\mani \ar[r] &  0.
}
\eear 
\end{lemma}
\begin{proof} We work in the setting of the previous remark \ref{transtanM}. First, one observes that the transformations of equation \eqref{tfcot1} identify the sections $\{ dx_a \}_{a = 1, \ldots, p+q}$ as a local basis of $\pi^\ast \Omega^1_{\mani}$ - notice the slight abuse of notation regarding the pullbacks. Further, the first summand in the transformations of equation \eqref{tfcot2} identify the transformations of the parity-reversed of $\pi^\ast \Omega^1_{\mani}$, \emph{i.e.} of $\pi^\ast \mathcal{T}_\mani^\ast$, as the $d\mathpzc{F}_a$ have opposite parity with respect to the corresponding $dx_a$. This is hence identified with $\pi^\ast \mathcal{T}^\ast_\mani.$ The second summand in \eqref{tfcot2} gives the off-diagonal terms of the extension of $\pi^\ast \mathcal{T}^\ast_{\mani}$ with $\pi^\ast \Omega^1_\mani.$
\end{proof}
\begin{remark} 
\noindent On a general ground, extensions as in \eqref{extomega} are classified by the first $Ext$-group
\bear \label{cohoclass}
Ext^1_{\Pi T\mani} (\pi^\ast \mathcal{T}^\ast_{\mani}, \pi^\ast \Omega^1_{\mani}) \cong H^1 (|(\Pi {T}\mani)_{\mathpzc{red}}|, \mathcal{H}om_{\tiny{\cat{T}}\mani} (\pi^\ast \mathcal{T}_{\mani}^\ast, \pi^\ast \Omega^1_{\mani} )),
\eear
where the subscript $\Pi T \mani$ means that we are working over $\mathcal{O}_{\Pi {T}\mani}.$
In the case $\mani$ is \emph{smooth} and $\Pi {T} \mani$ is the associated smooth supermanifold defined as above, the extension \eqref{extomega} is always \emph{split}, as the following corollary shows.
\begin{theorem}[Splitting \& Reduction] \label{splittingsmooth} Let $\mani$ be a smooth supermanifold and let $\Pi {T}\mani$ be the smooth supermanifold associated to $\mani$ as defined above. Then the following are true.
\begin{enumerate}[leftmargin=*]
\item[(1)] The extension 
\bear
\xymatrix{
0 \ar[r] & \pi^{\ast} \Omega^1_{\mani} \ar[r] & \Omega^{1}_{\Pi {T}\mani} \ar[r] & \pi^\ast \mathcal{T}^\ast_\mani  \ar[r] \ar@{->}@/_1.2pc/[l] & 0
}
\eear
is split, \emph{i.e.} $\Omega^{1}_{\Pi T\mani} \cong \pi^{\ast} \Omega^1_{\mani} \oplus  \pi^\ast \mathcal{T}^\ast_\mani$ non-canonically.
\item[(2)] There exists a reduction of the structure group of $\Omega^1_{\Pi {T}\mani}$ as follows 
\bear
GL (p+q|p+q) \longrightarrow \left \{ \left ( \begin{array}{c|c} 
T & \\
\hline 
& \Pi T
\end{array}
\right ):
T \in GL (p|q)
\right \},
\eear
where $ GL(p+q|p+q)$ is the general linear supergroup.
\end{enumerate}
\end{theorem}
\begin{proof} The first point follows from the existence of smooth partitions of unity in the smooth category. This leads to the exactness of the \v{C}ech cochain complex of any sheaf in degree $i > 0$, which is therefore fine, hence soft and acyclic. In particular, one has that the cohomology class of equation \eqref{cohoclass} vanishes. 
The second point is a consequence of the first one and the very definition of the $Ext^1$-module.
\end{proof}
\end{remark}
\begin{remark}
Notice that the extension \eqref{extomega} might indeed be \emph{non-split} if one starts from a complex or algebraic supermanifold. Studying the related cohomology class \eqref{cohoclass} is a non-trivial and interesting problem. 
\end{remark}
\noindent The above characterization of $\Omega^1_{\Pi {T}\mani}$ as an extension allows one to prove the following important property of $\Pi {T} \mani$. 
\begin{lemma}[$\mathcal{B}er(\Pi {T}\mani) \cong \pi^\ast \mathcal{O}_{\mani}$] \label{canvol} Let $\mani$ be a real supermanifold and let $\Pi {T} \mani$ be defined as above. Then there is a canonical isomorphism 
\bear
\mathcal{B}er(\Pi {T}\mani) \cong \pi^\ast \mathcal{O}_{\mani}.
\eear
\end{lemma} 
\begin{proof} From the extension exact sequence \eqref{extomega} and the definition of the Berezinian sheaf it follows that  
\bear
\mathcal{B}er (\Pi {T}\mani) \cong \pi^\ast \left ( \mathcal{B}er (\mani) \otimes \mathcal{B}er (\mathcal{T}^\ast_\mani) \right ),
\eear
Since for any locally-free sheaf $\mathcal{E}$ one has $\mathcal{B}er (\mathcal{E}^\ast) \cong \mathcal{B}er^\ast (\mathcal{E})$ and $\mathcal{B}er (\Pi \mathcal{E}) \cong \mathcal{B}er^\ast (\mathcal{E})$, then $\mathcal{B}er (\mani) = \mathcal{B}er(\Omega^1_\mani) \cong \mathcal{B}er (\mathcal{T}_\mani)$. In particular $\mathcal{B}er(\mathcal{T}_\mani^\ast) \cong \mathcal{B}er^\ast (\mani).$ It follows that 
\bear
\mathcal{B}er (\Pi T \mani) \cong \pi^\ast \left ( \mathcal{B}er (\mani) \otimes \mathcal{B}er^\ast (\mani) \right ) \cong \pi^\ast \mathcal{O}_\mani, 
\eear 
concluding the proof. \end{proof}
{\remark The above result could have also been proved upon using the explicit form of the transition functions for $\Omega^1_{\Pi {{T}} \mani}$, as given in Lemma \ref{transtanM}.   }
{\remark [$\Pi {T}\mani$ is a \virgolette Calabi-Yau'' Supermanifold] While in general there is \emph{no} natural choice for a section of the Berezinian sheaf on the supermanifold $\mani$, instead it is a crucial consequence of lemma \ref{canvol} that the associated supermanifold $\Pi {T}\mani$  comes endowed with a \emph{canonical Berezinian} or \emph{canonical volume form}, which is independent of the choice of coordinates on $\mani$, and therefore on $\Pi {T}\mani$. When working in the complex holomorphic category one would say that the complex supermanifold $\Pi {T}\mani$ is a \emph{Calabi-Yau supermanifold}. We will denote the canonical volume form on $\Pi {T} \mani$ with 
\bear
\mathcal{D}_{\Pi {T} \mani}(x, \mathpzc{F}) \in \mathcal{B}er(\Pi {T}\mani).
\eear}
{\remark [Integration on $\Pi {T} \mani$] Thanks to the existence of a canonical volume form for $\Pi {T} \mani$, any function on $\Pi {T} \mani$ can be mapped naturally to the Berezin integral on the supermanifold $\Pi {T} \mani$, \emph{i.e.} 
\bear \label{divint}
\mathcal{O}_{\Pi {{T}} \mani} \owns f (x, \mathpzc{F}) \longmapsto \int_{\Pi {T} \mani} \mathcal{D}_{\Pi {{T}} \mani}(x, \mathpzc{F}) f (x,\mathpzc{F}). 
\eear
It is important to notice, though, that the above integral can be \emph{divergent}, thus making the mapping ill-defined. This is related to the fact the supermanifold $\Pi {T}\mani$ is in general \emph{not} compact since the fibers of $\Pi {T}\mani$ above every point of the $p|q$-dimensional supermanifold $\mani$ are isomorphic to $\mathbb{R}^{q|p}$, see also \cite{Manin}. \\
An important class of functions in $\mathcal{O}_{\Pi {T}\mani}$ which are in general \emph{not} integrable over $\Pi {T}\mani$ are \emph{differential forms} on $\mani$, \emph{i.e.}\ functions on $\Pi {T} \mani$ which have polynomial dependence on the fiber coordinates. To see this consider the example in Remark \ref{pseudorem}. We take the superpoint $\mathbb{R}^{0|1}$ and the related $1|1$-dimensional supermanifold $\Pi {T}\mathbb{R}^{0|1}$. The differential form $P (\theta, d\theta) = \theta d\theta \in \Omega^1_{\mathbb{R}^{0|1}}$ is indeed an element of $\mathcal{O}_{\Pi {T} \scriptstyle{\mathbb{R}^{0|1}}}$ having polynomial dependence on the fibers, and whose integral is divergent since the ordinary Riemann-Lebesgue integral on the even variable $d\theta$ is divergent,
\bear
\int_{\Pi {T}\mathbb{R}^{0|1}} \mathcal{D} (\theta | d\theta) \, \theta d\theta =  \int_{\mathbb{R}^{0|1}} \left (\mathcal{D} (\theta) \theta \int_{\mathbb{R}^{1|0}} \mathcal{D} (d\theta) d\theta \right ).
\eear
On the other hand, the function $f (\theta, d \theta) = \theta e^{-(d\theta)^2}$ is integrable on $\cat{T}\mathbb{R}^{0|1}$, and indeed one has
\bear
\int_{\Pi {T}\mathbb{R}^{0|1}} \mathcal{D} (\theta | d\theta) \, \theta e^{-(d\theta)^2} =  \int_{\mathbb{R}^{0|1}} \left (\mathcal{D} (\theta) \theta \int_{\mathbb{R}^{1|0}} \mathcal{D} (d\theta) e^{-(d\theta)^2} \right ) = \sqrt{\pi}.
\eear
This justifies the following definition, which singles out the class of functions on $\Pi {T}\mani$ in relation to their Berezin integral \eqref{divint}, see also \cite{BL1, Manin, Witten}.}
\begin{definition}[Integrable Pseudodifferential Form] Let $\mani$ be a real supermanifold and let $\Pi {T}\mani$ be defined as above. We say that the pseudodifferential form $f\in \mathcal{O}_{\Pi {T} \mani}$ is integrable if its Berezin integral on $\Pi {T} \mani$ is convergent, \emph{i.e.}
$
\int_{\Pi {T} \mani}  \mathcal{D}_{\Pi {T} \mani} f  < \infty.
$ We will denote the sheaf of integrable pseudodifferential forms with $\mathcal{O}_{\Pi {T} \mani}^{\tiny{\cat{int}}}$. 
\end{definition}
\begin{remark}[Integral over the Fibers] Let us now restrict to \emph{integrable} pseudodifferential forms so that the integral \eqref{divint} makes sense, \emph{i.e.}\ we have a well-defined map
\bear \label{intTM}
\mathcal{O}_{\Pi {{T}} \mani}^{\tiny{\cat{int}}} \owns f_{\tiny{\cat{int}}} (x, \mathpzc{F}) \longmapsto \int_{\Pi {T} \mani} f_{\tiny{\cat{int}}} (x,\mathpzc{F})\mathcal{D}_{ \Pi {T} \mani}(x, \mathpzc{F}) \in {\mathbb{R}}. 
\eear
This integral can be understood as a (Berezin) integral \emph{along the fibers} of the fibration $\Pi {T} \mani \stackrel{\pi}{\longrightarrow} \mani$, \emph{i.e.}
\bear
 \pi_\ast: \mathcal{O}_{\Pi {T}\mani}^{\tiny{\cat{int}}} \longrightarrow \mathcal{B}er_{\mathpzc{int}} (\mani), 
\eear
followed by the \virgolette usual'' Berezin integral on the base supermanifold $\mani$, so that one has a map
\bear
\xymatrix@R=1.5pt{
\mathcal{O}_{\Pi {{T}}\mani}^{\tiny{\cat{int}}} \ar[rr]^{\pi_\ast} & &  \mathcal{B}er_{\mathpzc{int}} (\mani) \ar[rr]^{\int_\mani} && \mathbb{R} \\
f_{\tiny{\cat{int}}}  \ar@{|->}[rr] & & (\pi_\ast f_{\tiny{\cat{int}}})  \ar@{|->}[rr] && \int_{\mani} (\pi_{\ast} f_{\tiny{\cat{int}}}) . 
}
\eear
Loosely speaking, defining integration on $\Pi {T}\mani$ via the composition of the above maps, \emph{i.e.}\ $\int_{\Pi {T} \mani} \defeq \int_\mani \circ \; \pi_\ast$, has to be seen as a sort of factorization of the integral as an integral in the fiber, or \emph{vertical} directions followed by an integral over the base \cite{BottTu, Witten}. Notice that here $\mathcal{B}er_{\mathpzc{int}} (\mani)$ defines a Berezinian sheaf whose sections are integrable.

\noindent Nonetheless, it is important to stress that the map $\pi_\ast$ corresponds to a \emph{Berezin} integral. We choose a certain trivialization over an open set $U$ in the base of a $p|q$-dimensional real supermanifold $\mani$ and define the corresponding even and odd coordinate $x = y_i|\theta_\alpha$ and $\mathpzc{F} \defeq d\theta_{\alpha} | dy_i$ on the ${q|p}$-dimensional fibers over $U$ for $i = 1, \ldots p$ and $\alpha = 1, \ldots q$. By anticommutativity of the $dy$'s a function $f_{\tiny{\cat{int}}} (x , \mathpzc{F}) \in \Gamma_{\mathpzc{c}} (U, \mathcal{O}_{\Pi {{T}} \mani}^{\tiny{\cat{int}}})$ can be written as 
\bear
f_{\tiny{\cat{int}}} (x , \mathpzc{F}) = dy_{i_1} \ldots dy_{i_\ell} F_{\mathpzc{c}} (y_i, d\theta_\alpha | \theta_\alpha),
\eear
for $\ell \leq p$ and some compactly supported, integrable function $F_{\mathpzc{c}}$ in the coordinates $y$'s and $d\theta$'s. The map $\pi_\ast$ is thus defined so that it acts as a true Berezin integral on the odd fiber coordinates $dx$'s, \emph{i.e.} it yields
\bear
\pi_{\ast} \left ( f_{\tiny{\cat{int}}} (x , \mathpzc{F}) \right ) = \left \{
\begin{array}{lcl}
0 & & \qquad \ell < p \\
\left( \int_{\mathbb{R}^q}  F_{\mathpzc{c}} (y_i, d\theta_\alpha | \theta_\alpha) d\mu (d\theta_\alpha) \right ) \otimes \mathcal{D}_\mani (y_i | \theta_\alpha)& & \qquad \ell = p,
\end{array}
\right. 
\eear
where $d\mu (d\theta_a)$ is a (Lesbegue) measure for the real even variables $d\theta$'s. Notice the result of the integral is indeed a function in $\Gamma_{\mathpzc{c}} (U, \mathcal{O}_{\mani})$, depending on the coordinates $y$'s and $\theta$'s on the base manifold.\\
It can be proved that given a pair of open sets $U, V$ in $\mani$ with $U \cap V \neq \emptyset$ and defining $f_{\tiny{\cat{int}}}^U \in \mathcal{O}_{\Pi {{T}} \mani}^{\tiny{\cat{int}}} (U)$ and $f_{\tiny{\cat{int}}}^V \in \mathcal{O}_{\Pi {{T}}\mani}^{\tiny{\cat{int}}}(V)$, then $\pi_\ast (f^U_{\tiny{\cat{int}}}) = \pi_\ast (f_{\tiny{\cat{int}}}^V)$ in the intersection. It follows that for a certain open covering $\{ U_i\}_{i \in I} $ and the related trivializations, one finds that $\{ \pi_\ast f^{U_i}_{\tiny{\cat{int}}} \}_{i \in I}$ glue together, yielding a section $\pi_\ast f_{\tiny{\cat{int}}} \in \mathcal{B}er_{\mathpzc{int}} (\mani)$, see \cite{BottTu, Witten}.
\end{remark}
{\remark \label{WittenNot} In light of recent physics-oriented literature, a brief remark on notation is in order here. 
In the influential review \cite{Witten}, Witten denotes a section of the Berezinian of a supermanifold $\mani$ as $[dx | d\theta] \in \mathcal{B}er (\mani)$ for a choice of coordinates $x |\theta$ on $\mani$. Accordingly, the canonical volume form on $\Pi {T}\mani$ is denoted with 
\bear
[dx \, d(d\theta)  | d\theta \, d(dx) ] \in \mathcal{B}er(\Pi  {T}\mani),
\eear 
as to reminds that $d\theta$'s and the $dx$'s are now seen as coordinates for $\Pi {T}\mani$. In other words, the expressions $d(dx)$'s and $d(d\theta)$'s are just the symbols corresponding to the $d\mathpzc{F}$'s in the notation of lemma \ref{transtanM}. As such, one should not read them as the application of the de Rham differential $d$ on a local basis of the bundle of 1-forms, which would be vanishing as $d\circ d = 0$.}

\subsection{Special Class of Integrable Pseudoforms: Distributions on the Fibers} In the spirit of \emph{supersymmetric localization} \cite{Pestun} it is convenient to focus on a particular class of pseudoforms in $\mathcal{O}_{\Pi {T}\mani}^{\tiny{\cat{int}}}$, allowing only for a particular dependence of the fiber coordinates $X = d\theta_\alpha | dx_i$. These are the pseudoforms Belopolsky \cite{Belo}, Castellani, Catenacci, and Grassi \cite{CCG} and Witten \cite{Witten} focus on. Nonetheless, they have been introduced in the early days of the theory, see for example \cite{BL2, VZ2}.   
\begin{definition}[Delta Forms] \label{deltaforms} Let $\mani$ be a real supermanifold and let $\Pi T \mani$ be defined as above. We call delta forms the class of integrable pseudoforms in $\mathcal{O}_{\Pi T\mani}^{\tiny{\cat{int}}}$ whose dependence of the \emph{even} fiber coordinates of $\Pi T \mani$ is \emph{distributional} and \emph{supported at the origin}. We denote the sheaf of delta forms by $\mathcal{O}^{\delta}_{\Pi {T} \mani}.$
\end{definition}
\noindent Given an open set $U$ in $\mani$, in the local trivialization of $\Pi T \mani$ over $U$ with coordinates $x = x_i |\theta_\alpha $ and $\mathpzc{F} = d\theta_\alpha | dx_i$, a delta form $\omega \in \mathcal{O}^{\delta}_{\Pi T \mani} $ can be written as 
\bear \label{deltaf}
\omega_U (x, dx) = \sum_{\underline{\epsilon}_k, \underline{\ell}_j} f_{ \underline{\epsilon}_k, \underline{\ell}_j} (x_i | \theta_\alpha) (dx_{1})^{\epsilon_{1}}  \ldots  (dx_{p})^{\epsilon_{p}} \delta^{\ell_1} (d\theta_{1})  \ldots  \delta^{\ell_q} (d\theta_{q}),
\eear
for $\ell_j \geq 0$ and $\epsilon_k \in \{ 0, 1\}$ and where $f_{\underline{\epsilon}, \underline{\ell}} \in \pi^\ast \mathcal{O}_{\mani}$. The expressions $\delta^{\ell_j} (d\theta_j)$ are Dirac's delta distributions \cite{CCG, Witten} - and their derivative, when $\ell_j >0 $ - supported at the origin in the even real variable $d\theta_j$ of $\cat{T}\mani.$ They can be seen as linear functionals acting on differential forms, instead of on ordinary functions: these kinds of mathematical objects are called \emph{de Rham current} \cite{GH}. In this paper they will be treated in a formal algebraic fashion.  \\
\noindent The sheaf $\mathcal{O}^\delta_{\Pi {T} \mani}$ is endowed with several different structures, which interplay with the notion of integration on this particular class of pseudoforms. These structures are spelled out in the following remarks.


{\remark[$\mathcal{O}_\mani$-module structure of $\mathcal{O}^\delta_{\Pi {T} \mani}$] We first note that the sheaf $\mathcal{O}^{\delta}_{\Pi {T} \mani}$ is \emph{not} a sheaf of algebras. This comes from the well-known fact from analysis that the product of two distributions is not well-defined. On the other hand, $\mathcal{O}^{\delta}_{\Pi {T} \mani}$ carries the structure of a sheaf of $\pi^\ast \mathcal{O}_\mani$-modules, as the multiplication of delta forms by sections coming from the supermanifold $\mani$ is well-defined. }

{\remark[$\mathcal{D}_{\Omega^1_{\mani}}$-module structure of $\mathcal{O}^\delta_{\Pi {T} \mani}$] The sheaf of delta forms $\mathcal{O}^\delta_{\Pi {T} \mani}$ is a $\mathcal{D}$-module. More precisely, it carries the structure of $\mathcal{D}_{\pi^\ast\Omega^1_{\mani}}$-module. Indeed, locally, the \emph{Clifford-Weyl superalgebra} $\mathcal{CW}_{q|p} (\mathbb{R}) $ generated by $\{d\theta_\alpha, \partial_{d\theta_\alpha} | dx_i, \partial_{dx_i}\}$, for $i = 1, \ldots, p$ and $\alpha = 1, \ldots, q$ with non trivial (super)commutation relations given by
\bear
\left [ \frac{\partial}{\partial d\theta_\alpha}, d \theta_\beta \right ] = \delta_{\alpha \beta}, \qquad 	\quad \left \{ \frac{\partial}{\partial dx_i}, d x_j \right \} = \delta_{ij},
\eear
acts on sections in $\mathcal{O}^\delta_{\Pi {T} \mani} (U)$ according to the following definitions (given on monomials)
\begin{align} \label{eqD1}
dx_i \cdot & \left ( (dx_{1})^{\epsilon_{1}}  \ldots  (dx_i)^{\epsilon_i}   \ldots  (dx_{p})^{\epsilon_{p}} \delta^{\ell_1} (d\theta_{1})  \ldots  \delta^{\ell_q} (d\theta_{q}) \right )   \nonumber \\
& \defeq  
\delta_{0\epsilon_{i}}\, (-1)^{\sum_{j=1}^i \epsilon_j} (dx_{1})^{\epsilon_{1}} \ldots   {(dx_i)}  \ldots  (dx_{p})^{\epsilon_{p}} \delta^{\ell_1} (d\theta_{1})  \ldots  \delta^{\ell_q} (d\theta_{q})
\\
\frac{\partial}{\partial{dx_i}} \cdot & \left ( (dx_{1})^{\epsilon_{1}}  \ldots  (dx_i)^{\epsilon_i}   \ldots  (dx_{p})^{\epsilon_{p}} \delta^{\ell_1} (d\theta_{1})  \ldots  \delta^{\ell_q} (d\theta_{q}) \right )   \nonumber \\
& \defeq  
\delta_{1 \epsilon_i} \, (-1)^{\sum_{j=1}^i \epsilon_j}(dx_{1})^{\epsilon_{1}} \ldots  \widehat {dx_i}  \ldots  (dx_{p})^{\epsilon_{p}} \delta^{\ell_1} (d\theta_{1})  \ldots  \delta^{\ell_q} (d\theta_{q})
\\
\label{actiondtheta} {d\theta_\alpha} \cdot  & \left ( (dx_{1})^{\epsilon_{1}}  \ldots  (dx_{p})^{\epsilon_{p}} \delta^{\ell_1} (d\theta_{1})  \ldots  \delta^{\ell_\alpha} (d\theta_\alpha)  \ldots \delta^{\ell_q} (d\theta_{q}) \right )   \nonumber \\
& \defeq  
(-1)^{\ell_\alpha} \ell_{\alpha} (dx_{1})^{\epsilon_{1}}  \ldots  (dx_{p})^{\epsilon_{p}} \delta^{\ell_1} (d\theta_{1})  \ldots  \delta^{\ell_\alpha -1} (d\theta_\alpha)  \ldots  \delta^{\ell_q} (d\theta_{q}) 
\\ \label{eqD4}
\frac{\partial}{\partial{d\theta_\alpha}}  & \left ( (dx_{1})^{\epsilon_{1}}  \ldots  (dx_{p})^{\epsilon_{p}} \delta^{\ell_1} (d\theta_{1})  \ldots  \delta^{\ell_\alpha} (d\theta_\alpha)  \ldots \delta^{\ell_q} (d\theta_{q}) \right )   \nonumber \\
& \defeq  
(dx_{1})^{\epsilon_{1}}  \ldots  (dx_{p})^{\epsilon_{p}} \delta^{\ell_1} (d\theta_{1})  \ldots  \delta^{\ell_\alpha +1} (d\theta_\alpha)  \ldots  \delta^{\ell_q} (d\theta_{q}), 
\end{align}
where again $\epsilon_i = \{ 0,1\}$ and $\ell_{\alpha} \geq 0,$ for any $i = 1, \ldots, p $ and any $\alpha = 1, \ldots, q.$ Notice that the action \eqref{actiondtheta} can be seen as arising from \virgolette integration by parts'', and in particular one defines
\bear
d\theta_\alpha \cdot  \left ( (dx_{1})^{\epsilon_{1}} \wedge \ldots \wedge (dx_{p})^{\epsilon_{p}} \delta^{\ell_1} (d\theta_{1}) \wedge \ldots \wedge \delta^{0} (d\theta_\alpha) \wedge \ldots \delta^{\ell_q} (d\theta_{q}) \right )  = 0.
\eear
This is can be seen as the \virgolette de Rham current'' analog of usual relation that characterizes a Dirac delta distribution in functional analysis: working for simplicity over the real line $\mathbb{R}$, for $\delta_{x_0} \in \mathscr{D}^\prime ({\mathbb{R}})$, we have 
\bear \label{diracR}
\delta_{x_0} (f) = f(x_0),
\eear
for any point $x_0 \in \mathbb{R}$ and any test function $f \in \mathscr{S} (\mathbb{R}^n)$ of Schwartz class. The algebraic way to \eqref{diracR} consist into endowing the $\mathcal{O}_\mathbb{R}$-module $\{ \delta^{(\ell)}_{x_0}\}_{\ell \geq 0}$ of Dirac delta and its derivatives with a $\mathcal{D}_\mathbb{R}$-module structure, by defining the $\mathcal{D}$-action by
\bear
\frac{d}{dx} \cdot \delta^{(\ell)}_{x_0} = \delta^{(\ell+1)}_{x_0}, \quad (x-a) \cdot \delta^{(\ell)}_a = (-1)^{\ell} \ell \, \delta_a^{(\ell-1)}, \quad (x-a) \cdot \delta_a^{(0)} = 0.
\eear
In light of these, the previous \eqref{eqD1}-\eqref{eqD4} should not be surprising, as they are exactly the same relations, but given in the context of forms (or de Rham current) on a supermanifold.}

{\remark[$\mathbb{Z}_2$-grading of $\mathcal{O}^\delta_{\Pi {T} \mani}$] There is a twist in the description of the $\mathbb{Z}_2$-grading of the sheaf $\mathcal{O}^{\delta}_{\Pi {T} \mani}.$ Indeed, whereas the $d\theta$'s are even, the $\delta (d\theta)$'s and their derivatives are defined to be odd, so that their product has parity $q$. It follows that a generic monomial in the expression \eqref{deltaf} for $\omega_U $ above has parity given by $q + |f| + \sum_{k=1}^p \epsilon_k $,
where $f$ is assumed homogeneous in its $\mathbb{Z}_2$-degree. The reason behind the odd parity of the delta's can be seen by looking at the following integral \cite{Witten}    
\bear
I = \int_{\scriptsize{\cat{T}}\mathbb{R}^{0|2}} \theta_1 \theta_2 \delta (d\theta_1) \delta (d\theta_2) \mathcal{D} (d\theta | \theta ) 
\eear
Thanks to the delta's, when applying $\pi_\ast$ to integrate along the fibers $d\theta$'s, the integral is localized to the ordinary Berezin integral of $\theta_1 \theta_2$ over the superpoint $\mathbb{R}^{0|2}$, which yields 1. 
Exchanging $\theta_1 \leftrightarrow \theta_2$, would yield an integral equal to $-1$ instead, unless the transformation of $\delta (d\theta_1) \delta (d\theta_2)$ would correct it with $-1$. Formally, we thus say the delta's anticommute, \emph{e.g.} $\delta (d\theta_1)  \delta (d\theta_2) = - \delta (d\theta_2)  \delta (d\theta_1).$ }

{\remark[$\mathbb{Z}$-grading of $\mathcal{O}^\delta_{\Pi {T} \mani}$] \label{Zgrad} A $\mathbb{Z}$-gradation can also be introduced on $\mathcal{O}^{\delta}_{\Pi {T} \mani}$. Considering a section $\omega \in \mathcal{O}^{\delta}_{\Pi {T}\mani}$ trivialized as in \eqref{deltaf}, its $\mathbb{Z}$-degree is given by
\bear
\deg_{\mathbb{Z}} (\omega) = \sum_{k=1}^p \epsilon_k - \sum_{j=1}^q \ell_j,
\eear
which implies that the $\mathbb{Z}$-degree of a delta form is such that $-\infty < \deg_\mathbb{Z} (\omega) < p.$ This is stable under change of coordinates, and therefore the definition of the $\mathbb{Z}$-degree is well-posed, as we shall see shortly. We will call the sub-sheaf the degree-$k$ delta forms $\mathcal{O}^{\delta (k)}_{\Pi {T} \mani}$.\\
Notice that delta forms in degree $k \leq p $ can be generated via the $\mathcal{D}_{\Omega^1_\mani}$-module structure introduced above starting from the (unique, up to a multiplication by a section of $\mathcal{O}_\mani$) delta form in degree $p$, given by
\bear
\omega_{(p)}(x, d\theta | \theta, dx) = dx_1 \ldots  dx_p \delta(d\theta_1)  \ldots  \delta (d\theta_q),
\eear
in a certain choice of coordinates. Then, sections $ \omega_{(k)} \in \mathcal{O}^{\delta (k)}_{\Pi {T} \mani} $ of degree $k < p$ are obtained by acting with differential operators of order $p-k$ on $\omega_{(p)}$. These are locally constructed from the $\partial_{d\theta}$'s and $\partial_{dx}$'s. In particular, we will have that locally
\bear \label{deract}
\frac{\partial^{|I|}}{\partial d\theta^J \partial dx^K } \cdot \omega_{(p)} (x, d\theta | \theta, dx) \in \mathcal{O}^{\delta (p - |I|)}_{\Pi {T} \mani},  
\eear 
where $I$, $J$ and $K$ are multi-indices such that $I = (J,K)$ so that $|I| = |J| + |K|$. Notice that $K$ is such that $|K| \leq p$, since the $\partial_{dx}$'s are anticommuting, while $J$ can be of any order $|J| \geq 0$. On the other hand, given a form in degree $n<p$, a differential form of degree $k$ acts via the $\mathcal{D}_{\Omega^1_\mani}$-action defined above by raising the degree of the delta form by $k$. This means that
\bear \label{multact}
d\theta^I dx^K \cdot \omega_{(n)} (x, d\theta | \theta, dx) \in \mathcal{O}^{\delta (n+ |I|)}_{\Pi {T} \mani},  
\eear
if again $|I| = |J| + |K|$, for some multi-indices $I = (J,K).$ Clearly, the (local) $\mathcal{D}_{\Omega^1_{\mani}}$-action can also result in annihilating the delta form, both in the case of \eqref{deract} and \eqref{multact}, as it is clear from the \eqref{eqD1}-\eqref{eqD4}.\\
Notice that in the above picture delta forms can be seen locally as elements of a \emph{Fock space}, which is constructed via the $\mathcal{D}_{\Omega^1_\mani}$-action starting from a pivot $\omega_{(p)} \in  \mathcal{O}^{\delta (p)}_{\Pi {T}\mani}$ \virgolette state'', see \cite{Witten}. \\
Further, it is to be observed that, given the above definition, delta forms exist in the very same degrees as integral forms in $\Sigma^n_\mani$, where again $-\infty <n \leq p$. This does not happen by chance, indeed it is possible to prove that there exists an isomorphism between integral and delta forms. Before we see this, a remark on the transformation properties of the delta's is in order.}

{\remark[Transformations Properties of the Delta's] It has to be stressed that an expression involving any number of delta's which is lesser than the odd dimension $q$ of base supermanifold $\mani$, \emph{e.g.}\ $\delta (d\theta_1) \ldots  \delta (d\theta_{q-1})$ does \emph{not} make sense, as its transformation properties are \emph{not} well-defined. To see this in an informal way, let us consider a generic $1|2$ dimensional supermanifold and look at the transformation properties of a single delta $\delta (d\theta)$ under a change or coordinates $d\theta^\prime = \alpha d\theta + \beta d\psi + \gamma dx$, for $\alpha, \beta, \gamma \neq 0$, $\alpha, \beta$ even and $\gamma$ odd. First, we observe that since $dx$ is \virgolette infinitesimal'' (as it is nilpotent) we can formally expand about it in the following fashion
\begin{align} \label{transdelta1}
 \delta (d\theta^\prime) & = \delta \left (  {\alpha}d\theta + {\beta} d\psi + \gamma dx\right ) = \delta \left ( {\alpha} d\theta + {\beta}d\psi  \right ) +  \gamma dx \delta^\prime (\left ( {\alpha} d\theta + {\beta}d\psi  \right ).
\end{align}
Now the problem is to make sense out of the expression $\delta \left ( {\alpha} d\theta + {\beta}d\psi  \right )$ and its derivative. Keep working formally, focusing on the first summand, one could get to the following expression 
\begin{align} \label{illdef}
\delta \left ( {\alpha} d\theta + {\beta}d\psi  \right ) &= \delta \left (\alpha \left ( d\theta + \frac{\beta}{\alpha} d \psi \right )\right ) = \frac{1}{\alpha} \left ( \delta (d\theta) + \frac{\beta}{\alpha} d\psi \delta^\prime (d\theta) + \ldots \right ) \nonumber \\
& =  \sum_{k= 0}^\infty \frac{\beta^k}{k! \, \alpha^{k+1}} (d\psi)^k \delta^{k} (d\theta),
\end{align}
which would suggest that, if $\delta (d\theta)$ was as a section, the corresponding sheaf would not be locally-free of \emph{finite} rank. But clearly, this is just the tip of the iceberg, as there are more inconsistencies: in the first place $d\theta $ and $d\psi$ are honest even variables on $\Pi {T} \mani,$ so that the expression $\alpha d\theta$ and $\beta d\psi$ are not at all nilpotent nor infinitesimal. Further, also forgetting about this, one might have chosen to expand about $d\psi$ instead of $d\theta.$ \\
On the other hand, by definition, all of the $q$ delta's (or their derivatives) are required to appear in sections of $\mathcal{O}^\delta_{\Pi {T} \mani}$. We shall see in the next theorem that this requirement leads to well-defined transformation properties.}
\begin{theorem}[Delta Forms are Isomorphic to Integral Forms] \label{isodeltaint} Let $\mani$ be a real supermanifold of dimension $p|q$. Then the sheaf $\mathcal{O}^{\delta (k)}_{\Pi {T}\mani}$ of delta forms of degree $k$ is isomorphic to the sheaf $\Sigma^k_\mani$ of integral forms of degree $k$ for any $k \leq p$.
\end{theorem}
\begin{proof} To prove the statement is enough to verify that the transformation properties of the generating sections do coincide. Let us start from degree $p$, corresponding to $\Sigma^p_\mani = \mathcal{B}er (\mani)$. Without loss of generality, we can restrict ourselves to consider coordinate transformations $\varphi $ of $\mani$ of the split type, \emph{i.e.}\ $x_i^\prime = f_i (x)$ and $\theta_\alpha^\prime = \sum_{\beta} g_{\alpha \beta}(x) \theta_\beta$. The only degree $p$ delta form in $\mathcal{O}^{\delta (p)}_{\Pi {T}\mani}$ is given in a certain trivialization by 
\bear
\omega_{(p)} = dx_1 \ldots  dx_p \delta(d\theta_1)  \ldots  \delta (d\theta_q).
\eear
The part $dx_1 \ldots  dx_p $ contributes with the determinant of the Jacobian of the change of coordinates $x_i^\prime = f_i (x)$, while the Dirac-delta part contributes in the following way
\bear
\delta(d\theta_1^\prime)  \ldots  \delta (d\theta_q^\prime) = \delta \left (\sum_{\beta=1}^q (\partial_{\theta_\beta}\theta^\prime_1) d\theta_\beta  + \sum_{i=1}^p (\partial_{x_i} \theta^\prime_1) dx_i \right )   \ldots  \delta \left (\sum_{\beta=1}^q (\partial_{\theta_\beta} \theta^\prime_q) d\theta_\beta  + \sum_{i=1}^p (\partial_{x_i} \theta^\prime_q) dx_i \right ). \nonumber
\eear
The part proportional to $dx$'s does not contribute by \eqref{transdelta1}, due to the presence of $dx_1  \ldots dx_p$, so that one is left with 
\begin{align}
\delta(d\theta_1^\prime)  \ldots  \delta (d\theta_q^\prime) & = \delta \left (\sum_{\beta=1}^q ( \partial_{\theta_\beta} \theta^\prime_1) d\theta_\beta  \right )   \ldots  \delta \left (\sum_{\beta=1}^q ( \partial_{\theta_\beta} \theta^\prime_q) d\theta_\beta \right )  \nonumber \\
& = \left (\det \left (\frac{\partial \theta^\prime_\alpha}{\partial \theta_\beta}\right) \right )^{-1} \delta(d\theta_1^\prime)  \ldots  \delta (d\theta_q^\prime)
\end{align}
upon using the properties of Dirac's delta distributions. Putting the pieces together, we find
\bear
\omega^{\prime}_{(p)} = {\det \left (\frac{\partial x^\prime_i}{\partial x_j}\right)} \det \left ( \frac{\partial \theta^\prime_\alpha}{\partial \theta_\beta} \right)^{-1} \omega_{(p)} = \mbox{B}er (\mathcal{J}ac (\varphi))\,  \omega_{(p)}.
\eear
This settles the degree $p$ case. 
For degree lower than $p$ it is enough to observe that, following remark \ref{Zgrad}, delta forms in $\mathcal{O}^{\delta (p-k)}_{\Pi {T}\mani}$ are obtained via the $\mathcal{D}_{\Omega^1_\mani}$-action of differential operators of order $k$ on the above section $\omega_{(p)} \in \Sigma^p_\mani \cong \mathcal{O}^{\delta (p)}_{\Pi {{T}}\mani}$. In particular for degree $p-1$, working locally, we have an action of $\partial_{dx}$'s and $\partial_{d\theta}$'s: these can be as linear maps acting on $\pi^\ast \Omega^1_\mani$, hence they belong to $(\pi^\ast \Omega^1_\mani)^\ast \cong \pi^\ast \Pi \mathcal{T}_\mani$. It follows that the transformation of a section $\omega_{(p-1)} \in \mathcal{O}^{\delta (p-1)}_{\Pi {{T}} \mani}$ will be given by
\bear
\omega_{(p-1)}^\prime =  \mathcal{J}ac (\varphi)^{\Pi} \otimes \mbox{B}er (\mathcal{J}ac (\varphi)) )\,  \omega_{(p)}, 
\eear
where $\mathcal{J}ac (\varphi)^{\Pi}$ is the parity-transpose of the Jacobian of the change of coordinates, which identifies the transition functions of the locally-free sheaf $\Pi \mathcal{T}_\mani$. This yields the isomorphism $\mathcal{O}^{\delta (p-1)}_{\Pi {{T}}\mani} \cong \mathcal{B}er(\mani) \otimes \Pi \mathcal{T}_\mani = \Sigma^{p-1}_\mani:$ explicitly $\partial_{dx_i} \mapsto \pi \partial_{x_i}$ and $\partial_{d\theta_\alpha} \mapsto \pi \partial_{\theta_\alpha}$, for any $i = 1, \ldots, p$ and $\alpha = 1, \ldots, q$. In the very same fashion, higher degree differential operators are identified with sections on $\cat{S}^k \Pi \mathcal{T}_\mani$, thus showing that $\mathcal{O}^{\delta (p-k)}_{\tiny{\cat{{T}}}\mani} \cong \mathcal{B}er(\mani) \otimes \cat{S}^k \Pi \mathcal{T}_\mani = \Sigma^{p-k}_\mani.$ 
\end{proof}

{\remark[Integration theory] It follows from the previous theorem that integration on supermanifolds via integral forms $\Sigma^k_\mani$ parallels integration theory via delta forms $\mathcal{O}^{\delta (k)}_{\Pi {{T}} \mani}$ on the supermanifold $\Pi T \mani$. While this is clear in the case of Berezinian, which accounts for integration on the full supermanifold, it might be helpful to consider a codimension $1|0$ example. To this end, let us consider the supermanifold $\mathbb{R}^{1|2}$, with a system of coordinates given by $ x |\theta_1, \theta_2$. 
According to Bernstein and Leites \cite{BL1}, the integral form $\sigma_0 = \mathcal{D} (x|\theta_1, \theta_2) \theta_1 \theta_2 \otimes \pi \partial_{x} \in H^{0}_{\mathpzc{Sp}} (\mani)$ can be integrated over a codimension $1|0$ sub-supermanifold of $\mani$. In particular, let us consider the sub-supermanifold $\mathpzc{N} \stackrel{\iota}{\longhookrightarrow} \mathbb{R}^{1|2} $ defined as follows 
\bear \label{subsup}
\mathpzc{N} \defeq \{ x|\theta_1, \theta_2 \in \mathbb{R}^{1|2}: x= 0 \}.
\eear
As explained by Deligne and Morgan in \cite{Deligne}, local (even) equations $u_1 = \ldots = u_i = 0$ for a sub-supermanifold $\mathpzc{N}$ define a section $\omega_{\mathpzc{N}} \in \Omega^i_{\mani, \mathpzc{c}}$ (independent of the particular choice of equations), which gives a cohomology class in $H^{i}_{\mathpzc{dR},\mathpzc{c}} (\mani)$. In the case of the sub-supermanifold $\mathpzc{N}$ given above in \eqref{subsup} one has $\omega_{\mathpzc{N}} \in H^1_{\mathpzc{dR}, \mathpzc{c}} (\mani)$ and usually a formal representative of the form of a generalized section $\omega_{\mathpzc{N}} = \delta (x) dx$ is taken, where the Dirac delta distribution $\delta (x)$ can be seen as a suitable limit of compactly supported functions - see \cite{Deligne} and \cite{Witten} on this regard.
The restriction $\iota^{\ast} \sigma_0 = \sigma_0 \lfloor_{\mathpzc{N}}$ of the integral form $\sigma_{0} \in H^0_{\mathpzc{Sp}}(\mani)$ to $\mathpzc{N}$ is implicitly defined by imposing the following equality
\bear
\label{codimsub}
\int_{\mathpzc{N}} f  \sigma_0 \lfloor_{\mathpzc{N}} \defeq \int_\mani f \sigma_0 \cdot \omega_\mathpzc{N},
\eear
for $f$ any compactly supported function on $\mani,$ see \cite{Deligne}, Chapter 3, Section 14. Based upon the previous \eqref{codimsub}, the section $\omega_{\mathpzc{N}}$ is given the interpretation of the \emph{Poincaré dual form} of the sub-supermanifold $\mathpzc{N}$.
It then follows from equation \eqref{codimsub} that in the above example one gets
\bear
\int_{ \mathpzc{N} \subset \mathbb{R}^{1|2}} \sigma_0 \lfloor_{\mathpzc{N}} = \int_{\mathbb{R}^{1|2}} \mathcal{D} (x|\theta_1 \theta_2) \theta_1 \theta_2 \otimes \pi \partial_{x} \cdot \delta (x) dx = \int_{\mathbb{R}^{1|2}} \mathcal{D}(x | \theta_1, \theta_2) \theta_1 \theta_2 \delta (x) = 1,
\eear
where we have used the duality pairing between $\pi \partial_x $ and $dx,$ given by $dx (\pi \partial_x) = 1$ and the properties of the Dirac delta distribution. \\
Similarly, via delta forms, one first observes that the integral form $\sigma_0$ corresponds to the delta form $\theta_1 \theta_2 \delta (d\theta_1) \delta (d\theta_2)$. Now, instead of duality, the integral uses the $\Omega^\bullet_\mani$-module structure (or $\mathcal{D}_{\Omega^1_\mani}$-structure), multiplying $\sigma_0 \in \mathcal{O}^{\delta (0)}_{\Pi T\mani}$ in the delta form representation by the form $\omega_\mathpzc{N}$ given as above: this yields a delta form in $\mathcal{O}^{\delta (1)}_{\Pi {T} \mani}$. More precisely, one has
\begin{align}
\int_{ \mathpzc{N} \subset \mathbb{R}^{1|2}} \sigma_0 \lfloor_\mathpzc{N} \defeq \int_{\Pi T \mathbb{R}^{1|2}} & \mathcal{D} (x, d\theta_1, d\theta_2 |\theta_1, \theta_2, dx)   \theta_1 \theta_2 \delta(d\theta_1) \delta (d\theta_2) \delta (x) dx \nonumber \\
& = \int_{\mathbb{R}^{1|2}} \mathcal{D}(x | \theta_1, \theta_2) \theta_1 \theta_2 \delta (x) = 1. 
\end{align}
Notice that the integral of $\sigma_0 \lfloor_{\mathpzc{N}}$ on $\mathpzc{N}$ is well defined as it only depends on the (co)homology classes. Let us indeed consider the sub supermanifold  
\bear
\widetilde{\mathpzc{N}} \defeq \{ x|\theta_1, \theta_2 \in \mathbb{R}^{1|2} : x + \theta_1 \theta_2 = 0\} \subset \mathbb{R}^{1|2}.
\eear
Notice that setting the odd variables to zero one has $\widetilde{\mathpzc{N}}_{\, \mathpzc{red}} = \mathpzc{N}_{\, \mathpzc{red}}$, and hence one expects the integrals over $\mathpzc{N}$ and $\widetilde{\mathpzc{N}}$ to be equal. We have $\omega_{\widetilde{\mathpzc{N}}} = \delta (x)  dx + d\theta_1 \theta_2 - \theta_1 d\theta_2$ and it is easy to see that 
\bear
\int_{\widetilde{\mathpzc{N}}} \sigma_0  = \int_{\mani} \sigma_0\cdot \omega_{\widetilde{\mathpzc{N}}} = \int_{\mani} \sigma_0 \cdot (\omega_{\mathpzc{N}} + d\eta ) = \int_{\mani} \sigma_0 \cdot \omega_{\mathpzc{N}} = \int_{\mathpzc{N}} \sigma_0,
\eear
where $d\eta = d (\theta_1 \theta_2) = d\theta_1 \theta_2 - \theta_1 d\theta_2$. Notice that this consideration is completely general - and not limited to the present easy example, as it relies on Stokes theorem \ref{stokesthm} - indeed any summand of the kind $\sigma_0 \cdot d\eta$ is exact and does not contribute to the Berezin integral - and Poincaré duality \ref{PD}.}

{\remark[Delta Forms and Pseudoforms] Delta forms have been defined by requiring that \emph{all} of the coordinate $d\theta$'s have distributional dependence of Dirac delta type. It was this very requirement that allowed us to prove theorem \ref{isodeltaint} above, thus showing that the formalism of delta forms is equivalent to that of integral forms that we have previously defined. \\
Nonetheless, the above requirement can be relaxed to a less stringent one, allowing, for example, a \virgolette mixed setting'', in which some of the $d\theta$'s have distributional dependence (hence of the kind of a delta-integral form) and the remaining have a polynomial dependence (hence of the kind of a differential form) \cite{Witten}. Even if there are important mathematical problems related to this framework - as we shall see -, this particular kind of (generally non-integrable) {pseudoforms} is the one which is considered in superstring perturbation theory \cite{Witten, WittenSuper}. In this context the number of localized variables $d\theta$'s is referred to as \emph{picture number} $\mathpzc{p}$ of the (pseudo)form. In particular, forms having picture number $\mathpzc{p}= 0$ are differential forms and forms having maximal picture number $\mathpzc{p} = q$ (which equal the odd dimension $q$ of the supermanifold) are integral forms. An example of pseudoform having middle dimensional - \emph{i.e.}\ non minimal and non maximal - picture number can be given considering again the easy case of $\mathbb{R}^{1|2}$. The most general pseudoform of picture $\mathpzc{p} = 1$ is given by
\begin{align}
\omega (x, d\theta | \theta, dx)& = \sum_{k_1, k_2 = 0}^{\infty}  f_{k_1, k_2}(x | \theta) (d\theta_1)^{k_1} \delta^{(k_2)} (d\theta_2) + \sum_{\ell_1, \ell_2 = 0}^\infty g_{\ell_1, \ell_2} (x | \theta) (d\theta_2)^{k_2} \delta^{(k_1)} (d\theta_1)  \nonumber  \\
& \quad + \sum_{i_1, i_2 = 0}^\infty h_{i_1, i_2} (x|\theta) dx (d\theta_1)^{i_1} \delta^{(i_2)} (d\theta_2) + \sum_{i_1, i_2 = 0}^\infty c_{i_1, i_2} (x|\theta) dx (d\theta_2)^{j_2} \delta^{(j_1)} (d\theta_1),
\end{align}
where $f, g, h, c \in \mathcal{O}_{\mathbb{R}^{1|2}}$ for any choice of indices. The degree of this kind of forms is defined to agree with the definition given for differential and integral forms: a $k$-derivative of any of the delta's counts $-k$. It is thus easy to see that there exists pseudoforms of middle picture $1 < \mathpzc{p} < q $ at \emph{any} degree, whereas differential forms have always non-negative degree and we have defined integral forms so that they have degree lower or equal than the even dimension of the supermanifold. Further, working in the same way as above, one sees that any module of pseudoforms of a certain middle dimensional picture $1 < \mathpzc{p} < q$ at a fixed degree $k \in (- \infty, + \infty)$ - we call it $\Omega^{k, \mathpzc{p}}_\mani$ - has an \emph{infinite} number of generators, hence cannot be described as a vector bundle or a locally-free sheaf of $\mathcal{O}_\mani$-modules of finite rank. For example, pseudoforms of picture $1$ and degree $k$ on $\mathbb{R}^{1|2}$ would be generated by  
\bear
\Omega^{k, 1}_{\mathbb{R}^{1|2}} = \mathcal{O}_{\mathbb{R}^{1|2}} \cdot \{ (d\theta_1)^{\ell_1 }\delta^{(\ell_2)} (d\theta_2), dx (d\theta_1)^{j_1} \delta^{j_2} (d\theta_2), \; 1 \leftrightarrow 2 \},
\eear
for any $\ell_1 - \ell_2 = k$ and $ j_1 - j_2 = k - 1$. On the other hand, the crucial problem which prevents us from having a well-given mathematical definition of these modules of pseudoforms with a middle-dimensional picture $0 <\mathpzc{p} < q$ is rooted in the ill-defined transformation of a single delta (and its derivative) $\delta^{(k_i)} (d\theta_i)$. Indeed, as explained above in the discussion around equation \eqref{illdef}, a single delta $\delta (d\theta)$ does not define a section of any vector bundle on $\mani$ and in particular it does not survive a change of coordinate. As it stands, a single $\delta (d\theta)$ - and more in general an expression consisting of a non-maximal number of delta's - is no more than a symbol: only the delta forms where all of the $q$ $d\theta$'s are present and have distributional dependence of the Dirac delta type do indeed yield well-defined sections of a vector bundle, as seen in theorem \ref{isodeltaint}. \\
It is to be noted, though, that when restricted to a specific immersed sub-supermanifold of the right codimension in $\mani$, pseudoforms of non-maximal picture are well-behaved, as they define the integral forms of the sub-supermanifold. As such, pseudoforms of non-maximal picture $0 < \mathpzc{p} < q$ are seen to be in relation to the integration over general sub-supermanifolds of codimension $k|q-\mathpzc{p}$, for any $k = 0, \ldots, p,$ depending on the degree of the pseudoform. }

\begin{remark}[Densities \& Super Grassmannians: a Teaser] As sketched above, pseudoforms should serve the purpose of having available objects that can be integrated over generic immersed sub-supermanifolds of a supermanifold, \emph{i.e.}\ $\iota: \mathpzc{N} \hookrightarrow \mani$, with $\mathpzc{N}$ of generic dimension $d \defeq d_1 | d_2 \leq p|q = \dim \mani.$ Geometrically, this leads to consider (relative) super Grassmannians over $\mani$ of the form $\cat{G}^d_\mani \defeq \cat{G} (d_1 |d_2 ; \mathcal{T}_\mani) \stackrel{\pi}{\longrightarrow} \mani $, see \cite{Manin, NojaG}. Notice that over a point $x \in \mani$, the fiber of the map $\pi$ is the ordinary super Grassmannian of $d_1 | d_2$-dimensional subspaces inside the vector superspace $\mathbb{R}^{p|q}$, \emph{i.e.}
\bear
\pi^{-1} (x) = \cat{G} (d_1 | d_2 ; \mathbb{R}^{p|q}).
\eear
Any relative super Grassmannian $\cat{G}^d_\mani$ comes endowed with its (relative) tautological sheaf $\mathcal{S}_{\cat{\tiny{G}}}$, see \cite{Manin}: this enters the definition of $d$-density.
\begin{definition}[$d$-Density] Let $\mani$ be a supermanifold of dimension $p|q$ and let $\cat{G}_\mani^d$ be a relative super Grassmanian over $\mani$ as defined above. Then a $d$-density on $\mani$ is a section of sheaf $\mathcal{B}er (\Pi \mathcal{S}_{\cat{\tiny{G}}}^\ast)$ on $\cat{G}_\mani^d$, for $d = d_1 | d_2 \leq p|q$ and $\mathcal{S}_{{\cat{\tiny{G}}}}$ the tautological sheaf on $\cat{G}_{\mani}^d$. We will denote this sheaf with $\mathfrak{D}^d$, to stress its dependence on the dimension $d = d_1 | d_2$.
\end{definition}
\noindent Now, upon restricting to sections of $\mathfrak{D}^d$ which are polynomial along the fibers in $\pi^{-1} (U)$, for $U\subset \mani$ (or working directly in the holomorphic or algebraic category), it turns out that in the extremal cases $d = k| 0$ and $d = k | q$ there are maps to differential forms and integral forms. 
\begin{theorem}[\cite{Deligne, Manin}] Let $\mani$ be a supermanifold of dimension $p|q$ and let $\mathfrak{D}^d$ be the sheaf of $d$-densities on it. Let $U$ be an open set in $\mani$, then the following are true.
\begin{enumerate}[leftmargin=*]
\item If $d = k | 0$, then there exists a map 
\bear
\varphi_{k|0} : \Gamma (U, \Omega^k_{\mani}) \stackrel{\sim}{\longrightarrow} \Gamma (\pi^{-1}(U), \mathfrak{D}^{k|0});
\eear 
\item if $d = k|q$, then there exists a map
\bear
\varphi_{k|0} : \Gamma (U, \Sigma^k_{\mani}) \stackrel{\sim}{\longrightarrow} \Gamma (\pi^{-1}(U), \mathfrak{D}^{k|n}),
\eear 
\end{enumerate}
In particular, a $p|q$-density on $\mani$ is a section of the Berezinian sheaf $\mathcal{B}er(\mani).$
\end{theorem}
\noindent In the remaining cases, for $0 < d_2 < q,$ there are no such sections and only $d$-densities are known. It would be then interesting to elucidate the relations between pseudoforms of non-maximal picture $\mathpzc{p}$ and the $d$-densities for intermediate dimension as defined above: this would be a crucial advance toward the mathematization of the formalism introduced by Witten and Belopolski for string theory purposes. Together with other topics, this problem is currently under investigation.

\end{remark}

\clearpage

\appendix

\section{Nilpotent Operators in Superalgebra} \label{app1}

\noindent In this appendix we report an easy yet very useful result that gives a criterion to establish the nilpotency of an operator. This appears as lemma 3 in \cite{Manin}, chapter 3, section 4 - the reader be advised of a little confusing misprint in the proof in the given reference. 
Let us consider the following generic setting: let ${S}$ and ${T}$ be $A$-modules, for $A $ a supercommutative ring and let $\sigma_a : S \rightarrow S$ and $\tau_a : T \rightarrow T$ two families of homogeneous homomorphisms for a finite set of indices $a= 1, \ldots, n.$ \\
We say that $\sigma_a $ and $\sigma_b$ \emph{commute} if
\bear
[\sigma_a, \sigma_b] \defeq \sigma_a \sigma_b - (-1)^{|\sigma_a| |\sigma_b|} \sigma_b \sigma_a =0
\eear  
We say that $\sigma_a$ and $\sigma_b$ \emph{anticommute} if 
\bear
\{\sigma_a, \sigma_b\} \defeq \sigma_a \sigma_b + (-1)^{|\sigma_a| |\sigma_b|} \sigma_b \sigma_a = 0
\eear  
Also we assume the following: 
\begin{enumerate}[leftmargin=*]
\item $|\sigma_a| + |\tau_a| $ does not depends on $a$;
\item the pairs $(\sigma_a, \sigma_b)$ and $(\tau_a, \tau_b)$ either commute or anticommute.   
\end{enumerate}
Then, it makes sense to define the following operator
\bear
\xymatrix@R=1.5pt{
d \defeq \sum_a \sigma_a \otimes \tau_a : S \otimes_A T \ar[r] & S \otimes_A T \nonumber \\ 
f = s \otimes t \ar@{|->}[r] & d (f)  \defeq \sum_a (-1)^{|\tau_a| |s|} \sigma_a (s) \otimes \tau_a (t).
}
\eear
We note that it has a well-defined parity (even or odd), since we have assumed that the parity of $|\sigma_a| + |\tau_a|$ does not depend on $a$. We have the following lemma. 
\begin{lemma} \label{lemmaA1} Let $d \defeq \sum_a \sigma_a \otimes \tau_a$ be as above, then if either one of the following is satisfied    
\begin{enumerate}[leftmargin=*]
\item $d$ is \emph{even} and the pair $(\sigma_a, \sigma_b)$ and $(\tau_a, \tau_b)$ have opposite commutation rules 
\item $d$ is \emph{odd} and the pair $(\sigma_a, \sigma_b)$ and $(\tau_a, \tau_b)$ have the same commutation rules 
\end{enumerate}
then $d \circ d = 0.$
\end{lemma} 
\begin{proof} Let us consider the expression for $d^2 \defeq d\circ d$. One finds the sum
\bear
d^2 \owns (\sigma_a \otimes \tau_a ) (\sigma_b \otimes \tau_b ) + (\sigma_b \otimes \tau_b )(\sigma_a \otimes \tau_a )
\eear
for some $a,b$. Commuting one has 
\begin{align} \label{sum}
& (\sigma_a \otimes \tau_a ) (\sigma_b \otimes \tau_b ) + (\sigma_b \otimes \tau_b )(\sigma_a \otimes \tau_a ) \nonumber \\
& \quad = ( (-1)^{|\sigma_b| |\tau_a|} + (-1)^{|\sigma_a| |\tau_b| + |\sigma_a| |\sigma_b| + |\tau_a| | \tau_b|} \epsilon_\sigma \epsilon_\tau) (\sigma_a \sigma_b \otimes \tau_a \tau_b) 
\end{align}
where $\epsilon_\sigma = \pm 1$ and $\epsilon_\tau = \pm 1$ depending on the commutation relations of $\sigma$ and $\tau$. \\
Say $d $ is even. Then this implies that $|\sigma_a| =|\tau_a|$ and $|\sigma_a| =|\tau_a|$ and $\epsilon_\sigma \epsilon_\tau = -1$ since one pair of morphisms commutes and the other anticommutes. Then one finds that
\bear
(-1)^{|\sigma_a| |\sigma_b|} - (-1)^{|\sigma_a| |\sigma_b| + |\sigma_a| |\sigma_b| + |\sigma_a| | \sigma_b|} = (-1)^{|\sigma_a| |\sigma_b|} - (-1)^{|\sigma_a| |\sigma_b|} = 0.
\eear
Say $d$ is odd. Then this implies that $|\sigma_a| =|\tau_a|+1$ and $|\sigma_b| =|\tau_b| +1$ and $\epsilon_\sigma \epsilon_\tau = +1$ since the two pairs of morphisms both commutes or anticommutes. Then one finds
\begin{align}
 & (-1)^{|\sigma_b| |\sigma_a| +|\sigma_b| } + (-1)^{|\sigma_a| |\sigma_b| + |\sigma_a| + |\sigma_a| |\sigma_b| + |\sigma_a| |\sigma_b| + |\sigma_a| + |\sigma_b| +1}  \nonumber \\
 & = {(-1)^{|\sigma_b| |\sigma_a| +|\sigma_b| }} + (-1)^{|\sigma_a| |\sigma_b| + |\sigma_b| + 1} \nonumber \\
& = {(-1)^{|\sigma_b| |\sigma_a| +|\sigma_b| }} - (-1)^{|\sigma_a| |\sigma_b| + |\sigma_b| } = 0
\end{align}
Therefore, in both cases, the sum in the parenthesis in \eqref{sum} vanishes.
\end{proof}
\noindent Notice that, even if it is given in the context of superalgebra, the above lemma \ref{lemmaA1} applies to a broad range of geometrical constructions - for example, the ordinary de Rham differential can be immediately proved to be nilpotent as a consequence of lemma \ref{lemmaA1} as $\sigma_a = dx_a$ and $\tau_a = \partial_a$.

\section{Right $\mathcal{D}$-modules and Canonical Sheaf} \label{app2}

\noindent In this appendix we prove lemma \eqref{DrightX} and we comment further on the relations between the Lie derivative and the $\mathcal{D}_M$-module structure of the canonical sheaf of an ordinary manifold $M$. The interested reader is invited to compare and appreciate the similarities of these constructions with those of section \ref{LieBerRight}, where the right $\mathcal{D}_\mani$-module structure on the Berezinian sheaf is discussed, for $\mathpzc{M}$ a supermanifold.

\begin{lemma}[$\omega_M$ is a Right $\mathcal{D}_M$-module] \label{DrightXApp} Let $M$ be a real or complex manifold and let $\Omega^{\dim M}_M$ be its canonical sheaf. Then $\Omega^{\dim M}_M$ is a sheaf of right $\mathcal{D}_M$-modules.
\end{lemma}
\begin{proof} By the previous theorem \ref{rightD}, it is enough to show that we can define a flat right connection on $\omega_X$. For sections $\omega^{top } \in \Omega^{\dim M}_M$ and $f \in \mathcal{O}_M$ and $X \in \mathcal{T}_M$ we give the following definition
\begin{align}
&\Delta_\mathpzc{R} (\omega^{top}\otimes f) \defeq \omega^{top} f \nonumber \\
&\Delta_\mathpzc{R} (\omega^{top} \otimes X) \defeq - \mathcal{L}_X (\omega^{top}).
\end{align}
First off, observe that 
\begin{align}
\Delta_{\mathpzc{R}} (\omega^{top} \otimes f \circ X) & = \mathcal{L}_{fX} (\omega^{top}) = d \circ \iota_{fX} \omega^{top} = d (f \iota_X \omega^{top}) \nonumber \\
& = df \wedge \iota_X (\omega^{top}) + f d (\iota_X (\omega^{top}))
\end{align}
On the other hand, one has
\begin{align}
\Delta_{\mathpzc{R}} (\omega^{top}f \otimes X) & = \mathcal{L}_{X} (\omega^{top} f)  = \iota_X \circ d ( \omega^{top} f) + d \circ \iota_X (\omega^{top} f) \nonumber \\
& = \iota_X (df \wedge \omega^{top} + f \wedge d\omega^{top}) + d \circ ( \iota_X (\omega^{top}) f) \nonumber \\
& = df \wedge \iota_X (\omega^{top}) + f d (\iota_X (\omega^{top}))
\end{align}
since $df \wedge \omega^{top}  = 0 = d\omega^{top}$, so that $\Delta_{\mathpzc{R}} (\omega^{top} \otimes f \circ X) = \Delta_{\mathpzc{R}} (\omega^{top}f \otimes X),$ which is the third defining property of a right connection. Further, we have that 
\bear
\Delta_{\mathpzc{R}} (\omega^{top} \otimes X  ) f = - \mathcal{L}_X (\omega^{top}) f,
\eear
but also 
\begin{align}
\Delta_\mathpzc{R} (\omega^{top} \otimes X \circ f) 
& = - \mathcal{L}_{fX} (\omega^{top}) + \omega^{top} X(f) = - \mathcal{L}_X ( \omega^{top} f) + \omega^{top} X(f) \nonumber \\
& = - \mathcal{L}_X (\omega^{top})f - \omega^{top}X(f) + \omega^{top} X(f)  = - \mathcal{L}_X (\omega^{top}) f,
\end{align}
so that we have indeed $\Delta_\mathpzc{R} (\omega^{top} \otimes X \circ f) = \Delta_{\mathpzc{R}} (\omega^{top} \otimes X  ) f, $ which proves the second defining property for a right connection. Finally, it is an obvious property of the Lie derivative that $\mathcal{L}_{[X, Y]} = [\mathcal{L}_X, \mathcal{L}_Y]$, which settles flatness.
\end{proof}
\remark Notice that it is crucial for the above to hold true that $\omega^{top}$ is really a section of the canonical sheaf. In other words, $\Omega^i_{M}$ is not a right $\mathcal{D}_M$-module unless $i = \dim M$. 
Also, notice that working locally in a chart $ U \subset M$ with local coordinates $x_1, \ldots, x_n$ such that a section of the canonical sheaf $\Omega^{\dim M}_M $ over $U$ reads $\omega^{top}_U = \omega (x) f$ for some functions $f \in \mathcal{O}_M(U)$ and $\omega(x) = dx_1 \wedge \ldots \wedge dx_{n} $ and considering a vector fields over $U$ such that $X_U = \sum_i X^i \partial_{x_i}$, then one easily has that 
\bear \label{LieComm}
\mathcal{L}_{X} (\omega^{top}) = \omega (x) \sum_{i=1}^n \partial_{x_i } ( X^i f ),
\eear 
indeed 
\begin{align}
\mathcal{L}_{X} (\omega^{top}) & = \mathcal{L}_X (\omega (x)) f + \omega(x) \mathcal{L}_X (f) =  d \iota_X (\omega (x)) f + \omega (x)  X (f) \nonumber \\
& = \omega (x) \partial_{x_i} (X^i) f + \omega (x) X^i \partial_{x_i} (f)  = \omega (x) \partial_{x_i} (X^i f).
\end{align}
Starting from the above \eqref{LieComm}, it can be seen that there exists a \emph{unique} right connection on $\Omega^{\dim M}_M$ satisfying the condition $\Delta_{\mathpzc{R}} (\omega(x) \otimes \partial_{x_i}) = 0 $ for all $i = 1, \ldots, n$ in any coordinate system. In particular, the following holds true.
\begin{lemma} \label{LieLemma} Let $M$ be a real or complex manifold and of dimension $n$ let $\Omega^{n}_X$ be its canonical sheaf. Then there exists a unique right connection on $\Omega^{n}_X$ such that 
\bear
\Delta_{\mathpzc{R}} (\omega(x) \otimes \partial_{x_i}) = 0
\eear
for any $i= 1, \ldots, n$ and for all system of local coordinates $(U, x_1, \ldots, x_n)$, with $\omega (x) = dx_1\wedge \ldots \wedge dx_n$ a generating section of $\Omega^{n}_M $ and $\partial_{x^i}$ is a coordinate vector field over $U$.
\end{lemma}   
\begin{proof} It is immediate using \eqref{LieComm} to see that indeed $\Delta_{\mathpzc{R}} (\omega (x) \otimes \partial_{x_i}) = - \mathcal{L}_{\partial_{x_i}} (\omega(x)) = 0$. Uniqueness follows from the fact that $\{ \partial_{x_i} \}_{i = 1, \ldots, n}$ is a system of generators for $\mathcal{T}_\mani$ and $\omega (x) = dx_1 \wedge \ldots \wedge dx_n$ is a generator for $\Omega^{n}_X$. It is an exercise to check that changing coordinates to $x^\prime_i = x^\prime_i (x) $ one still gets $\Delta_{\mathpzc{R}} (\omega (x^\prime) \otimes \partial_{x_i^\prime}) = 0$ for any $i$, thus concluding the proof. 
\end{proof}

\remark Notice that the above Lemma \ref{LieLemma} can be rephrased in terms of $\mathcal{D}_M$-module theory by saying that the right $\mathcal{D}_M$- module structure on $\Omega^n_M$ is uniquely characterized by the right action 
\bear
\omega (x) \cdot \partial_{x_i} = 0.
\eear 
for any $i = 1, \ldots, \dim M.$ This is to be related to Theorem \ref{BerRightTheo} and Corollary \ref{BerRightCor}


\begin{thebibliography}{99}

\bibitem{AC} D.\ Alekseevsky, V. Cortes, \emph{Classification of N-(super)-Extended Poincaré Algebras and Bilinear Invariants of the Spinor Representation of $Spin(p,q)$}, Comm.\ Math.\ Phys.\ \emph{183} (3) (1997) 477-510


\bibitem{BR} C.\ Bartocci, U.\ Bruzzo, D.\ Hern\`andez Ruiperez, \emph{The Geometry of Supermanifolds}, Reidel (1991)


\bibitem{Bat} M.\ Batchelor, \emph{The Structure of Supermanifolds}, Trans.\ Am.\ Math.\ Soc.\ {\bf 253} (1979) 329-338


\bibitem{Belo} A.\ Belopolsky, \emph{De Rham cohomology of the supermanifolds and BRST cohomology}, Phys. Lett. {\bf B} 403 (1997) 47-50

\bibitem{Belo1} A.\ Belopolsky, \emph{Picture Changing Operators in Supergeometry and Superstring Theory}, arXiv:9706033

\bibitem{Belo2}  A.\ Belopolsky, \emph{New Geometrical Approach to Superstrings},  hep-th/9703183 (1997)

\bibitem{Bettadapura1} K.\ Bettadapura, \emph{Higher Obstructions of Complex Supermanifolds}, SIGMA {\bf 14} (2018), 094

\bibitem{Bettadapura2} K.\ Bettadapura, \emph{Obstructed Thickenings of Supermanifolds}, J.\ Geom.\ Phys.\ {\bf 139} (2019) 25-49

\bibitem{Berezin} F.A.\ Berezin, \emph{Introduction to Superanalysis}, D. Reidel Publishing (1987) 

 \bibitem{BL1} J. Bernstein, D. Leites, \emph{Integral forms and Stokes formula on supermanifolds}, Funct. Anal. Appl.
11, 1, 55-56 (1977).

\bibitem{BL2}  J. Bernstein, D. Leites, \emph{How to integrate differential forms on supermanifolds}, Funct. Anal. Appl.
11, 3, 70-71 (1977).

\bibitem{BerLei} J.N.\ Bernstein, D.A.\ Leites, \emph{Invariant Differential Operators and Irreducible Representation of Lie Algebras of Vector Fields}, Serdica, Bulg. Math. Journ. {\bf 7} (1981) 320-334

\bibitem{BottTu} R.\ Bott, L.W.\ Tu, \emph{Differential Forms in Algebraic Topology}, Springer (1982)

\bibitem{P2} S.\ Cacciatori, S.\ Noja, R.\ Re, \emph{Non Projected Calabi-Yau Supermanifolds over $\mathbb{P}^2$}, Math.\ Res.\ Lett.\ {\bf 26} (4) (2019) 1027-1058

\bibitem{CNR} S.L.\ Cacciatori, S.\ Noja, R. Re, \emph{The Universal de Rham / Spencer Double Complex on a Supermanifold}, Doc.\ Math.\ {\bf 27} (2022) 487-518

\bibitem{CDF} L.\ Castellani, R. D'Auria, P. Fre, \emph{Supergravity and Superstring: a Geometric Perspective}, Vol 1, World Scientific (1991)

\bibitem{CCG} R. Catenacci, M.\ Debernardi, P.A. Grassi, D.\ Matessi, \emph{Cech and de Rham Cohomology of Integral Forms}, J. Geom.\ Phys.\ {\bf 62}, 4 (2012) 890-902

\bibitem{CCG1} L.\ Castellani, R. Catenacci, P.A. Grassi, \emph{The Integral Form of Supergravity}, JHEP 10 (2016) 049

\bibitem{SQM} L.\ Castellani, R. Catenacci, P.A. Grassi, \emph{Super Quantum Mechanics in the Integral Forms Formalism}, Ann. Henri Poincaré {\bf 19} (5), 1385-1417 (2018)

\bibitem{CGNinf} R. Catenacci, P.A. Grassi, S. Noja, \emph{$A_\infty$-Algebra from Supermanifolds}, Ann.\ Henri Poincar\'{e} {\bf 20} (12) 4163--4195 (2019)  

\bibitem{CGN} R.\ Catenacci, P.A.\ Grassi, S.\ Noja, \emph{Superstring Field Theory, Superforms and Supergeometry}, J. Geom. Phys., {\bf 148} 103559 (2020)

\bibitem{CCGN} R.\ Catenacci, C.\ Cremonini, P.A.\ Grassi, S.\ Noja, \emph{On Forms, Cohomology, and BV Laplacians in Odd Symplectic Geometry}, Lett.\ Math.\ Phys.\ {\bf 111} (2), (2021) 1-32

\bibitem{CCGN} R.\ Catenacci, C.\ Cremonini, P.A.\ Grassi, S.\ Noja, \emph{Cohomology of Lie Superalgebras: Forms, Integral Forms, and Coset Superspaces}, arXiv:2012.05246

\bibitem{CA1} C.A.\ Cremonini, P.A.\ Grassi, \emph{Pictures from Super Chern-Simons Theory}, JHEP 03 (2020) 043

\bibitem{CGP} C.A.\ Cremonini, P.A.\ Grassi, S.\ Penati, \emph{Supersymmetric Wilson Loops via Integral Forms}, JHEP (2020) 161 

\bibitem{CG2} C.A.\ Cremonini, P.A.\ Grassi, \emph{Cohomology of Lie Superalgebras: Forms, Pseudoforms, and Integral Forms}, arXiv:2106.11786 

\bibitem{Deligne} P. Deligne, J.\ Morgan, \emph{Notes on Supersymmetry (following Joseph Bernstein)}, in \emph{Quantum Field Theory and Strings: a Course for Mathematicians}, Vol 1, AMS (1999) 

\bibitem{DeligneFreed} P. Deligne, D.\ S.\ Freed, \emph{Supersolutions}, in \emph{Quantum Field Theory and Strings: a Course for Mathematicians}, Vol 1, AMS (1999) 

\bibitem{DonWit} R.\ Donagi, E.\ Witten, \emph{Supermoduli Space is Not Projected}, Symp.\ Pure Math.\ {\bf 90} (2015) 19-72 

\bibitem{Eisenbud} D.\ Eisenbud, \emph{Commutative Algerbra - with a view toward Algebraic Geometry}, Springer GTM (1995)

\bibitem{Freed} D.\ S.\ Freed, \emph{Five Lectures on Supersymmetry}, AMS (1999)

\bibitem{FMS1} D.\ Friedan, E.\ Martinec, S.\ Shenker, \emph{Covariant Quantization of Superstrings}, Phys.\ Lett.\ {\bf B} 60 (1985) 55 

\bibitem{FMS2} D.\ Fridan, E.\ Martinec, S.\ Shenker, \emph{Conformal Invariance, Supersymmetry and String Theory}, Nucl. Phys. {\bf B} (1986) 93

\bibitem{Gaiduk} A.V.\ Gaiduk, H.M.\ Khudaverdian, A.S.\ Schwarz, \emph{Integration over Surfaces in Superspace}, Theor.\ Math.\ Phys.\ {\bf 52} (1982)

\bibitem{Gaw} K.\ Gaw\c{e}dzki, \emph{Supersymmetry - Mathematics of Supergeometry}, Ann.\ Inst.\ H.\ Poincaré Sect.\ A (N.S.) {\bf 27} (1977) 335-366

\bibitem{Green} P.\ Green, \emph{On Holomorphic Graded Manifolds}, Proc.\ Am.\ Math. Soc. {\bf 85}, 4 (1982)

\bibitem{GH} P.\ Griffiths, J.\ Harris, \emph{Principles of Algebraic Geometry}, Wiley (1978)


\bibitem{Kapranov} M.\ Kapranov, \emph{Supergeometry in Mathematics and Physics}, in {New Spaces in Physics}, M.\ Anel, G.\ Catren eds, CUP (2021)

229-231 (1989)

\bibitem{Khudaverdian1}  H.M. Khudaverdian, \emph{Geometry of superspace with Even and Odd Brackets}, J. Math. Phys 32,
1938-1941 (1991)

\bibitem{KN1}  H.M. Khudaverdian, A.P. Nersessian, \emph{On Geometry of Batalin-Vilkovisky Formalism}, Mod. Phys.
Lett. A 8 (25), 2377-2385 (1993)

\bibitem{KN2}  H.M. Khudaverdian, A.P. Nersessian, \emph{Batalin-Vilkovisky Formalism and Integration Theory on
Manifolds}, J. Math.Phys 37, 3713-3724 (1996)

\bibitem{Khudaverdian2}  H.M. Khudaverdian, \emph{Odd Invariant Semidenstiy and Divergence-like Operators on Odd Symplectic
Superspace}, Comm.\ Math.\ Phys.\ {\bf 198}, 591-606 (1998)

\bibitem{Khuda2} H. M.\ Khudaverdian, \emph{Laplacians in Odd Symplectic Geometry}, Contemp. Math. {\bf 315}, 199-212 (2002)

\bibitem{Khuda1} H. M.\ Khudaverdian, \emph{Semidensities on Odd Symplectic Supermanifolds}, Commun. Math. Phys., {\bf 247}, 353-390 (2004)


\bibitem{Konstant} B.\ Konstant, \emph{Graded Manifolds, Graded Lie Theory, and Prequantization}, in Differential geometrical methods in mathematical physics, K.\ Bleuler and A.\ Reetz (eds.) 177-306, Lecture Notes in Mathematics {\bf 570}, Springer-Verlag (1977)

\bibitem{Leites} D.\ A.\ Leites, \emph{Introduction to the Theory of Supermanifolds} (Russian), Uspekhi Mat. Nauk. {\bf 35} (1980) 3-57, English translation in Russian Math. Surveys {\bf 35} (1980) 1-64

\bibitem{Sabbah} P. Maisonobe, C. Sabbah, \emph{Aspect of the Theory of $\mathscr{D}$-modules}, Lecture Notes - Kaiserslautern 2002, available at \url{http://www.math.polytechnique.fr/cmat/sabbah/livres/kaiserslautern.pdf}

\bibitem{Manin} Yu.\ I.\ Manin, \emph{Gauge Fields and Complex Geometry}, {Springer-Verlag}, (1988)

\bibitem{ManinSuper} Yu.\ I. Manin, I.\ B.\ Penkov, A.\ A.\ Voronov, \emph{Elements of Supergeometry}, J. Soviet Math. {\bf 51} (1990) 2069-2083

\bibitem{Mnev} 
  P.\ Mnev,
  \emph{Quantum Field Theory: Batalin-Vilkovisky Formalism and its Applications}, AMS (2019)

\bibitem{NojaG} S.\ Noja, \emph{Non-Projected Supermanifolds and Embedding in Super Grassmannians}, Universe {\bf 4} 11 (2018)

\bibitem{NojaRe} S.\ Noja, R.\ Re, \emph{A Note on Super Koszul Complex and the Berezinian}, Ann.\ Mat.\ Pura Appl. {\bf 201} (1) 403-421 (2022) 

\bibitem{Noja} S.\ Noja, \emph{On BV Supermanifolds and the Super Atiyah Class}, Eur.\ J.\ Math.\ {\bf 9} (1) (2023)

\bibitem{OP} O.V.\ Ogievetskii, I.B.\ Penkov, \emph{Serre Duality for Projective Supermanifolds}, Funct. Anal. its Appl. {\bf 18} 68-70 (1984)

\bibitem{Ogiev} O.V.\ Ogievetskii, \emph{private communication} 

\bibitem{Penkov} I.\ B.\ Penkov, \emph{$\mathscr{D}$-Modules on Supermanifolds}, Invent.\ Math.\ {\bf 71}, 501-512, (1983)

\bibitem{Pestun} V.\ Pestun, M.\ Zabzine, eds., \emph{Localization techniques in quantum field theories}, J. Phys. A: Math. Theor., {\bf 50} (44) 440301 (2016)

\bibitem{Ruiperez} D.\ Hern\`andez Ruiperez, J. Mu\~{n}oz Masque, \emph{Construction Intrinsique du faisceau de Berezin d'une variet\`e gradu\`ee}, C. R. Acad. Sc. Paris {\bf 301} 915-918 (1985)

\bibitem{Severa} P.\ \v{S}evera, \emph{On the Origin of the BV Operator on Odd Symplectic Supermanifolds}, Lett.\ Math.\ Phys.\ {\bf 78} 55-59 (2006)

  \bibitem{Schwarz} 
A.\ S.\ Schwarz, \emph{Geometry of Batalin-Vilkovisky Quantization}, Comm.\ Math.\ Phys.\ {\bf 155} (1993)
249-260; 

\bibitem{BerezinLife} M.\ Shifman (edited by), \emph{Felix Berezin: Life and Death of the Mastermind of Supermathematics}, World Scientific (2007)

\bibitem{SuZhang} Y.\ Su, R.\ Zhang, \emph{Mixed Cohomology of Lie Superalgebras}, J.\ Alg.\ {\bf 549} (2020) 1-29

\bibitem{Varadarajan} V.S.\ Varadarajan, \emph{Supersymmetry for Mathematician: an Introduction}, Courant Lecture Notes, AMS (2004)

 \bibitem{VZ1}  T. Voronov, A. Zorich, \emph{Complexes of forms on a supermanifold}, Funct. Anal. Appl. {\bf 20} (2), 58-59
(1986)

\bibitem{VZ2}  T. Voronov, A. Zorich, \emph{Integral Transformations of Pseudodifferential Forms}, Russian Math. Surv.
{\bf 41} (6), 221-222 (1986)

\bibitem{VZ3}  T. Voronov, A. Zorich, \emph{Integration on vector bundles}, Funct. Anal. Appl. {\bf 22} (2), 94 103 (1988)


\bibitem{V3} T. Voronov, \emph{Quantization of Forms on the Cotangent Bundle}. Comm.\ Math.\ Phys.\ {\bf 205}, 315-336 (1999)

\bibitem{V4}  T. Voronov, \emph{Dual forms on supermanifolds and Cartan calculus}, Comm.\ Math.\ Phys.\ {\bf 228}, 1-16
(2002)

\bibitem{Voronov} Th.\ Th.\ Voronov, \emph{Geometric Integration Theory on Supermanifolds}, Cambridge Scientific Publisher (2014)


\bibitem{Witten} E.\ Witten, \emph{Notes on Supermanifolds and Integration}, Pure\ Appl.\ Math.\ Q., {\bf 15} (1) 3--56 (2019)

\bibitem{WittenRiem} E.\ Witten, \emph{Notes on Super Riemann Surfaces and Their Moduli}, Pure\ Appl.\ Math.\ Q., {\bf 15} (1) 57--211 (2019)

\bibitem{WittenSuper} E.\ Witten, \emph{Superstring Perturbation Theory via Super Riemann Surfaces: an Overview}, Pure\ Appl.\ Math.\ Q., {\bf 15} (1) 517--607 (2019)

\end{thebibliography}
\end{document}